\documentclass[preprint,12pt]{elsarticle}

\usepackage{amssymb}
\usepackage{amsmath}

\usepackage{blindtext}
\usepackage{geometry}
\usepackage[utf8]{inputenc}
\usepackage{amssymb,float,latexsym,amsmath,amsthm,booktabs,graphicx,epstopdf,epsfig,caption,subcaption,color,mathtools,doi,cases,verbatim}
\usepackage[overload]{empheq}
\usepackage{enumitem}
\usepackage[english]{babel}
\usepackage[labelfont={bf}]{caption}
\usepackage[mathscr]{eucal}
\usepackage{cases,xcolor}
\usepackage[T1]{fontenc}
\usepackage[bb=dsserif]{mathalpha}
\usepackage{bm}

\usepackage{float}
\newcommand{\bsx}{\boldsymbol{x}}
\newcommand{\bsy}{\boldsymbol{y}}
\newcommand{\bS}{\mathbb{S}}
\newcommand{\R}{\mathbb{R}}
\newcommand{\N}{\mathbb{N}}
\newcommand{\bE}{\mathbb{E}}
\newcommand{\bC}{\mathbb{C}}
\newcommand{\calE}{\boldsymbol{\mathcal{E}}}
\newcommand{\calK}{\boldsymbol{\mathcal{K}}}
\newcommand{\calL}{\mathcal L}
\newcommand{\calA}{\mathcal A}
\newcommand{\calC}{\mathcal C}
\newcommand{\calI}{\mathcal I}
\newcommand{\calJ}{\mathcal J}

\renewcommand{\Re}{\operatorname{Re}}

\newcommand{\mi}{\mathrm{i}}

\makeatletter
\def\Arg{\mathop{\operator@font Arg}\nolimits}
\makeatother

\newtheorem{theo}{Theorem}[section]
\newtheorem{lem}{Lemma}[section]
\newtheorem{prop}{Proposition}[section]
\newtheorem{rem}{Remark}[section]

\newtheorem{ass}{Assumption}[section]
\newtheorem{defin}{Definition}[section]
\numberwithin{equation}{section}
\counterwithin{figure}{section}

\journal{Chaos, Solitons \& Fractals}

\begin{document}

\begin{frontmatter}
	
	
	
	\title{Evolution of time-fractional stochastic hyperbolic diffusion equations on the unit sphere}
	
	
	\author[label1]{Tareq Alodat} 
	
	\affiliation[label1]{organization={Department of Mathematical and Physical Sciences},
		addressline={La Trobe University}, 
		city={Melbourne},
		country={Australia}}
	\author[label2]{Quoc T. Le Gia} 
	\affiliation[label2]{organization={School of Mathematics and Statistics},
		addressline={UNSW}, 
		city={Sydney},
		country={Australia}}

	\begin{abstract}
This paper examines the temporal evolution of a two-stage stochastic model for spherical random fields. The model uses a time-fractional stochastic hyperbolic diffusion equation, which describes the evolution of spherical random fields on $\bS^2$ in time. The diffusion operator incorporates a time-fractional derivative in the Caputo sense. In the first stage of the model, a homogeneous problem is considered, with an isotropic Gaussian random field on $\bS^2$ serving as the initial condition. In the second stage, the model transitions to an inhomogeneous problem driven by a time-delayed Brownian motion on $\bS^2$. 
The solution to the model is expressed through a series of real spherical harmonics. To obtain an approximation, the expansion of the solution is truncated at a certain degree $L\geq1$. The analysis of truncation errors reveals their convergence behavior, showing that convergence rates are affected by the decay of the angular power spectra of the driving noise and the initial condition. In addition, we investigate the sample properties of the stochastic solution, demonstrating that, under some conditions, there exists a local H\"{o}lder continuous modification of the solution. To illustrate the theoretical findings, numerical examples and simulations inspired by the cosmic microwave background (CMB) are presented.
\end{abstract}




\end{frontmatter}

\section{Introduction}
In recent years, there has been a surge of interest in investigating the evolution of stochastic systems that involve random fields on the sphere (or spherical random fields), particularly those defined by stochastic partial differential equations (SPDEs) \cite{Vho2021,CohLan21,KazLeG19,LanSch15}. 
 Spherical random fields find widespread applications in a variety of fields, including earth sciences (e.g., \cite{Christakos1992,Christakos2017,Fisher1993,Stein2007}) and cosmology (e.g., \cite{Schafer2015}).

One stochastic system of particular interest, and the main focus of this work, is the time-fractional stochastic hyperbolic diffusion equation for spherical random fields (see \cite{Olenko2019,Leonenko}). In \cite{Olenko2019}, the authors considered the classical stochastic hyperbolic diffusion equation, which they defined as a homogeneous problem governed by ordinary partial derivatives in time and a fractional diffusion operator in space. This model included a random initial condition, represented by a strongly isotropic Gaussian random field on the unit sphere $\bS^2$. They derived an explicit stochastic solution and its approximation, both expressed as series expansions of elementary functions. Additionally, they established upper bounds for the convergence rates of the approximation errors and investigated various properties of the stochastic solution and its approximation.

In \cite{Leonenko}, the authors considered a homogeneous problem defined by a hyperbolic diffusion equation governed by time-fractional derivatives and a fractional diffusion operator in space with a random initial condition. They then derived the stochastic solution to the equation, expressed as a series expansion of isotropic (in space) spherical random fields.

Although the models considered in \cite{Olenko2019,Leonenko} depict initial isotropic random fields evolving through hyperbolic diffusion processes, their structures incorporate only a single source of randomness, originating from the initial condition, with no external noise present at any time. Additionally, \cite{Leonenko} exclusively presents the solution to the equation in the form of an infinite series. However, the necessity for an approximate solution becomes apparent when contemplating numerical implementations.
These research gaps motivate us to extend the previous works \cite{Leonenko,Olenko2019} with a more practical setting.  

Let ${^C}D_{t}^{\beta}$ be the time-fractional derivatives of order $\beta>0$ in the Caputo (or Caputo-Djrbashian) sense. It is defined for a suitably regular function $g(t)$ (see \cite{Podlubny}), for $ t>0$, as 
\begin{align}\label{Dcaputo}
  {^C}D_{t}^{\beta} g(t):&=
  \begin{cases}
		\dfrac{1}{\Gamma(m-\beta)}\displaystyle\int_0^t \dfrac{1}{(t-s)^{1+\beta-m}} \dfrac{\partial^{m}}{\partial s^{m}}g(s) ds,& m-1<\beta<m,\\
		\dfrac{\partial^m}{\partial t^m}g(t),&\beta=m,
	\end{cases}
\end{align}
where $m\in\N$ and $\beta>0$.

The fractional integral of order $\beta$ for $g$, denoted as ${_a}\calJ_{t}^{\beta}g(t):= {_a}{^C}D_t^{-\beta}$, $a,\beta>0$, is defined in \cite[Chapter 2]{Podlubny}, as
\begin{align}\label{fracint}
   {_a}\calJ_{t}^{\beta}g(t):=  \frac{1}{\Gamma(\beta)}\int_{a}^{t} (t-s)^{\beta-1}g(s)ds.
\end{align}

To establish the framework of the model presented in this paper, we introduce two key operators, $\calE$ and $\calK$, which are defined as follows:
\[
\calE:= \frac{1}{c^2}{^C}D_t^{2\alpha} +\frac{1}{\gamma} {^C}D_t^{\alpha} - k^2\Delta_{\bS^2}
\]
and
\[
\calK:= I+\frac{c^2}{\gamma}  {_0}\calJ_{t}^{\alpha}- (ck)^2 {_0}\calJ_{t}^{2\alpha}\Delta_{\bS^2},
\]
where $\Delta_{\bS^2}$ is the Laplace-Beltrami operator.

We investigate a time-fractional hyperbolic diffusion equation that incorporates a time-delayed Brownian motion $W_{\tau}$ on $\bS^2$ as the external noise, i.e.,
\begin{align}\label{Sys}
   \calE U(\bsx,t)dt = 
   \begin{cases}
		0,&(\bsx,t)\in \bS^2\times(0,\tau],\\
		dW_{\tau}(\bsx,t),&(\bsx,t)\in \bS^2\times[\tau,\infty),
	\end{cases} 
\end{align}
with the initial conditions
\begin{align}\label{cond1}
       U(\bsx,0)= \xi(\bsx),\ \frac{\partial}{\partial t}U(\bsx,t)|_{t=0}=0,
\end{align}
\noindent
where $\xi$ is a strongly isotropic Gaussian random field on $\bS^2$, $\alpha$ is a constant satisfying $\frac{1}{2} < \alpha \leq1$, and the parameters $c,\gamma,k>0$ are constants. 


It is known (see \cite{Podlubny}) that fractional differential equations of order $\alpha$ require $\lceil\alpha \rceil$ initial conditions, where $\lceil\alpha \rceil$ is the lowest integer greater than $\alpha$. Therefore, solving the system \eqref{Sys} with $\frac{1}{2}<\alpha\leq1$ requires two initial conditions (since $\lceil 2\alpha \rceil=2$), thus justifying the conditions specified by \eqref{cond1}.

It is worth noting that \cite{Leonenko} considered the case where $\alpha\in(0,1]$, while in present work we focus exclusively on the range $\alpha\in(\frac{1}{2},1]$. Although it is possible to obtain a solution to equation \eqref{Sys} for the entire range $\alpha\in(0,1]$, we have found that addressing the case where $\alpha\in(0,\frac{1}{2}]$ requires additional work. Specifically, further research is necessary to establish the convergence of the solution in the space $L_2(\bS^2\times\Omega)$ where $\Omega$ is a sample space.

The model given by equation \eqref{Sys} describes the evolution of an isotropic spherical random field subject to time-fractional hyperbolic diffusion. Notably, there are no external inputs acting on the system until a specific time instant $\tau$, at which point the external noise $W_\tau$ is introduced. This noise serves as an external driving force, influencing the subsequent evolution of the system.

Given that the random field $\xi$ and the driving noise $W_\tau$ are uncorrelated, and taking into account the linearity of equation \eqref{Sys}, we can write $U$ as
\begin{align}\label{Full}
	U:= U^{H}+U^{I}, 
\end{align}
where $U^{H}$ solves the homogeneous equation
\begin{align}\label{Hom}
\calE U^H(\bsx,t)= 0,\  (\bsx,t)\in \bS^2\times(0,\infty),
\end{align}
with the conditions 
\begin{align}\label{Homcond1}
       U^H(\bsx,0)= \xi(\bsx),\ \frac{\partial}{\partial t}U^H(\bsx,t)|_{t=0}=0,
\end{align}
and $U^{I}$ solves the inhomogeneous equation
\begin{align}\label{Inhom}
  \calE U^I(\bsx,t)dt =
 \begin{cases}
		0,&(\bsx,t)\in \bS^2\times(0,\tau],\\
		dW_{\tau}(\bsx,t),&(\bsx,t)\in \bS^2\times [\tau,\infty),
	\end{cases}  
\end{align}
with the conditions $U^I(\bsx,t)=0$, for $(\bsx,t)\in \bS^2\times(0,\tau]$ and 
\begin{align}\label{Inhomcond1}
       \frac{\partial}{\partial t}U^I(\bsx,t)|_{t=0}=0,\quad (\bsx,t)\in \bS^2\times[\tau,\infty).
\end{align}
Using properties of fractional integrals \eqref{fracint}, the solution to \eqref{Inhom} can be  then interpreted as the solution to the corresponding fractional stochastic integral equation, for $t\geq\tau$,
\begin{align}\label{IntegformInhom}
  \calK U^I(\bsx,t)-\dfrac{c^2}{\Gamma(2\alpha)}\int_{\tau}^{t} (t-s)^{2\alpha-1}dW_{\bsx,\tau}(s)=0. 
\end{align}
The structure of equation \eqref{Sys} incorporates two different sources of randomness: the initial condition $\xi$ and the driving noise $W_\tau$. Due to the linearity of the system defined by \eqref{Sys}, the stochastic solution $U$ can be split into two independent stochastic components, as shown in equation \eqref{Full}. This decomposition effectively isolates the two sources of randomness. Specifically, the homogeneous solution $U^H$ is influenced solely by the initial condition, while the inhomogeneous solution $U^I$ is driven exclusively by noise $W_\tau$.

We used the Laplace transform to derive the solution of the model. The derived solution is shown to satisfy the time-fractional equation by establishing a novel result (see Theorem \ref{fracstochin}) that provides a rigorous foundation for applying fractional integrals to stochastic integrals.

The main result of this paper is Theorem~\ref{Theo3}, which provides the solution, in $L_2(\Omega\times\bS^2)$, to the system described by equation~\eqref{Sys}. The solution $U$ is given in the form of an infinite series in terms of real-valued orthonormal spherical harmonics $\{Y_{\ell,m}: \ell\in\N_{0}, m=-\ell,\dots,\ell\}$. This result is obtained by combining the solutions for the homogeneous and inhomogeneous problems. The homogeneous solution is established in Proposition~\ref{Theo1} in the \(L_2(\Omega \times \bS^2)\) sense, and its isotropy and Gaussianity are demonstrated in Proposition~\ref{Gausshomo}. For the inhomogeneous problem, Proposition~\ref{inhom_esta} derives the solution, and Proposition~\ref{covUI} confirms that it is also isotropic and Gaussian. Finally, Proposition~\ref{Fulliso} shows that the combined solution of the system remains Gaussian.

Given that the solution to equation \eqref{Sys} is expressed as an infinite series, we shall obtain
an approximate solution, denoted $U_L$, for the solution $U$ by truncating its expansion at a certain level $\ell=L\geq1$. This truncated version provides significant computational advantages. 

The paper is organized as follows. Section \ref{Sec2} provides an essential background on spherical functions, Gaussian random fields, \(L_2(\bS^2)\)-valued time-delayed Brownian motions, and technical results on Mittag-Leffler functions and their stochastic integrals. 
Section \ref{specsol} presents the spectral solutions for the model defined by \eqref{Sys}, the homogeneous and inhomogeneous equations. Specifically, Subsection \ref{Sec_sol_Hom} derives the solution for the homogeneous equation and analyzes its statistical properties, while Subsection \ref{Sec_sol_inHom} addresses the solution for the inhomogeneous equation and examines its statistical properties. Then the solution of equation \eqref{Sys} is derived in Subsection \ref{FullSec}, while its sample Hölder continuity is investigated in Subsection \ref{SecHoder}. Section \ref{SecApprox} develops an approximation for the solution of \eqref{Sys} and evaluates the convergence rate of truncation errors. Finally, Section \ref{Num} conducts numerical studies to observe the evolution of the solution \(U\) for equation \eqref{Sys} using simulated data inspired by the CMB map, and assesses the convergence rates of truncation errors associated with the approximations.

\section{Preliminaries} \label{Sec2}
\subsection{Spherical Functions}
In this section, we present fundamental notation and results from the theory of functions on the unit sphere. 

Consider the 3-dimensional Euclidean space, denoted as $\R^3$. Let $\bsx$ and $\bsy$ be elements of $\R^3$ such that $\bsx=(x_1,x_2,x_3)$ and $\bsy=(y_1,y_2,y_3)$. The inner product of $\bsx$ and $\bsy$ is given by $\bsx\cdot\bsy:=\sum_{i=1}^{3}x_{i}y_{i}$, and the Euclidean norm of $\bsx$ is defined as $|\bsx|:=\sqrt{\bsx\cdot\bsx}$. Let $\bS^2$ be the 2-dimensional unit sphere in $\R^3$, i.e., $\bS^2 :=\left\lbrace\bsx \in \R^3: |\bsx |=1 \right\rbrace$. The pair $(\bS^2,\tilde{d})$ forms a compact metric space, where $\tilde{d}$ is the geodesic metric, defined as $\tilde{d}(\bsx,\bsy):=\arccos(\bsx\cdot\bsy)$ for any $\bsx,\bsy\in\bS^2$. 

Let $\mu$ be the normalized Riemann surface measure on $\bS^2$ (i.e. $\int_{\bS^2}d\mu=1$). Let $L_{2}(\bS^2):=L_{2}(\bS^2,\mu)$ be the space of real-valued square $\mu$-integrable functions on $\bS^2$, i.e.,
\[
L_{2}(\bS^2)=\left\{f: \int_{\bS^2}(f(\bsx))^2d\mu(\bsx)<\infty\right\}
\]
with the inner product
\[
\langle f,g\rangle_{L_{2}(\bS^2)}:=\int_{\bS^2}f(\bsx)g(\bsx)d\mu(\bsx),\quad f,g\in L_{2}(\bS^2).
\]
The norm of $f\in L_{2}(\bS^2)$ is defined as
$\|f\|_{L_{2}(\bS^2)}:=\sqrt{\langle f,f\rangle_{L_{2}(\bS^2)}}$.

Let $\N_{0}=\{0,1,2,\dots\}$ and $\{ Y_{\ell,m}: \ell\in\N_{0}, m=-\ell,\dots,\ell \}$ be an orthonormal basis of real-valued spherical harmonics of $L_{2}(\bS^2)$. In terms of the spherical coordinates \begin{equation}\label{sph-coord}
\bsx(\theta,\varphi)=(\sin\theta\cos\varphi,\sin\theta\sin\varphi,\cos\theta)\in\bS^2, \quad\theta\in[0,\pi],\, \varphi\in[0,2\pi), 
\end{equation}
the real-valued spherical harmonic functions can be written as (see \cite{Whittaker}), for $\ell\in\N_{0}$,
\begin{align*}
Y_{\ell m}(\theta, \varphi) := \begin{cases}
    (-1)^m \sqrt{2} \sqrt{\dfrac{(2\ell + 1)(\ell - |m|)!}{(\ell + |m|)!}} \, P_{\ell,|m|}(\cos \theta) \sin(|m|\varphi) & \text{if } m < 0, \\
    \sqrt{2\ell + 1} \, P_{\ell,m}(\cos \theta) & \text{if } m = 0, \\
    (-1)^m \sqrt{2} \sqrt{\dfrac{(2\ell + 1)(\ell - m)!}{(\ell + m)!}} \, P_{\ell,m}(\cos \theta) \cos(m\varphi) & \text{if } m > 0,
\end{cases}
\end{align*}
where 
$P_{\ell,m}$, $\ell\in\N_{0}$, $m=0,\dots,\ell$, are the associated Legendre functions, 
and $P_\ell$, $\ell\in\N_{0}\}$, are the Legendre polynomials of degree $\ell$.

In what follows, we adopt the notation $Y_{\ell,m}$ to represent spherical harmonics as a function of both spherical and Euclidean coordinates. Specifically, for all $\bsx \in \mathbb{S}^2$, we shall write (with a slight abuse of notation)
		$Y_{\ell,m}(\bsx) := Y_{\ell,m} (\theta, \varphi)$, where $\bsx = (\sin \theta \cos \varphi, \sin \theta \sin \varphi, \cos \theta)$ and $(\theta, \varphi) \in [0, \pi] \times [0, 2\pi).$

The basis $Y_{\ell,m}$ and the Legendre polynomial $P_{\ell}$ satisfy the addition theorem (see \cite{Muller}), i.e.,
\begin{align}\label{addition}
	\sum_{m=-\ell}^{\ell} Y_{\ell,m}(\bsx)Y_{\ell,m}(\bsy)=(2\ell+1)P_{\ell}(\bsx\cdot\bsy),\quad \bsx,\bsy\in\bS^2.
\end{align}
Since the spherical harmonic functions $\{Y_{\ell,m}: \ell\in\N_{0}, m=-\ell,\dots,\ell \}$ are the eigenfunctions of the negative Laplace-Beltrami operator $-\Delta_{\bS^2}$ on the sphere $\bS^2$ with eigenvalues (see \cite{Muller,Dai})
\begin{align}\label{lam}
	\lambda_{\ell}:=\ell(\ell+1),\quad \ell\in\N_{0}, 
\end{align}
we have $-\Delta_{\bS^2}Y_{\ell,m}=\lambda_{\ell}Y_{\ell,m},\quad \ell\in\N_{0}, m=-\ell,\dots,\ell.$

Let $f\in L_{2}(\bS^2)$, then $f$ can be expanded in terms of a Fourier-Laplace series, in the $ L_{2}(\bS^2)$ sense,
\begin{align}\label{Expansion}
	f&=\sum_{\ell=0}^{\infty}\sum_{m=-\ell}^{\ell}\widehat{f}_{\ell,m}Y_{\ell,m},\ \text{with}\  \widehat{f}_{\ell,m}:=\int_{\bS^2}f(\bsx)Y_{\ell,m}(\bsx)d\mu(\bsx),
\end{align}
where  $\widehat{f}_{\ell,m}$, for $\ell\in\N_{0},m=-\ell,\dots,\ell$, are the Fourier coefficients for the function $f$ under the Fourier basis $Y_{\ell,m}$. The series \eqref{Expansion} along with Parseval’s theorem gives
$
\|f\|_{L_{2}(\bS^2)}^2= \sum_{\ell=0}^{\infty}\sum_{m=-\ell}^{\ell} (\widehat{f}_{\ell,m} )^2.
$

\subsection{Spherical isotropic random fields}
This subsection introduces a class of random fields, namely isotropic Gaussian fields on $\bS^2$. These fields can be characterized by an expansion using real spherical harmonics.

Throughout, we denote by $(\Omega,\mathcal{F},\mathbb{P})$ a probability space. Let $L_{2}(\Omega,\mathbb{P})$ be the $L_{2}$-space on $\Omega$ with the probability measure $\mathbb{P}$ and equipped with the norm $|\cdot|_{L_{2}(\Omega)}$. For two random variables $X$ and $Y$ on $(\Omega,\mathcal{F},\mathbb{P})$, we denote by $\bE [X]$ the expected value of $X$ and $Cov(X,Y):=\bE[(X-\bE[X])(Y-\bE[Y])]$ the covariance between $X$ and $Y$. When two random variables $X$ and $Y$ have the same distribution, we shall use the notation $X\stackrel{d}{=} Y$. If $X\stackrel{d}{=}Y$, $Cov(X,X)$ is the variance of $X$, denoted by $Var[X]$. Moreover, $\varphi_{X}(\cdot)$ is the characteristic function of $X$ and is defined as $\varphi_{X}(r):=\bE\big[e^{\mi r X}\big]$.

The Borel $\sigma$-algebra on the sphere $\bS^2$ is represented by $\mathfrak{B}(\bS^2)$, while $SO(3)$ denotes the rotation group on $\R^3$. 
A real-valued random field on $\bS^2$ is a jointly $\mathcal{F}\otimes\mathfrak{B}(\bS^2)$-measurable function $\xi: \Omega\times\bS^2\to\R$. The real-valued $L_2$-space of random fields on $\Omega\times\bS^2$, with the product measure $\mathbb{P}\otimes\mu$, is denoted by $L_{2}(\Omega\times\bS^2,\mathbb{P}\otimes\mu)=: L_{2}(\Omega\times\bS^2)$. In this paper, we assume that the random field $\xi$ belongs to $L_{2}(\Omega\times\bS^2)$. Using Fubini's theorem, the inner product of $\xi_1,\xi_2\in L_{2}(\Omega\times\bS^2)$ can be expressed as:
\begin{align*}
	\langle \xi_1,\xi_2\rangle_{L_{2}(\Omega\times\bS^2)}:&=\int_{\Omega\times\bS^2}\xi_1(\omega,\bsx) \xi_2(\omega,\bsx)d(\mathbb{P}\otimes\mu)(\omega,\bsx)
	= \bE\left[	\langle \xi_1,\xi_2\rangle_{L_{2}(\bS^2)}\right].
\end{align*}
The norm of a real-valued random field $\xi_1\in L_{2}(\Omega\times\bS^2)$ is given by
\[
\|\xi_1\|_{L_{2}(\Omega\times\bS^2)}:=\sqrt{\langle \xi_1,\xi_1\rangle_{L_{2}(\Omega\times\bS^2)}}.
\]
In particular, $\xi(\omega,\cdot)\in L_{2}(\bS^2)$ $\mathbb{P}$-a.s., since 
$\|\xi\|_{L_{2}(\Omega\times\bS^2)}^2=\bE\big[	\|\xi\|_{L_{2}(\bS^2)}^2\big]$.  
For brevity, we write $\xi(\omega,\bsx)$ as $\xi(\bsx)$ or $\xi(\omega)$ if no confusion arises.

\begin{defin} [{\rm\cite{MarPec11}}]
    A random field $\xi$ on $\bS^2$ is said to be strongly isotropic if, for any $k\in\N$ and for all sets of $k$ points $\bsx_1,\dots,\bsx_k\in\bS^2$, and for any rotation $\rho\in SO(3)$, the joint distributions of $\xi(\bsx_1),\dots,\xi(\bsx_k)$ and $\xi(\rho\bsx_1),\dots,\xi(\rho\bsx_k)$ coincide.
\end{defin}

The random field $\xi$ is called a $2$-weakly, isotropic random field if for all $\bsx\in\bS^2$, the second moment of $\xi(\bsx)$ is finite, that is $\bE[|\xi(\bsx)|^2]<\infty$, and for all $\bsx_1,\bsx_2\in\bS^2$, there holds
\[
\bE [\xi(\bsx_1)]=\bE [\xi(\rho\bsx_1)],\ \bE [\xi(\bsx_1)\xi(\bsx_2)]=\bE [\xi(\rho\bsx_1)\xi(\rho\bsx_2)],
\]
for any rotation $\rho\in SO(3)$.

\begin{defin}
A Gaussian random field on $\mathbb{S}^2$ is a random field $\xi: \Omega \times \mathbb{S}^2 \rightarrow \mathbb{R}$ such that for every $k\in\mathbb{N}$ and every set of $k$ points $\mathbf{x}_1,\dots,\mathbf{x}_k\in \mathbb{S}^2$, the vector $(\xi(\mathbf{x}_1),\dots,\xi(\mathbf{x}_k))$ has a multivariate Gaussian distribution.
\end{defin}
We remark that a Gaussian 2-weakly isotropic field is also strongly isotropic (see \cite[Proposition 5.10]{MarPec11}). In this paper, we shall focus on Gaussian random fields that are strongly isotropic.

An arbitrary random field $\xi(\omega,\cdot)\in L_{2}(\bS^2)$ $\mathbb{P}$-a.s. admits an expansion in terms of real spherical harmonics, $\mathbb{P}$-a.s.,
\begin{align}\label{xi}
	\xi=\sum_{\ell=0}^\infty \sum_{m=-\ell}^{\ell}  \widehat{\xi}_{\ell,m} Y_{\ell,m},\quad \widehat{\xi}_{\ell,m}= \int_{\bS^2}\xi(\bsx)Y_{\ell,m}(\bsx)d\mu(\bsx).
\end{align}

\begin{rem}\label{randrem}
In {\rm\cite{MarPec11}}, the spectral representation of real-valued random fields relies on complex spherical harmonics. Since we adopt real spherical harmonics in our approach, the conditions on the random variables \(\widehat{\xi}_{\ell,m}\) are adjusted accordingly.
\end{rem}

If $\xi$ is a centered, 2-weakly isotropic Gaussian random field, by adapting the proof of 
\cite[Proposition 6.6]{MarPec11}, the Fourier coefficients $\widehat{\xi}_{\ell,m}$ of $\xi$ in \eqref{xi} are uncorrelated mean-zero real-valued random variables, i.e., for $\ell\in\N_{0}$, $m=-\ell,\dots,\ell$,
\begin{align}\label{cellm}
\bE\left[\widehat{\xi}_{\ell,m}\right]=0,\quad\bE\left[\widehat{\xi}_{\ell,m}\widehat{\xi}_{\ell^{\prime},m^{\prime}}\right]=
\calC_{\ell}\delta_{\ell\ell^\prime}\delta_{mm^\prime},
\end{align}
where $\delta_{\ell\ell^\prime}$ is the Kronecker delta function and the elements of the sequence  $\{\calC_{\ell}: \ell\in\N_{0}\}$ are non-negative numbers, which is called the angular power spectrum of $\xi$. It follows that the field $\xi$ is a centered random variable with the covariance function given by
\begin{align}
\bE\left[\xi(\bsx)\xi(\bsy)\right]&=\sum_{\ell=0}^\infty \sum_{m=-\ell}^{\ell}\calC_{\ell}Y_{\ell,m}(\bsx)Y_{\ell,m}(\bsy) \nonumber\\
&=\sum_{\ell=0}^{\infty} (2\ell+1)\calC_{\ell}P_{\ell}(\bsx\cdot\bsy), \quad \bsx,\bsy\in\bS^2,\label{EqCl}
\end{align}
where $P_{\ell}$ are the Legendre polynomials. The last step follows from the addition theorem for spherical harmonics \eqref{addition}.

\begin{ass}\label{Ass_cell} We assume that the angular power spectrum of a centered, $2$-weakly isotropic random field $\xi$ satisfies 
\begin{align}\label{eq:Cell}
    \sum_{\ell=0}^{\infty} (2\ell+1)\calC_{\ell}<\infty.
\end{align}
\end{ass}
\begin{rem} \label{rem1-1}
Let $\xi$ be a centered, strongly isotropic Gaussian random field defined on $\bS^2$ with angular power spectrum $\{\calC_{\ell}:\ell\in\N_{0}\}$ satisfying the Assumption {\rm\ref{Ass_cell}}. Let $\{Z_{\ell,m},\ \ell\in\N_0,\ m=-\ell,\dots,\ell \}$ be a set of independent, real-valued, standard normally distributed random variables, then the Karhunen--Lo\`eve expansion of $\xi$ is
\begin{align*}
\xi&\stackrel{d}{=}\sum_{\ell=0}^{\infty}\sum_{m=-\ell}^{\ell} \sqrt{\calC_{\ell}}Z_{\ell,m}Y_{\ell,m},
\end{align*}
where the series is convergent in the $L_{2}(\Omega\times\bS^2)$ sense.
\end{rem}

\subsection{Mittag-Leffler functions and their upper bounds}\label{mittagFun}

Let $\Gamma(\cdot)$ be the gamma function and $\bC$ is the set of complex numbers. We introduce the three-parameter Mittag-Leffler function $E_{a,b}^q(z)$ (see \cite{Kilbas2006}) defined as
\begin{align}\label{qMitt}
  E_{a,b}^q(z):=\sum_{r=0}^{\infty} \frac{\Gamma(q+r)}{r!\Gamma(q)}\; \frac{z^r}{\Gamma(ar+b)},\quad z\in\bC.  
\end{align}
When $q=1$, \eqref{qMitt} reduces to the two-parameter Mittag-Leffler function (or the Mittag-Leffler type function) $E_{a,b}(z)$ (see \cite{Kilbas2006}) defined as
\begin{align}\label{Mittag2}
    E_{a,b}(z):=\sum_{r=0}^{\infty} \frac{z^r}{\Gamma(ar+b)},\quad \Re a>0,\; b>0,\; z\in\bC.
\end{align}
If $b=1$, then the function $E_{a,1}(z):=E_{a}(z)$ represents the classical Mittag-Leffler function.

From \cite[equation (1.9.13)]{Kilbas2006} there holds
\begin{align}\label{invlap}
    \calL\big(t^{b-1}E_{a,b}^q(-\mu t^{\alpha})\big)=\frac{z^{aq-b}}{(z^\alpha+\mu)^q},
\end{align}
where $ \calL$ denotes the Laplace transform operator.

Let $\varkappa$ be defined as
\begin{align}\label{conA}
\varkappa:=
\frac{1}{2}(\sqrt{1+4\omega^2}-1),
\end{align}
where 
$
\omega(c,\gamma,k)=\omega := \dfrac{c}{2\gamma k},\quad c,\gamma,k>0.
$

Noting that $\ell > \varkappa$ if and only if $\lambda_{\ell} > \omega^2$, where $\lambda_{\ell}$ is defined by \eqref{lam}. 

We define, for $\ell\in\N_0$,
\begin{align}\label{Ml}
    M_{\ell} :=
\begin{cases}
\sqrt{1-\lambda_{\ell}\omega^{-2}}, & \ell\leq \varkappa,\\
\mi\sqrt{\lambda_{\ell}\omega^{-2}-1}, & \ell>\varkappa,
\end{cases}
\end{align}
and for $\ell\neq\varkappa$,
\begin{align}\label{z1z2}
   z_{\ell}^{\pm}(c,\gamma)= z_{\ell}^{\pm}:=\frac{1}{2}c^2\gamma^{-1}(1\pm M_{\ell})\quad \ell\neq\varkappa.
\end{align}
Since $M_\varkappa=0$, then from \eqref{z1z2} we have
\begin{align}\label{za}
    z_{\varkappa}^{\pm}=: z_\varkappa=\frac{1}{2}c^2\gamma^{-1}.
\end{align}

\noindent

The following proposition provides an upper bound for the function $E_{\alpha,\beta}\big(-z_{\ell}^{\pm}t^{\alpha}\big)$, where $z_{\ell}^{\pm}$ are given by \eqref{z1z2}.
Our result generalizes the one obtained in \cite{Leonenko}, offering detailed proof and modifications to the approach used in \cite{Leonenko}.

\begin{prop}\label{BoundE}
    Let $\alpha\in(\frac{1}{2},1]$ and $\beta$ be an arbitrary real number, then for $t\geq0$ there holds
      \begin{numcases}{  | E_{\alpha,\beta}\big(-z_{\ell}^{\pm}t^{\alpha}\big)|\leq}
     C_1\big(1+kct^\alpha\sqrt{\lambda_\ell}\big)^{\frac{1-\beta}{\alpha}}e^{-K_{\ell}\;t\big(k^2c^2\lambda_\ell\big)^{1/2\alpha}}+\dfrac{C_2}{1+kct^\alpha\sqrt{\lambda_\ell}},\;\ \ell>\varkappa, \label{largeA0} \\
     \dfrac{C_3}{1+\big| z_{\ell}^{\pm}\big| t^{\alpha}},\;\;\;\quad\quad\qquad\qquad\qquad\qquad\qquad\qquad\qquad\qquad \ell\leq \varkappa, \label{smallA0}
\end{numcases}
where $\varkappa$ is given by \eqref{conA}, $z_{\ell}^{\pm}$ are given by \eqref{z1z2}, $K_\ell> 0$, and $C_i,\; i=1,2,3,$ are real constants.
\end{prop}

\begin{proof}
Proposition \ref{BoundE} follows from the following upper bounds (see \cite{Podlubny}, Theorems 1.5 and 1.6), for $a<2$, $b\in\R$, and $\theta$ satisfying $a\pi/2<\theta<\min\{\pi,a\pi\}$, 
\begin{numcases}{ | E_{a,b}(-z)|\leq}
 C_1(1+|z|)^{\frac{1-b}{a}}e^{\Re z^{1/a}}+ \dfrac{C_2}{1+\vert z\vert},\quad |\Arg z|\leq\theta,\; |z|\geq0, \label{largeA} \\
    \dfrac{C_3}{1+\vert z\vert},\qquad\qquad\qquad\qquad\qquad\qquad \theta\leq|\Arg z|\leq\pi,\; |z|\geq0, \label{smallA}
\end{numcases}
where $C_i$, $i=1,2,3$, are real constants.

\indent We seek expressions for $ E_{\alpha,\beta}\big(-z_{\ell}^{\pm}t^{\alpha}\big)$. Using \eqref{z1z2}, then for $\ell>\varkappa$, we write
\begin{align*}
   z_{\ell}^{\pm}&=\frac{1}{2}c^2\gamma^{-1}(1\pm M_{\ell})= \omega ck \Big(1\pm\mi \sqrt{\lambda_\ell \omega^{-2}-1}\Big).
\end{align*}
Thus, for $\ell>\varkappa$ we have
\begin{align}\label{absz}
  \vert -z_{\ell}^{\pm}\vert=\omega ck \sqrt{\lambda_\ell \omega^{-2}}=kc\sqrt{\lambda_\ell}.
\end{align}
Now we need to analyse the term $|\Arg z_{\ell}^{\pm}|$. Using \eqref{z1z2} we can write  
\begin{align}\label{Argz}
    \vert\Arg z_{\ell}^{\pm}\vert&=\Big\vert\mp\Big(\pi-\arctan \sqrt{\lambda_\ell \omega^{-2}-1}\Big)\Big\vert
    =\pi-\arctan \sqrt{\lambda_\ell \omega^{-2}-1}.
\end{align}
Using \eqref{Argz}, then $|\Arg z_{\ell}^{\pm}|\leq\alpha\pi$ if and only if
$\dfrac{\lambda_\ell}{\omega^2}>\sec^{2}((1-\alpha)\pi)\geq1$, since $\frac{1}{2}<\alpha\leq1$.

Also, if $\lambda_\ell\leq \omega^2$, then the terms $z_{\ell}^{\pm}$ are real constants and hence $|\Arg(- z_{\ell}^{\pm}t^{\alpha})|=\pi$.
This means for $\ell>\varkappa$ we use the upper bound given by \eqref{largeA}, while for $\ell\leq\varkappa$ we use \eqref{smallA}. Thus, if $\ell\leq \varkappa$ we have, using \eqref{smallA},
\[
| E_{\alpha,\beta}\big(-z_{\ell}^{\pm}t^{\alpha}\big)|\leq \dfrac{C_3}{1+\vert z_{\ell}^{\pm}\vert t^{\alpha}}.
\]
To complete the proof we need an expression for $\Re (-z_{\ell}^{\pm}t^{\alpha})^{1/\alpha}$, $\ell>\varkappa$. Using \eqref{absz} and \eqref{Argz} we get
\begin{align}\label{Rez}
    \Re (-z_{\ell}^{\pm}t^{\alpha})^{1/\alpha}&=\Re\Big(|-z_{\ell}^{\pm}t^{\alpha}|e^{\mi\Arg (z_{\ell}^{\pm}t^{\alpha})}\Big)^{1/\alpha}\notag\\
    &= \Re\Bigg(\Big(kct^\alpha\sqrt{\lambda_\ell} \Big)^{1/\alpha}e^{\mi\Arg (z_{\ell}^{\pm}/\alpha})\Bigg)\notag\\
    &=t\big(k^2c^2\lambda_\ell\big)^{1/2\alpha} \cos \dfrac{\Arg z_{\ell}^{\pm}}{\alpha} =-K_{\ell}t\big(k^2c^2\lambda_\ell\big)^{1/2\alpha},
\end{align}
where $K_{\ell}=-\cos \dfrac{\Arg z_{\ell}^{\pm}}{\alpha}$.

Using \eqref{Argz}, we have $\frac{\pi}{2}<\Arg z_{\ell}^{\pm}\leq\alpha\pi$ and $\dfrac{\pi}{2\alpha}<\dfrac{\Arg z_{\ell}^{\pm}}{\alpha}\leq\pi$ for $\alpha\in(\frac{1}{2},1]$, ensuring 
$K_{\ell}>0$.
Thus, for $\ell>\varkappa$, by applying \eqref{Rez} and \eqref{absz} with \eqref{largeA}, gives 
\[
 | E_{\alpha,\beta}\big(-z_{\ell}^{\pm}t^{\alpha}\big)|\leq C_1\big(1+kct^\alpha\sqrt{\lambda_\ell}\big)^{\frac{1-\beta}{\alpha}}e^{-K_{\ell}\;t\big(k^2c^2\lambda_\ell\big)^{1/2\alpha}}+\dfrac{C_2}{1+kct^\alpha\sqrt{\lambda_\ell}},
\]
which completes the proof.
\end{proof}

Let for $\alpha\in(0.5,1]$ and $\ell\in\N_0$,
\begin{align}\label{def_psinew}
	    	\psi_{\ell,\alpha}( t)&:= \calL^{-1}\Big\{ \frac{ 1 }{ \frac{1}{c^2}z^{2\alpha} +\frac{1}{\gamma}z^{\alpha}+  k^2\lambda_\ell}\Big\},
	\end{align}
where $\calL^{-1}$ is the inverse Laplace transform operator. 
 
\begin{lem}\label{pf:psi}
For $\alpha\in(0.5,1]$, $\ell\in\N_0$, and $t>0$, the $\psi_{\ell,\alpha}(t)$, defined by \eqref{def_psinew}, takes the form 
\begin{align}
 \psi_{\ell,\alpha}(t):=
\begin{cases}\label{psi}
\dfrac{\gamma}{M_\ell}t^{\alpha-1}\left[E_{\alpha,\alpha}\big(-z_{\ell}^{-}t^{\alpha}\big)-E_{\alpha,\alpha}\big(-z_{\ell}^{+}t^{\alpha}\big)\right], & \ell\neq \varkappa, \\
		\dfrac{c^2}{\alpha}t^{2\alpha-1}\big[E_{\alpha,2\alpha-1}(-z_\varkappa t^\alpha)-(1-\alpha)E_{\alpha,2\alpha}(-z_\varkappa t^\alpha)\big], & \ell=\varkappa.
	\end{cases}
\end{align}
\end{lem}
\begin{proof}
    In order to give an explicit form for $\psi_{\ell,\alpha}$, we first write \eqref{def_psinew} as 
    \[
  \psi_{\ell,\alpha}( t)=c^2\calL^{-1}\Big\{ \frac{ 1 }{ z^{2\alpha} +c^2\gamma^{-1}z^{\alpha}+  c^2k^2\lambda_\ell}\Big\}.
    \]
Using \cite[equation (5.24)]{Podlubny} we write
	\begin{align*}
	    \frac{ 1 }{ z^{2\alpha} +c^2\gamma^{-1}z^{\alpha}+  c^2k^2\lambda_\ell}=\sum_{r=0}^{\infty}(-c^2k^2\lambda_\ell)^{r}\frac{z^{-\alpha(r+1)}}{(z^{\alpha}+c^2\gamma^{-1})^{r+1}}. 
	\end{align*}
Now, for $\ell\neq\varkappa$, applying equation \eqref{invlap} gives
\begin{align}\label{ER}
    \psi_{\ell,\alpha}( t)
    &=c^2t^{2\alpha-1}\sum_{r=0}^{\infty}(-c^2k^2\lambda_\ell t^{2\alpha})^{r}E_{\alpha,2\alpha(r+1)}^{r+1}\big(-c^2\gamma^{-1}t^{\alpha}\big).
\end{align}

By solving  
\begin{align}\label{z1andz2}
    z_1(t)z_2(t)=c^2k^2\lambda_\ell t^{2\alpha},\quad z_1(t)+z_2(t)=-c^2\gamma^{-1}t^{\alpha}
\end{align}
for $z_1(t)$ and $z_2(t)$, we obtain 
$z_1(t)=-z_{\ell}^{-}t^{\alpha}$ and $z_2(t)=-z_{\ell}^{+}t^{\alpha}$, where $z_{\ell}^{\pm}$, $\ell\neq \varkappa$, are defined by \eqref{z1z2}. Thus by \cite[Theorem 3.2]{Soubhia2010}, we can write \eqref{ER} as
\begin{align*}
    \psi_{\ell,\alpha}( t)&=c^2t^{2\alpha-1}\left[\bigg(\frac{1+M_\ell}{2M_\ell}\bigg) E_{\alpha,2\alpha}\big(-z_{\ell}^{+}t^{\alpha}\big)-\bigg(\frac{1-M_\ell}{2M_\ell}\bigg) E_{\alpha,2\alpha}\big(z_{\ell}^{-}t^{\alpha}\big)\right]\notag\\
    &= c^2\frac{t^{2\alpha-1}}{2M_\ell}\left[(1+M_\ell)E_{\alpha,2\alpha}\big(-z_{\ell}^{+}t^{\alpha}\big)-(1-M_\ell)E_{\alpha,2\alpha}\big(-z_{\ell}^{-}t^{\alpha}\big)\right]\notag\\
    &=c^2\frac{t^{2\alpha-1}}{2M_\ell}\left[-\frac{2\gamma}{c^2t^\alpha}E_{\alpha,\alpha}\big(-z_{\ell}^{+}t^{\alpha}\big)+\frac{2\gamma}{c^2t^\alpha}E_{\alpha,\alpha}\big(-z_{\ell}^{-}t^{\alpha}\big)\right]\notag\\
    &= \frac{\gamma}{M_\ell}t^{\alpha-1}\left[E_{\alpha,\alpha}\big(-z_{\ell}^{-}t^{\alpha}\big)-E_{\alpha,\alpha}\big(-z_{\ell}^{+}t^{\alpha}\big)\right],
\end{align*}
where the third step used the relation \cite[equation (5.1.13)]{Gorenflo2014}
\begin{align}\label{Mitpro}
    zE_{a,b}(z)=E_{a,b-a}(z)-\frac{1}{\Gamma(b-a)},\quad a,b>0,\ (a-b)\notin\N_0.
\end{align}
Now let $\ell=\varkappa$, then by using equation \eqref{ER} with \cite[Corollary 3.3]{Soubhia2010} and \cite[relation (5.1.14)]{Gorenflo2014} we can write $\psi_{\varkappa,\alpha}$ as
    \begin{align*}
     \psi_{\varkappa,\alpha}(t)&= c^2t^{2\alpha-1}(E_{\alpha,\alpha}(-z_\varkappa t^{\alpha})-z_\varkappa t^{\alpha}E_{\alpha,3\alpha}^{2}(-z_\varkappa t^{\alpha}))\\
     &= c^2t^{2\alpha-1}E_{\alpha,2\alpha}^{2}(-z_\varkappa t^{\alpha})\\
     &=\frac{1}{\alpha} c^2t^{2\alpha-1}\big(E_{\alpha,2\alpha-1}(-z_\varkappa t^\alpha)-(1-\alpha)E_{\alpha,2\alpha}(-z_\varkappa t^\alpha)\big),
    \end{align*}
    where $E_{a,b}^{q}(\cdot)$ is defined by \eqref{qMitt}). Thus the proof is complete.
\end{proof}

\indent
We define the function $F_{\ell,\alpha}(t)$, $t\geq0$, $\ell\in\N_0$, as
\begin{align}\label{Fl}
   F_{\ell,\alpha}(t):=
   \begin{cases}
       F_{\ell,\alpha}^{(1)}(t)+F_{\ell,\alpha}^{(2)}(t),& \ell\neq \varkappa,\\
       E_{\alpha}\big(-z_\varkappa t^{\alpha}\big)\big(1+\frac{1}{2\alpha}c^2\gamma^{-1}t^{\alpha}\big),& \ell=\varkappa,
   \end{cases}
\end{align}
where
\begin{align}\label{F1}
   F_{\ell,\alpha}^{(1)}(t):=\frac{1}{2}\left[E_{\alpha}\big(-z_{\ell}^{-}t^{\alpha}\big)+E_{\alpha}\big(-z_{\ell}^{+}t^{\alpha}\big)\right],
\end{align}
\begin{align}\label{F2}
  F_{\ell,\alpha}^{(2)}(t):=\frac{1}{2M_{\ell}}\left[E_{\alpha}\big(-z_{\ell}^{-}t^{\alpha}\big)-E_{\alpha}\big(-z_{\ell}^{+}t^{\alpha}\big)\right],
\end{align}
and $z_{\ell}^{\pm}$, $z_\varkappa$, are given in \eqref{z1z2} and  \eqref{za} respectively.

We prove the following lemma.
\begin{lem}
    Let $\alpha\in(\frac{1}{2},1]$. Then for $t\geq0$ and $\ell\in\N_{0}$ there holds
\begin{align}\label{NewUp}
     \dfrac{1}{c^2}{^C}D_t^{2\alpha} F_{\ell,\alpha}(t)+\dfrac{1}{\gamma} {^C}D_t^{\alpha} F_{\ell,\alpha}&=-k^2\lambda_\ell F_{\ell,\alpha}(t),
\end{align}
    where $F_{\ell,\alpha}(\cdot)$ is defined by \eqref{Fl}.
\end{lem}

\begin{proof}
Using \eqref{F1} and \eqref{F2} with the help of \eqref{Dcaputo}, we have, for $\ell\neq \varkappa$,
\begin{align}\label{dervF1}
 {^C}D_t^{\alpha} F_{\ell,\alpha}^{(1)}(t)=-\dfrac{1}{2}\Big(z_{\ell}^{-} E_{\alpha}\big(-z_{\ell}^{-}t^{\alpha}\big)+z_{\ell}^{+} E_{\alpha}\big(-z_{\ell}^{+}t^{\alpha}\big)\Big)   
\end{align}
and
\begin{align}\label{dervF2}
   {^C}D_t^{\alpha}F_{\ell,\alpha}^{(2)}(t)=-\dfrac{1}{2M_\ell}\Big(z_{\ell}^{-} E_{\alpha}\big(-z_{\ell}^{-}t^{\alpha}\big)-z_{\ell}^{+} E_{\alpha}\big(-z_{\ell}^{+}t^{\alpha}\big)\Big). 
\end{align}
By the sequential fractional derivative  \cite[see equation (2.170)]{Podlubny}, we obtain
\begin{align}\label{D2alphaF1}
 {^C}D_t^{2\alpha} F_{\ell,\alpha}^{(1)}(t)=\dfrac{1}{2}\Big((z_{\ell}^{-})^2 E_{\alpha}\big(-z_{\ell}^{-}t^{\alpha}\big)+(z_{\ell}^{+})^2 E_{\alpha}\big(-z_{\ell}^{+}t^{\alpha}\big)\Big)   
\end{align}
and
\begin{align}\label{D2alphaF2}
{^C}D_t^{2\alpha}F_{\ell,\alpha}^{(2)}(t)=\dfrac{1}{2M_\ell}\Big((z_{\ell}^{-})^2 E_{\alpha}\big(-z_{\ell}^{-}t^{\alpha}\big)-(z_{\ell}^{+})^2 E_{\alpha}\big(-z_{\ell}^{+}t^{\alpha}\big)\Big). 
\end{align}

From \eqref{F1} and \eqref{F2} with the help of \eqref{dervF1}-\eqref{D2alphaF2} we obtain, for $\ell\neq\varkappa$,
\begin{align}\label{Part1}
  &\dfrac{1}{c^2}{^C}D_t^{2\alpha} F_{\ell,\alpha}(t)= \dfrac{1}{c^2}\Big({^C}D_t^{2\alpha} F_{\ell,\alpha}^{(1)}(t)+{^C}D_t^{2\alpha} F_{\ell,\alpha}^{(2)}(t)\Big)\notag\\
  &=\dfrac{1}{2} c^{-2}\Big( (z_{\ell}^{-})^2E_{\alpha}\big(-z_{\ell}^{-}t^{\alpha}\big)\Big(1+\dfrac{1}{M_\ell}\Big)+(z_{\ell}^{+})^2E_{\alpha}\big(-z_{\ell}^{+}t^{\alpha}\big)\Big(1-\dfrac{1}{M_\ell}\Big)\Big)
\end{align}
and
\begin{align}\label{Part2}
&\dfrac{1}{\gamma} {^C}D_t^{\alpha} F_{\ell,\alpha}=\dfrac{1}{\gamma}\Big({^C}D_t^{\alpha} F_{\ell,\alpha}^{(1)}(t)+{^C}D_t^{\alpha} F_{\ell,\alpha}^{(2)}(t)\Big)\notag\\
&=-\dfrac{1}{2\gamma}\Big( z_{\ell}^{-}E_{\alpha}\big(-z_{\ell}^{-}t^{\alpha}\big)\Big(1+\dfrac{1}{M_\ell}\Big)+z_{\ell}^{+}E_{\alpha}\big(-z_{\ell}^{+}t^{\alpha}\big)\Big(1-\dfrac{1}{M_\ell}\Big)\Big).
\end{align}
Summing up equations \eqref{Part1} and \eqref{Part2}, we obtain
\begin{align}\label{FulD}
    \dfrac{1}{c^2}{^C}D_t^{2\alpha} F_{\ell,\alpha}(t)+\dfrac{1}{\gamma} {^C}D_t^{\alpha} F_{\ell,\alpha}&=E_{\alpha}\big(-z_{\ell}^{-}t^{\alpha}\big)\Big(\dfrac{1+M_\ell}{2M_\ell}\Big)\Big(c^{-2}(z_{\ell}^{-})^2-\gamma^{-1}z_{\ell}^{-}\Big)\notag\\
    &+E_{\alpha}\big(-z_{\ell}^{+}t^{\alpha}\big)\Big(\dfrac{M_\ell-1}{2M_\ell}\Big)\Big(c^{-2}(z_{\ell}^{+})^2-\gamma^{-1}z_{\ell}^{+}\Big).
\end{align}
By \eqref{z1z2}, we have $1\pm M_\ell=2c^{-2}\gamma z_{\ell}^{\pm}$. Thus \eqref{FulD} becomes
\begin{align}\label{FulD2}
    &\dfrac{1}{c^2}{^C}D_t^{2\alpha} F_{\ell,\alpha}(t)+\dfrac{1}{\gamma} {^C}D_t^{\alpha} F_{\ell,\alpha}=2c^{-2}z_{\ell}^{+}z_{\ell}^{-}\gamma\notag\\
    &\times\Big(\dfrac{E_{\alpha}\big(-z_{\ell}^{-}t^{\alpha}\big)}{2M_\ell}\Big(\dfrac{\gamma z_{\ell}^{-}}{c^2}-1\Big)-\dfrac{E_{\alpha}\big(-z_{\ell}^{+}t^{\alpha}\big)}{2M_\ell}\Big(\dfrac{\gamma z_{\ell}^{+}}{c^2}-1\Big)\Big).
\end{align}
Using \eqref{absz}, we obtain  $z_{\ell}^{-}z_{\ell}^{+}=|z_{\ell}^{\pm}|^2=c^2k^2\lambda_\ell$. Thus, the result follows by simplifying equation \eqref{FulD2}.

The case $\varkappa\in\N$ can be handled similarly. However, one needs to employ the corresponding definition of $F_{\ell,\alpha}$ given by \eqref{Fl}. Thus, the proof is complete.
\end{proof}

Note that by the upper bounds \cite[(24) and (25)]{Leonenko} we can show that
\begin{align}\label{F_upper}
    \big\vert F_{\ell,\alpha}(t) \big\vert \le \dfrac{1}{\sqrt{\lambda_\ell}}\mathcal{H}(t),\quad \ell\geq1,
\end{align}
where for some real constants $C_0^{\prime},C_1^{\prime},C_2^{\prime}>0$, the function $\mathcal{H}(t)$ is given by
\begin{align}\label{updf1} 
\mathcal{H}(t):= C_0^{\prime}+C_1^{\prime}t^{-1}+C_2^{\prime}t^{-\alpha},\quad t>0.
\end{align}

Thus using \eqref{NewUp} with \eqref{F_upper} we obtain
\begin{align}\label{dF1}
     \big\vert \dfrac{1}{c^2}{^C}D_t^{2\alpha}F_{\ell,\alpha}^{(1)}(t)+\dfrac{1}{\gamma}{^C}D_t^{\alpha} F_{\ell,\alpha}(t)\big\vert \leq \sqrt{\lambda_\ell}\mathcal{H}(t), \quad\ell\in\N.
\end{align}

\begin{prop}\label{Propsigma}
For $\ell\in\N_{0}$, $t\geq0$, and $\alpha\in(\frac{1}{2},1]$, let $\sigma_{\ell,t,\alpha}^2$ be defined as
\begin{align}\label{var}
\sigma_{\ell,t,\alpha}^2:=
\begin{cases} 
	\dfrac{\gamma^2}{|M_\ell|^2}	\int_{0}^{t} r^{2\alpha-2}\big\vert E_{\alpha,\alpha}\big(-z_{\ell}^{-}r^{\alpha}\big)-E_{\alpha,\alpha}\big(-z_{\ell}^{+}r^{\alpha}\big)\big\vert^2dr, & \ell\neq \varkappa, \\
	\dfrac{c^4}{\alpha^2}\int_{0}^{t} r^{4\alpha-2}\big\vert E_{\alpha,2\alpha-1}\big(-z_{\varkappa}r^{\alpha}\big)-(1-\alpha)E_{\alpha,2\alpha}\big(-z_{\varkappa}r^{\alpha}\big)\big\vert^2dr, & \ell=\varkappa.
	\end{cases}
\end{align}
Then there holds
\begin{numcases}{ \sigma_{\ell,t,\alpha}^2\leq}
	\dfrac{C_1^I}{\vert M_\ell\vert^2}\big| z_{\ell}^{-}\big|^{\frac{1}{\alpha}-2},\qquad \ell< \varkappa, \label{up1} \\
C_2^It^{4\alpha-1},\;\;\qquad\qquad \ell=\varkappa,\label{Aequal}\\
   	C_3^I \lambda_\ell^{\frac{1-\alpha}{\alpha}},\ \qquad\qquad\; \ell> \varkappa, \label{up2}
\end{numcases}
where $z_{\ell}^{\pm}$ are given by \eqref{z1z2} and the real constants $C_i^I>0$, $i=1,2,3$, do not depend on $\ell$.
\end{prop}
\begin{proof}
Let $\ell< \varkappa$, then by \eqref{var} we have
\begin{align}\label{I2I2}
&\sigma_{\ell,t,\alpha}^2=	\dfrac{\gamma^2}{|M_\ell|^2}\int_{0}^{t} r^{2\alpha-2}\Big\vert E_{\alpha,\alpha}\big(-z_{\ell}^{-}r^{\alpha}\big)-E_{\alpha,\alpha}\big(-z_{\ell}^{+}r^{\alpha}\big)\Big\vert^2dr\notag\\
&\leq 	\dfrac{\gamma^2}{|M_\ell|^2}\Big(\int_{0}^{t} r^{2\alpha-2}\big\vert E_{\alpha,\alpha}\big(-z_{\ell}^{-}r^{\alpha}\big)\big\vert^2dr + \int_{0}^{t} r^{2\alpha-2}\big\vert E_{\alpha,\alpha}\big(-z_{\ell}^{+}r^{\alpha}\big)\big\vert^2dr\Big)\notag\\
&=: \dfrac{\gamma^2}{|M_\ell|^2}(\Theta_1(t)+\Theta_2(t)).
\end{align}
The term $\Theta_1(t)$ in \eqref{I2I2} can be bounded, using \eqref{smallA0}, by
\begin{align}\label{I1}
 \Theta_1(t)&= \int_{0}^{t} r^{2\alpha-2}\big\vert E_{\alpha,\alpha}\big(-z_{\ell}^{-}r^{\alpha}\big)\big\vert^2dr\leq C_3^2 \int_{0}^{t} r^{2\alpha-2}\frac{1}{\big(1+|z_{\ell}^{-}|r^\alpha\big)^2}dr.
\end{align}

Using the change of variables $u=|z_{\ell}^{-}|^{\frac{1}{\alpha}} r$, we obtain from \eqref{I1} that
\[
\Theta_1(t)\leq C_3^2 |z_{\ell}^{-}|^{\frac{1-2\alpha}{\alpha}}\int_{0}^{|z_{\ell}^{-}|^{\frac{1}{\alpha}}t} u^{2\alpha-2}\frac{1}{\big(1+u^\alpha\big)^2}du.
\]
If $|z_{\ell}^{-}|^{\frac{1}{\alpha}}t\geq1$, we could write
\begin{align}\label{I11}
  \Theta_1(t)&\leq C_3^2 |z_{\ell}^{-}|^{\frac{1-2\alpha}{\alpha}}\Bigg(\int_{0}^{1}u^{2\alpha-2}du+\int_{1}^{|z_{\ell}^{-}|^{\frac{1}{\alpha}}t} u^{2\alpha-2}\frac{1}{\big(1+u^\alpha\big)^2}du\Bigg)\notag\\ 
  & \leq C_3^2 |z_{\ell}^{-}|^{\frac{1-2\alpha}{\alpha}}\Bigg(\frac{1}{2\alpha-1}+\int_{1}^{|z_{\ell}^{-}|^{\frac{1}{\alpha}}t} u^{-2}du\Bigg)\notag\\
  & \leq C_3^2 |z_{\ell}^{-}|^{\frac{1-2\alpha}{\alpha}}\Bigg(\frac{1}{2\alpha-1}+1-|z_{\ell}^{-}|^{-\frac{1}{\alpha}}t^{-1}\Bigg)\notag\\
  &\leq C_3^2 |z_{\ell}^{-}|^{\frac{1-2\alpha}{\alpha}}\Bigg(\frac{2\alpha}{2\alpha-1}\Bigg).
\end{align}
Similarly, we have
\begin{align}\label{I2}
   \Theta_2(t)&\leq C_3^2 |z_{\ell}^{+}|^{\frac{1-2\alpha}{\alpha}}\Bigg(\frac{2\alpha}{2\alpha-1}\Bigg).
\end{align}

Since for $\ell<\varkappa$, $\vert z_{\ell}^{-}\vert\leq \vert z_{\ell}^{+}\vert$, then by substituting \eqref{I11} and \eqref{I2} in  \eqref{I2I2} we obtain \eqref{up1}, where 
$
C_1^I:=4C_3^2\gamma^2/(2\alpha-1).
$

Now if $\ell>\varkappa$, then by \eqref{largeA0} we have
\begin{align*}
  \sigma_{\ell,t,\alpha}^2&=	\dfrac{\gamma^2}{|M_\ell|^2}\int_{0}^{t} r^{2\alpha-2}\Big\vert E_{\alpha,\alpha}\big(-z_{\ell}^{-}r^{\alpha}\big)-E_{\alpha,\alpha}\big(-z_{\ell}^{+}r^{\alpha}\big)\Big\vert^2dr\\
  &\leq 	\dfrac{\gamma^2}{|M_\ell|^2}\Big(\int_{0}^{t} r^{2\alpha-2}\big\vert E_{\alpha,\alpha}\big(-z_{\ell}^{-}r^{\alpha}\big)\big\vert^2dr+\int_{0}^{t} r^{2\alpha-2}\big\vert E_{\alpha,\alpha}\big(-z_{\ell}^{+}r^{\alpha}\big)\big\vert^2dr\Big)\\
  &\leq 2	\dfrac{\gamma^2}{|M_\ell|^2}\int_{0}^{t} r^{2\alpha-2}\notag\\
  &\times\Bigg( C_1\big(1+kcr^\alpha\sqrt{\lambda_\ell}\big)^{\frac{1-\alpha}{\alpha}}e^{-K_{\ell}\;r\big(k^2c^2\lambda_\ell\big)^{1/2\alpha}}+\dfrac{C_2}{1+kcr^\alpha\sqrt{\lambda_\ell}}\Bigg)^2dr.
\end{align*}
If $t(kc\sqrt{\lambda_\ell})^{1/\alpha}\geq1$, then the change of variables $u=(kc\sqrt{\lambda_\ell})^{1/\alpha}r$ gives

\begin{align}\label{S1}
  \sigma_{\ell,t,\alpha}^2&\leq 	\dfrac{2\gamma^2}{|M_\ell|^2}\big(kc\sqrt{\lambda_\ell}\big)^{\frac{1-2\alpha}{\alpha}}\notag\\
  &\times\int_{0}^{t(kc\sqrt{\lambda_\ell})^{1/\alpha}} u^{2\alpha-2}\Bigg( C_1\big(1+u^\alpha\big)^{\frac{1-\alpha}{\alpha}}e^{-K_{\ell}u}+\dfrac{C_2}{1+u^\alpha}\Bigg)^2du\notag\\
  &\leq \dfrac{2\gamma^2}{|M_\ell|^2}\big(kc\sqrt{\lambda_\ell}\big)^{\frac{1-2\alpha}{\alpha}}\notag\\
  &\times\Bigg[\int_{0}^{1} u^{2\alpha-2}\Bigg( C_1\big(1+u^\alpha\big)^{\frac{1-\alpha}{\alpha}}e^{-K_{\ell}u}+\dfrac{C_2}{1+u^\alpha}\Bigg)^2du\notag\\
  &+ \int_{1}^{t(kc\sqrt{\lambda_\ell})^{1/\alpha}} u^{2\alpha-2}\Bigg( C_1\big(1+u^\alpha\big)^{\frac{1-\alpha}{\alpha}}e^{-K_{\ell}u}+\dfrac{C_2}{1+u^\alpha}\Bigg)^2du\Bigg]\notag\\
  &=: 2	\dfrac{\gamma^2}{|M_\ell|^2}\big(kc\sqrt{\lambda_\ell}\big)^{\frac{1-2\alpha}{\alpha}}(\Xi_1+\Xi_2).
\end{align}

Note that the term $\Xi_1$ in \eqref{S1} can be bounded by
\begin{align}\label{J1}
\Xi_1\leq \frac{\big(2^{\frac{1-\alpha}{\alpha}}C_1+C_2\big)^2}{2\alpha-1},
\end{align}
while the term $\Xi_2$ can be bounded by
\begin{align}\label{J2}
    \Xi_2&=\int_{1}^{t(kc\sqrt{\lambda_\ell})^{1/\alpha}} u^{2\alpha-2}\Big( C_1\big(1+u^\alpha\big)^{\frac{1-\alpha}{\alpha}}e^{-K_{\ell}u}+\dfrac{C_2}{1+u^\alpha}\Big)^2du\notag\\
    &\leq \int_{1}^{t(kc\sqrt{\lambda_\ell})^{1/\alpha}} u^{2\alpha-2}\Big( C_1\big(2u^\alpha\big)^{\frac{1-\alpha}{\alpha}}e^{-K_{\ell}u}+C_2u^{-\alpha}\Big)^2du\notag\\
    &\leq\int_{1}^{t(kc\sqrt{\lambda_\ell})^{1/\alpha}} \Big( C_1 2^{\frac{1-\alpha}{\alpha}}+C_2u^{-1}\Big)^2du\notag\\
    &\leq t\big(2^{\frac{1-\alpha}{\alpha}}C_1+C_2\big)^2 (kc\sqrt{\lambda_\ell})^{\frac{1}{\alpha}}.
\end{align}

Substituting \eqref{J1} and \eqref{J2} in \eqref{S1} we obtain, for $\ell>\varkappa$, $ \sigma_{\ell,t,\alpha}^2\leq C_3^I \lambda_\ell^{\frac{1-\alpha}{\alpha}}$,
where 
$
    C_3^I:= \frac{\gamma^2}{|M_{\lfloor\varkappa\rfloor+1}|^2}(ck)^{\frac{1-2\alpha}{\alpha}}\big(2^{\frac{1-\alpha}{\alpha}}C_1+C_2\big)^2\Big(\frac{2^{1/\alpha}}{2\alpha-1}+t(ck)^{1/\alpha}\Big)
$
and $\lfloor \cdot\rfloor$ is the floor function.

Finally, we let $\ell=\varkappa$  (i.e., $\lambda_{\ell}=\omega^2$). Then by \eqref{var} and \eqref{smallA0} we write
\begin{align*}
\sigma_{\varkappa,t,\alpha}^2&=\dfrac{c^4}{\alpha^2}\int_{0}^{t} r^{4\alpha-2}\Big\vert E_{\alpha,2\alpha-1}\big(-z_{\varkappa}r^{\alpha}\big)-(1-\alpha)E_{\alpha,2\alpha}\big(-z_{\varkappa}r^{\alpha}\big)\Big\vert^2dr\\
&\leq \dfrac{c^4}{\alpha^2}\Big(\int_{0}^{t} r^{4\alpha-2}\Big\vert E_{\alpha,2\alpha-1}\big(-z_{\varkappa}r^{\alpha}\big)\Big\vert^2 dr+(1-\alpha)^2\int_{0}^{t} r^{4\alpha-2}\Big\vert E_{\alpha,2\alpha}\big(-z_{\varkappa}r^{\alpha}\big)\Big\vert^2 dr\Big)\\
&\leq 2\dfrac{C_3^2c^4(1+(1-\alpha)^2)}{\alpha^2}\int_{0}^{t} \dfrac{r^{4\alpha-2}}{(1+z_\varkappa r^\alpha)^2}dr\\
&\leq 2\dfrac{C_3^2c^4(1+(1-\alpha)^2)}{\alpha^2}\int_{0}^{t}r^{4\alpha-2}dr= C_{2}^It^{4\alpha-1},
\end{align*}
where 
\[
C_2^I:=\frac{2C_3^2c^4(1+(1-\alpha)^2)}{\alpha^2(4\alpha-1)}.
\]
Thus the proof is complete.
\end{proof}

\begin{rem}
By equations \eqref{z1z2} and \eqref{Ml} we have  $z_{0}^{+}=c^2\gamma^{-1}$, $z_{0}^{-}=0$, and $M_0=1$. Thus when $\ell=0$ there holds, for $t\ge0$,
\begin{align}\label{sig0}
   \sigma_{0,t,\alpha}^2&= \int_{0}^{t} r^{2\alpha-2}\Big\vert 1-E_{\alpha,\alpha}\big(-c^2\gamma^{-1}r^{\alpha}\big)\Big\vert^2dr\notag\\
   &\leq (1+C_3^2)\int_{0}^{t} r^{2\alpha-2}dr=\dfrac{1+C_3^2}{2\alpha-1}t^{2\alpha-1},
\end{align}
where the last step used that $E_{\alpha,\alpha}(-x^\alpha)\leq C_3$, $x\geq0$ (see equation \eqref{smallA0}). 
\end{rem}

\subsection{Stochastic integrals of Mittag-Leffler functions}\label{Sec3}

In this subsection, we provide some tools and technical results from the theory of stochastic integrals concerning $L_2(\bS^2)$-valued time-delayed Brownian motions on $\bS^2$. We also establish a novel result that provides a rigorous foundation for taking fractional integrals of stochastic integrals.

Let $L_2([s,t])$, $[s,t]\subset\R_{+}$, be a Hilbert space which consists of all real-valued square-integrable functions on $[s,t]$, i.e., $\{g:\int_{s}^{t} |g(u)|^2du< \infty\}$. For a non-random function $g\in  L_{2}([s,t])$ and a real-valued Brownian motion $B(t)$, stochastic integrals $\int_{s}^{t}g(u)dB(u)$, in the $L_{2}(\Omega\times[s,t])$ sense, are well-defined (see \cite{Kuo2006}). We can similarly define stochastic integrals, in the $L_{2}(\Omega\times[s,t])$ sense, of the form $\int_{s}^{t}g(u)dB_{\tau}(u)$, $t>s\geq\tau$, where $B_{\tau}(\cdot)$ is a time-delayed Brownian motion and $g\in  L_{2}([s,t])$.

Following \cite{Alodat2022} (see also \cite{DaPrato}) we define a time-delayed Brownian motion on the unit sphere $\bS^2$. In particular,  
let $W_\tau(t)$, $\tau\geq0$, be an $L_2(\bS^2)$-valued time-delayed Brownian motion. Then, for each real-valued $f$ in $L_2(\mathbb{S}^2)$, the stochastic process $\langle W_\tau(t),f \rangle_{L_2(\bS^2)}$ is a real-valued, time-delayed Brownian motion. Also, for arbitrary real-valued $f, g \in L_2(\bS^2)$, $t, s \ge \tau$, there holds
\[
\bE \big[\langle W_\tau(t),f\rangle_{L_2(\bS^2)} \langle W_\tau(s),f\rangle_{L_2(\bS^2)}\big] = 
\min(s,t)\bE \big[\langle W_\tau(\tau+1), f\rangle_{L_2(\bS^2)}^2\big]
\]
and
\[
\bE [\langle W_{\tau}(t),f\rangle_{L_2(\bS^2)} \langle W_{\tau}(s),g\rangle_{L_2(\bS^2)}] = \min(s,t)
\bE [\langle Qf,g\rangle_{L_2(\bS^2)}],
\]
where $Q$ is the covariance operator for 
the law of $W_{\tau}(1+\tau)$, see \cite[Section 2.3.1] {DaPrato}.

\begin{defin}[{\rm\cite{Alodat2022}}]\label{BRo}   
An $L_2(\bS^2)$-valued stochastic process $W_\tau(t)$, for $t \ge \tau$,
is called a Q-Wiener process if (i) $W_\tau(\tau) = 0$, (ii) $W_\tau$ has continuous trajectories, (iii) $W_\tau$ has independent increments, and (iv) $W_\tau(t)-W_\tau(s)$ follows the normal distribution $\mathcal{N}(0,(t-s)Q)$, $t\geq s\geq\tau$.
\end{defin}
For the complete orthonormal basis 
$Y_{\ell,m}$, there is a sequence of non-negative numbers $\mathcal{A}_{\ell}$,
\[
Q Y_{\ell,m} = \mathcal{A}_{\ell} Y_{\ell,m},\ 
\ell=0,1,2,\ldots; m=-\ell,\ldots,\ell.
\]

\begin{prop}[{\rm\cite{Alodat2022}}]
If $Q$ is of trace class, i.e., ${\rm Tr}(Q)<\infty$, then the following condition
holds true
\begin{equation}\label{eq:condAell}
		\sum_{\ell=0}^{\infty} (2\ell+1)\calA_{\ell} < \infty.
\end{equation}    
\end{prop}

\begin{prop}\label{RepW}
    Let $W_{\tau}$ be an $L_{2}(\bS^2)$-valued time-delayed Brownian motion on $\bS^2$. Then $W_{\tau}$ admits the following $L_{2}(\Omega\times\bS^2)$ representation
\begin{align}\label{Win}
	W_{\tau}(t)&\stackrel{d}{=} \sum_{\ell=0}^\infty\sum_{m=-\ell}^{\ell} \sqrt{\calA_{\ell}}\beta_{\ell,m,\tau}(t) Y_{\ell,m},\quad t\geq\tau\geq0,
\end{align}
where the angular power spectrum $\{\calA_{\ell},\ \ell\in\N_0\}$ of $W_{\tau}$ satisfies the condition \eqref{eq:condAell} and $\{ \beta_{\ell,m,\tau}(t),\ \ell\in\N_0,\ m=-\ell,\dots, \ell\}$ is a set of independent, real-valued time-delayed Brownian motions with variances $1$ when $t=\tau+1$.
\end{prop}


We denote by $\widehat{W}_{\ell,m,\tau}$ the Fourier coefficients for $W_{\tau}$ (see Proposition \ref{RepW}), i.e., 
\begin{align}\label{W_hat}
    \widehat{W}_{\ell,m,\tau}(t):=\sqrt{\calA_{\ell}}\beta_{\ell,m,\tau}(t),\ t\geq\tau, \ell\in\N_{0}, m=-\ell,\dots,\ell.
\end{align}
For a real-valued function $g\in L_{2}(\bS^2)$ and an $L_2(\bS^2)$-valued Brownian motion $W_{\tau}(t)$ on $\bS^2$, stochastic integrals $\int_{s}^{t}g(u)dW_{\tau}(u)$, in the $L_{2}(\Omega\times\bS^2)$ sense, can be defined as follows. 
\begin{defin}\label{def1}
	Let $ W_{\tau}$ be an $L_2(\bS^2)$-valued time-delayed Brownian motion on $\bS^2$. Let $\widehat{W}_{\ell,m,\tau}(t)$, be the Fourier coefficients for $ W_{\tau}$ defined by \eqref{W_hat}. Let $g\in  L_{2}(\bS^2)$ be real-valued. Then, for $t>s\geq\tau$, the stochastic integral $\int_s^t g(u)dW_{\tau}(u)$ is defined as
	\[
	\int_s^t g(u)dW_{\tau}(u):= \sum_{\ell=0}^\infty \sum_{m=-\ell}^\ell\left( \int_s^t g(u)d\widehat{W}_{\ell,m,\tau}(u)\right) Y_{\ell,m},
	\]
 where the above expansion is convergent in the $L_{2}(\Omega\times\bS^2)$ sense.
\end{defin}

The following result gives expressions for stochastic integrals of the form $\int_s^t g(u)dW_{\tau}(u)$ in terms of the angular power spectrum of $W_{\tau}$ \cite{Alodat2022}.
\begin{prop}[\cite{Alodat2022}]\label{Prop1}
	Let $ W_{\tau}$ be an $L_2(\bS^2)$-valued time-delayed Brownian motion on $\bS^2$. Let $\{\calA_\ell:\ell\in\N_{0}\}$ be the angular power spectrum of $W_{\tau}$. Then, for $t>s\geq\tau$, the stochastic integral $\int_s^t g(u)dW_{\tau}(u)$ given by Definition {\rm\ref{def1}} satisfies
	\[
	\bE	\left[ \bigg\|\int_s^t g(u)dW_{\tau}(u)  \bigg\|_{L_{2}(\bS^2)}^2\right]= \sum_{\ell=0}^{\infty}(2\ell+1)\calA_{\ell} \int_s^t |g(u)|^2 du.
	\]
\end{prop}

By It\^o's isometry (see \cite{Oksendal2003}) and \cite[Theorem 2.3.4]{Kuo2006}, the following result can be straightforwardly established.
\begin{prop}\label{PropVar}
For $\alpha\in(\frac{1}{2},1]$, let 
$\calI_{\ell,m,\alpha}(t)$, $\ell\in\N_{0}$, $m=0,\dots,\ell$, be stochastic integrals defined, for $t>0$, as
\begin{align}\label{Int}
\mathcal{I}_{\ell,m,\alpha}(t):=\int_{0}^{t} \psi_{\ell,\alpha}(t-u)d\beta_{\ell,m}(u),
\end{align}
where $\beta_{\ell,m}(u)$ are independent real-valued Brownian motions with variance $1$ at $u=1$. Then, for $m=0,\dots,\ell$, $\ell\in \N_{0}$,  
$\mathcal{I}_{\ell,m,\alpha}(t)\sim \mathcal{N}(0,\sigma_{\ell,t,\alpha}^2)$, where $\sigma_{\ell,t,\alpha}^2$ is given by {\rm \eqref{var}}.
\end{prop}

One of the critical aspects of our analysis is the ability to handle fractional integrals of stochastic integrals. This capability is pivotal for extending the range of applications of stochastic processes in various domains. To address this, we establish the following theorem, which provides a rigorous foundation for taking fractional integrals of stochastic integrals. 

\begin{theo}\label{fracstochin}
   Let $B(t)$ be a real-valued Brownian motion with variance $1$ at $t=1$. Let $\Upsilon(t,s): \R_{+}\times\R_{+}\to \R$ be a deterministic function such that both $\Upsilon(t,s)$ and  ${_s}\calJ_{t}^{\alpha} \Upsilon(t,s)$, $\alpha\in(0,1]$, are continuous in $t$ and $s$ in some region of the 
	$ts$-plane. Suppose that for $\alpha\in(0,1]$
	the stochastic integrals $\int_{0}^{t}\Upsilon(t,r)dB(r)$ and $\int_{0}^{t} \left({}_{r} \calJ_{t}^{\alpha} \Upsilon(t,r)\right)dB(r)$ are well-defined. Assume moreover that 
\begin{align}\label{con:Fubini}
    \int_{0}^{t} (t-s)^{\alpha-1}\Bigg(\int_{0}^{s}|\Upsilon(s,r)|^2 dr\Bigg)^{\frac{1}{2}}ds< \infty,
\end{align}
then for all $t>0$ there holds
    \[
     {_0}\calJ_{t}^{\alpha}\int_{0}^{t} \Upsilon(t,r) d B(r)\stackrel{d}{=} \int_{0}^t {_r}\calJ_{t}^{\alpha} \Upsilon( t,r) d B(r).
    \]
\end{theo}

\begin{proof}
    Using \eqref{fracint} and \eqref{con:Fubini} with the stochastic Fubini theorem (see \cite[Section 4.5] {DaPrato}) we get 
    \begin{align*}
        {_0}\calJ_{t}^{\alpha}\int_{0}^{t} \Upsilon(t,r) d B(r)&\stackrel{d}{=} \dfrac{1}{\Gamma(\alpha)}\int_{0}^{t} (t-s)^{\alpha-1} \Bigg(\int_{0}^{s}\Upsilon(s,r)dB(r)\Bigg)ds\\
        &\stackrel{d}{=} \int_{0}^{t} \Bigg(\dfrac{1}{\Gamma(\alpha)}\int_r^t (t-s)^{\alpha-1}\Upsilon(s,r)ds\Bigg)dB(r)\\
        & \stackrel{d}{=} \int_{0}^t {_r}\calJ_{t}^{\alpha} \Upsilon( t,r) d B(r),
    \end{align*}
    thus the proof is complete.
\end{proof}
Note that if we replace the function $\Upsilon(t,r)$ with the function $\Upsilon(t-r)$ in Theorem \ref{fracstochin}, we could write  
\[
{_0}\calJ_{t}^{\alpha}\int_{0}^{t} \Upsilon(t-r) d B(r)\stackrel{d}{=}  \int_{0}^t {_0}\calJ_{r}^{\alpha} \Upsilon( r) d B(r).
\]

\begin{rem}\label{frac_psi}
By properties of the Mittag-Leffler functions, both functions $ \psi_{\ell,\alpha}$ and $\psi^{b}_{\ell,\alpha}$, defined by \eqref{psi} and \eqref{psi_balpha} respectively, satisfy the conditions of Theorem {\rm\ref{fracstochin}}. Thus, by the equation \eqref{psi} with the help of the relation \rm{\cite[(1.100)]{Podlubny}} we have
\begin{align}\label{neqpsi1}
&{_0}\calJ_{t}^{\alpha}\int_{0}^{t} \psi_{\ell,\alpha}( t-s) d B(s)\stackrel{d}{=}\int_{0}^{t}{_0}\calJ_{r}^{\alpha}\psi_{\ell,\alpha}( s) d B(s)\notag\\
&\stackrel{d}{=}\dfrac{\gamma}{M_\ell}\int_{0}^{t}s^{2\alpha-1}(E_{\alpha,2\alpha}\big(-z_{\ell}^{-}s^{\alpha}\big)-E_{\alpha,2\alpha}\big(-z_{\ell}^{+}s^{\alpha}\big))d B(s)\notag\\
&\stackrel{d}{=}\int_{0}^{t} \psi^{2\alpha}_{\ell,\alpha}(s)dB(s)
\end{align}
and similarly 
\begin{align}\label{neqpsi2}
{_0}\calJ_{t}^{2\alpha}\int_{0}^{t} \psi_{\ell,\alpha}( t-s) d B(s)
&\stackrel{d}{=}\int_{0}^{t} \psi^{3\alpha}_{\ell,\alpha}(s)dB(s),
\end{align}
where 
\begin{align}\label{psi_balpha}
   \psi^{b}_{\ell,\alpha}(t):= \dfrac{\gamma}{M_\ell}t^{b-1}(E_{\alpha,b}\big(-z_{\ell}^{-}t^{\alpha}\big)-E_{\alpha,b}\big(-z_{\ell}^{+}t^{\alpha}\big)). 
\end{align}
\end{rem}
\section{Spectral solutions to the model}\label{specsol}
\subsection{Solution of the homogeneous equation}\label{Sec_sol_Hom}
Consider a random field $U^{H}(t)\in L_2(\Omega\times\bS^2),\ t\in (0,\infty)$. We say that the field $U^H$ satisfies the equation \eqref{Hom}, in the $L_2(\Omega\times\bS^2)$ sense, if for $t>0$, there holds
\[
	\sup_{t > 0} \bE \Bigg[\left\| \calE U^H(t) \right\|^2_{L_2(\bS^2)}\Bigg] = 0.
\]	

In this section, we derive the solution $U^{H}(t)$, $t\geq0$, to \eqref{Hom}, subject to the random initial conditions $U^{H}(0)=\xi$ and $\partial U^H/\partial t = 0$ at $t=0$.
It is worth noting that the solution to the homogeneous equation \eqref{Hom} is derived in \cite{Leonenko}, where complex coefficients $\widehat{\xi}_{\ell,m}$ were used, corresponding to their use of complex spherical harmonics. With the adaptation discussed in Remark \ref{randrem}, we adjust these conditions to accommodate real spherical harmonics. Consequently, the solution $U^{H}$, in the $L_2(\Omega\times\bS^2)$ sense, is expressed by the following series expansion:
\begin{align}\label{homo}
U^{H}(t) = \sum_{\ell=0}^\infty \sum_{m=-\ell}^{\ell} \widehat{\xi}_{\ell,m}F_{\ell,\alpha}(t) Y_{\ell,m}, \quad t\in[0,\infty),
\end{align}
where $F_{\ell,\alpha}(\cdot)$ is given by \eqref{Fl} only for the case $\varkappa\notin\N$ and $\ell \neq \varkappa$. However, it's crucial to emphasize that the expansion provided in \cite{Leonenko} is not applicable when $\ell = \varkappa$. In this paper, we shall derive the solution for all $\ell \in \N_0$ and $\varkappa\in\N$. To avoid repeating the proof provided in \cite{Leonenko}, we will provide a detailed proof only for the case $\ell=\varkappa$.

We seek a solution, in $L_2(\Omega\times\bS^2)$, to \eqref{Hom} in the form
	\begin{align*}
		U^{H}(t) = \sum_{\ell=0}^\infty \sum_{m=-\ell}^{\ell} \widehat{U^H}_{\ell,m}(t) Y_{\ell,m}, \quad t\in[0,\infty),
	\end{align*}
	where the random coefficients $\widehat{U^H}_{\ell,m}$, $\ell\in\N_{0}$, $m=-\ell,\dots,\ell$, are given by
	\begin{align}\label{four-hom}
		\widehat{U^H}_{\ell,m}(t)=\int_{\bS^2}U^{H}(\bsx,t)Y_{\ell,m}(\bsx)\mu(d\bsx),\quad t\in[0,\infty).
	\end{align}
Following the proof of \cite[Theorem 3]{Leonenko}, with \eqref{Homcond1}, we arrive at
\begin{align}\label{Eq15}
   \widehat{U^H}_{\ell,m}(t)=\xi_{\ell,m}\Big(1-c^2k^2\lambda_\ell t^{2\alpha}\sum_{r=0}^{\infty}\big(-c^2k^2\lambda_{\ell}t^{2\alpha}\big)^r\; E_{\alpha,1+2\alpha(r+1)}^{r+1}\big(-c^2\gamma^{-1}t^\alpha\big)\Big). 
\end{align}

Let $\ell=\varkappa$. Then by solving \eqref{z1andz2} for $z_1(t)$ and $z_2(t)$ gives
$z_1(t)=z_2(t)=-z_{\varkappa}t^{\alpha}$, where $z_{\varkappa}$ is defined by \eqref{za}. 
Then using the following relation \cite[Corollary 3.3]{Soubhia2010}
\begin{align}\label{1equalz1}
    \sum_{r=0}^{\infty}(-x^2)^r E_{a,2ar+b}^{r+1}(2x)=E_{a,b}(x)+x\dfrac{d}{dx}E_{a,b}(x)
\end{align}
and \cite[equation (4.3.1)]{Gorenflo2014} we can write \eqref{Eq15}, for $\ell=\varkappa$, as
\begin{align}\label{kapequalell}
    \widehat{U^H}_{\varkappa,m}(t)&=\xi_{\varkappa,m}\Big(1-(z_{\varkappa}t^{\alpha})^2\sum_{r=0}^{\infty}\big(-(z_{\varkappa}t^{\alpha})^2\big)^r\; E_{\alpha,1+2\alpha(r+1)}^{r+1}\big(-2z_{\varkappa}t^{\alpha}\big)\Big)\notag\\
    &=\xi_{\varkappa,m}\Big(1-(z_{\varkappa}t^{\alpha})^2\Big[E_{\alpha,1+2\alpha}(-z_{\varkappa}t^{\alpha})-z_{\varkappa}t^{\alpha}E_{\alpha,1+3\alpha}^{2}(-z_{\varkappa}t^{\alpha})\Big]\Big),
\end{align}
where $E_{a,b}^{2}(\cdot)$ is the three-Parametric Mittag-Leffler function (see \eqref{qMitt}).

Using \cite[relations (5.1.12)-(5.1.14)]{Gorenflo2014} then \eqref{kapequalell} becomes
\begin{align}\label{FA}
    \widehat{U^H}_{\varkappa,m}(t)&=\xi_{\varkappa,m}\Big(1-(z_{\varkappa}t^{\alpha})^2\Big[E_{\alpha,1+2\alpha}(-z_{\varkappa}t^{\alpha})-z_{\varkappa}t^{\alpha}E_{\alpha,1+3\alpha}^{2}(-z_{\varkappa}t^{\alpha})\Big]\Big)\notag\\
    &= \xi_{\varkappa,m}\Big(1+z_{\varkappa}t^{\alpha}\Big[E_{\alpha,\alpha+1}^{2}(-z_{\varkappa}t^{\alpha})-E_{\alpha,\alpha+1}(-z_{\varkappa}t^{\alpha})\Big]\Big)\notag\\
    &=\xi_{\varkappa,m}\Big(E_{\alpha,1}(-z_{\varkappa}t^{\alpha})+\dfrac{1}{\alpha}z_{\varkappa}t^{\alpha}E_{\alpha,1}(-z_{\varkappa}t^{\alpha})\Big)\notag\\
    &= \xi_{\varkappa,m}E_{\alpha,1}(-z_{\varkappa}t^{\alpha})\Big(1+\dfrac{1}{2\alpha}c^2\gamma^{-1}t^{\alpha}\Big),
\end{align}
where the last step used \eqref{za}.

Let $U_{L}^{H}$ be the truncated field of $U^H$. We define $V_{L}^{H}$ as
\begin{align}\label{VL}
		V_{L}^{H}(t):&= {^C}D^{2\alpha}U^{H}_L(t)+ {^C}D^{\alpha}U^{H}_L(t)\notag\\
  &=\sum_{\ell=1}^{L}\sum_{m=-\ell}^\ell
\Big(\dfrac{1}{c^2} {^C}D_t^{2\alpha}F_{\ell,\alpha}(t)+\dfrac{1}{\gamma}{^C}D_t^{\alpha}F_{\ell,\alpha}(t)\Big) \widehat{\xi}_{\ell,m} Y_{\ell,m}
	\end{align}
and $V^{H}$ as
\begin{equation}\label{eq:defVH}
V^{H}(t):= \sum_{\ell=1}^{\infty}\sum_{m=-\ell}^\ell
\Big(\dfrac{1}{c^2} {^C}D_t^{2\alpha}F_{\ell,\alpha}(t)+\dfrac{1}{\gamma}{^C}D_t^{\alpha}F_{\ell,\alpha}(t)\Big) \widehat{\xi}_{\ell,m} Y_{\ell,m},
\end{equation}
where $F_\ell,\ \ell\in\N$, are defined in \eqref{Fl}.
\begin{lem}\label{unif-sure}
	Let $\xi$ be a centered, $2$-weakly isotropic Gaussian random field on $\bS^2$. Let $\{\calC_{\ell}:\ell\in\N_{0}\}$, the angular power spectrum of $\xi$, satisfy 
\begin{equation}\label{eq:condCell}
\sum_{\ell=0}^{\infty} \lambda_{\ell}(2\ell+1)\calC_{\ell}<\infty.
\end{equation}
Let $t_0>0$ be given. Then for $t\geq t_0$, $V_{L}^{H}$ is convergent to $V^H$ in $L_{2}(\Omega\times\bS^2)$ as $L\to\infty$, that is
\[
\sup_{t \geq t_0} \bE \left\| V_{L}^{H}(t)-V^{H}(t) \right\|^2_{L_2(\bS^2)} 
\to 0,\quad  \text{ as } L \to \infty.
\]
\end{lem}

\begin{proof}
Let $t\geq t_0$ and $L<L^{\prime}$. Then by Parseval's formula and \eqref{dF1}, we obtain
\begin{align}
\bE \Big[\|V^H_L(t) - V^H_{L^{\prime}}(t) \|^2_{L_2(\bS^2)}\Big]
&= \bE \Big[\sum_{\ell=L+1}^{L^{\prime}}
\sum_{m=-\ell}^{\ell}
|\widehat{\xi}_{\ell,m}|^2 \notag\\
&\times \Big|\dfrac{1}{c^2} {^C}D_t^{2\alpha}F_{\ell,\alpha}(t)+\dfrac{1}{\gamma}{^C}D_t^{\alpha}F_{\ell,\alpha}(t)\Big|^2\Big] \nonumber\\
&\le\bE \Big[\sum_{\ell=L+1}^{L^{\prime}}
\sum_{m=-\ell}^{\ell}
\Big|\mathcal{H}(t)\sqrt{\lambda_\ell}\Big|^2
|\widehat{\xi}_{\ell,m}|^2  \Big]\nonumber\\
&\le\sum_{\ell=L+1}^{L^{\prime}}
(\mathcal{H}(t))^2
\lambda_\ell(2\ell+1) \calC_{\ell} 
\nonumber\\
&\le (\mathcal{H}(t_0))^2
\sum_{\ell=L+1}^{L^{\prime}} \lambda_\ell(2\ell+1) \calC_{\ell}\label{eq:cauchyCond},
\end{align}
where the function $\mathcal{H}(\cdot)$ is defined in \eqref{updf1}.
Using the condition \eqref{eq:condCell} with a given $\epsilon>0$, there is 
$L_0$ independent of $t$ such that for $L,L^{\prime}  \ge L_0$, 
the RHS of \eqref{eq:cauchyCond} is smaller than $\epsilon$. So
$\{V_L^H\}$ is a Cauchy sequence in $L_2(\Omega \times \bS^2)$ and hence it is convergent. We define the limit of $V_L^H(t)$ as $L \to \infty$ as in \eqref{eq:defVH}.

Note that by the condition \eqref{eq:condCell} and the upper bound \eqref{dF1} we obtain
\[
\bE \Big[\|V^H(t)\|^2_{L_2(\bS^2)}\Big]\le (\mathcal{H}(t_0))^2
\sum_{\ell=0}^{\infty} \lambda_\ell(2\ell+1) \calC_{\ell}<\infty,
\]
which guarantees that $V^H$ is well-defined in the $L_2(\Omega\times\bS^2)$ sense. Thus the proof is complete.
\end{proof}

\begin{rem}\label{remGL}
  Let
\begin{align}\label{GL}
    G^H_L(t) := k^2\Delta_{\bS^2} U^{H}_L(t)=\sum_{\ell=1}^L (-\lambda_{\ell}) F_{\ell,\alpha}(t)
\sum_{m=-\ell}^\ell \widehat{\xi}_{\ell,m} Y_{\ell,m}.
\end{align}
Using the upper bounds {\rm\cite[(24) and (25)]{Leonenko}} we can show that 
$\{G_L^H\}$ is a Cauchy sequence in $L_2(\Omega \times \bS^2)$ and hence it is convergent. We define the limit of $G_L^H$ as $L \to \infty$ by
\begin{align}\label{G}
    G^H(t):=\sum_{\ell=1}^\infty \sum_{m=-\ell}^\ell (-\lambda_{\ell}) F_{\ell,\alpha}(t)\widehat{\xi}_{\ell,m} Y_{\ell,m},
\end{align}
which is convergent in $L_2(\Omega \times \bS^2)$.  
\end{rem}

\begin{prop}\label{Theo1}
Let the angular power spectrum $\{\calC_{\ell}:\ell\in\N_{0}\}$ of the isotropic Gaussian random field $\xi$ satisfy assumption \eqref{eq:condCell}.
Then the random field $U^{H}$ defined by \eqref{homo} satisfies $\calE U^H=0$,
in the $L_2(\Omega\times\bS^2)$ sense, under the initial conditions \eqref{Homcond1}.
In particular, for a given $t_0>0$, there holds
\[
	\sup_{t > t_0} \bE \Bigg[\left\| \calE U^H(t) \right\|^2_{L_2(\bS^2)}\Bigg] = 0.
\]	
\end{prop}
\begin{proof}
 Let us fix $L\geq1$. By \eqref{VL} and \eqref{GL} with the help of \eqref{NewUp} we obtain  
 \[
 \bE \Big[ \left\|V^H_L(t)-G^H_L(t)\right\|^2_{L_2(\bS^2)}\Big]=0.
 \]
By the triangle inequality, we obtain
 \begin{align}
&I^H(t):= \bE \Bigg[\left\| \calE U^H(t) \right\|^2_{L_2(\bS^2)}\Bigg]\notag\\
&=\bE \Bigg[\left\| \frac{1}{c^2}{^C}D_t^{2\alpha} U^H(t) +\frac{1}{\gamma} {^C}D_t^{\alpha}U^H(t) - k^2\Delta_{\bS^2} U^H(t) \right\|^2_{L_2(\bS^2)}\Bigg] \notag\\
&\leq  2 \bE \Big[ \left\|V^H(t)-V^H_L(t)\right\|^2_{L_2(\bS^2)}\Big]+2 \bE \Big[ \left\|G^H(t)-G^H_L(t)\right\|^2_{L_2(\bS^2)}\Big].
 \end{align}
Thus, the result follows by Lemma \ref{unif-sure} and Remark \ref{remGL}.
\end{proof}

The following result shows that the solution $U^{H}$ to the equation \eqref{Hom} is a centered, 2-weakly isotropic Gaussian random field.
 The angular power spectrum $\{\calC_{\ell}\}$ of $\xi$ is assumed to satisfy the condition 
 \eqref{eq:condCell}. To guarantee that the condition is satisfied we assume that
$\{\calC_{\ell}:\ell\in\N_{0}\}$ decays algebraically with order $\kappa_1>4$, i.e., there exist constants $\widetilde{D},\widetilde{C}>0$ such that
\begin{align}\label{New-Cl}
	\calC_{\ell}\leq\begin{cases} 
		\widetilde{D}, & \ell=0, \\
		\widetilde{C}\ell^{-\kappa_1}, & \ell\geq1,\ \kappa_1>4.
	\end{cases}
\end{align}
For $\kappa_1>4$ and $\widetilde{C}>0$, we let
\begin{align}\label{Ckappa}
	\widetilde{C}_{\kappa_1}(\varkappa):=\bigg(\widetilde{C}\Big(\frac{2\lfloor \varkappa\rfloor^{2-\kappa_1}}{\kappa_1-2}+\frac{\lfloor \varkappa\rfloor^{1-\kappa_1}}{\kappa_1-1}\Big)\bigg)^{1/2},
\end{align}
where $\lfloor\cdot\rfloor$ is the floor function.

\begin{prop}\label{Gausshomo}
Let the field $U^{H}$, given in {\rm(\ref{homo})}, be the solution to the equation 
$\calE U^H = 0$
under the initial conditions: (i) $\partial U^H/\partial t = 0$ at $t=0$,
(ii) $U^{H}(0)= \xi$, where $\xi$ is a centered, $2$-weakly isotropic Gaussian random field on $\bS^2$. Let $\{\calC_{\ell}:\ell\in\N_{0}\}$, the angular power spectrum of $\xi$, satisfy \eqref{New-Cl}. Then, for $t\in[0,\infty)$, $U^{H}(t)$ is a centred, $2$-weakly isotropic Gaussian random field on $\bS^2$, and its random coefficients
\begin{align}\label{xile}
		\widehat{U^{H}}_{\ell,m}(t)=F_{\ell,\alpha}(t) \widehat{\xi}_{\ell,m},
\end{align} 
satisfy for $\ell,\ell^\prime\in\N_{0}$, $m=-\ell,\dots,\ell$ and $m^\prime=-\ell^\prime,\dots,\ell^\prime$, 
\begin{align}\label{var-xile}
		\bE \left[\widehat{U^{H}}_{\ell,m}(t)\widehat{U^{H}}_{\ell^\prime,m^\prime}(t)\right]=|F_{\ell,\alpha}(t)|^2 \calC_{\ell}\delta_{\ell\ell^\prime}\delta_{mm^\prime},
	\end{align}
where $F_{\ell,\alpha}(t)$ is given by \eqref{Fl} 
and $\delta_{\ell\ell^\prime}$ is the Kronecker delta function.
\end{prop}

\begin{proof}
Let $t\in[0,\infty)$. Since $\xi$ is centred we have $\bE[U^{H}(t)]=0$.  By {\rm(\ref{homo})} we can write
	\begin{align*}
		\bE \left[U^{H}(\bsx,t)U^{H}(\bsy,t)\right]&=\sum_{\ell=0}^\infty\sum_{\ell^\prime=0}^\infty \sum_{m=-\ell}^{\ell}
		\sum_{m\prime=-\ell^\prime}^{\ell^\prime} F_{\ell,\alpha}(t) F_{\ell^{\prime},\alpha}(t)\\ &\times\bE\left[\widehat{\xi}_{\ell,m}\widehat{\xi}_{\ell^\prime,m^\prime}\right] Y_{\ell,m}(\bsx)Y_{\ell^\prime,m^\prime}(\bsy).
	\end{align*}
	Since  $\xi$ is a centred, 2-weak isotropic Gaussian random field we have $\bE\left[\widehat{\xi}_{\ell,m}\widehat{\xi}_{\ell^\prime,m^\prime}\right]=\calC_{\ell}  \delta_{\ell \ell^\prime} \delta_{m m^\prime}$, and by the addition theorem (see equation \eqref{addition}) and \cite[Corollary 2.1]{Alodat2022}, we have
	\begin{align*}
		\bE\left[U^{H}(\bsx,t)U^{H}(\bsy,t)\right]&=\sum_{\ell=0}^\infty |F_{\ell,\alpha}(t)|^2\calC_{\ell}\sum_{m=-\ell}^{\ell} Y_{\ell,m}(\bsx)Y_{\ell,m}(\bsy)\\
		&=\sum_{\ell=0}^\infty |F_{\ell,\alpha}(t)|^2(2\ell+1)\calC_\ell P_{\ell}(\bsx\cdot\bsy),
	\end{align*}
where $P_{\ell}(\cdot)$, $\ell\in\N_{0}$, is the Legendre polynomial of degree $\ell$. As $P_{\ell}(\bsx\cdot\bsy)$ depends only on the inner product of $\bsx$ and $\bsy$, we conclude that the covariance function $\bE\left[U^{H}(\bsx,t)U^{H}(\bsy,t)\right]$ is rotationally invariant. Note that for $\ell,\ell^\prime\in\N_{0}$, $m=-\ell,\dots,\ell$ and $m^\prime=-\ell^\prime,\dots,\ell^\prime$, by the 2-weak isotropy of $\xi$, and \cite[Corollary 2.1]{Alodat2022}, we have	
\begin{align*}
\bE \left[\widehat{U^{H}}_{\ell,m}(t)\widehat{U^{H}}_{\ell^\prime,m^\prime}(t)\right]&=F_{\ell,\alpha}(t) F_{\ell^{\prime},\alpha}(t) \bE\left[\widehat{\xi}_{\ell,m}\widehat{\xi}_{\ell^\prime,m^\prime}\right]\\
&=|F_{\ell,\alpha}(t)|^2 \calC_{\ell}\delta_{\ell\ell^\prime}\delta_{mm^\prime}.
\end{align*}
Now we show that $U^{H}$ is Gaussian. By \eqref{Ml} we have $M_0=1$ and by \eqref{z1z2}, \eqref{F1}, and \eqref{F2}, one obtains $F_{0,\alpha}(t)=1$. Thus, the variance of $U^{H}(t)$, can be written as
\begin{align}\label{Homvar1}
    Var[U^{H}(t)]&= \sum_{\ell=0}^\infty |F_{\ell,\alpha}(t)|^2(2\ell+1)\calC_\ell\notag\\
    &= \sum_{\ell=0}^{\lfloor \varkappa\rfloor} |F_{\ell,\alpha}(t)|^2(2\ell+1)\calC_\ell+\sum_{\ell=\lfloor \varkappa\rfloor+1}^\infty |F_{\ell,\alpha}(t)|^2(2\ell+1)\calC_\ell.
\end{align}
If $\varkappa\notin \N$, then by using \eqref{F_upper} and \eqref{New-Cl}, \eqref{Homvar1} becomes
\begin{align}\label{Homvar12}
    Var[U^{H}(t)]&\leq \sum_{\ell=0}^{\lfloor \varkappa\rfloor} |F_{\ell,\alpha}(t)|^2(2\ell+1)\calC_\ell+\Big(\mathcal{H}(t)\Big)^2 \widetilde{C}\sum_{\ell=\lfloor \varkappa\rfloor+1}^\infty \dfrac{(2\ell+1)\ell^{-\kappa_1}}{\lambda_\ell}\notag\\
    & \le \sum_{\ell=0}^{\lfloor \varkappa\rfloor} |F_{\ell,\alpha}(t)|^2(2\ell+1)\calC_\ell+ \Big(\dfrac{\mathcal{H}(t)}{\sqrt{\lambda_{\lfloor \varkappa\rfloor}}}\Big)^2 \widetilde{C}\sum_{\ell=\lfloor \varkappa\rfloor+1}^\infty (2\ell+1)\ell^{-\kappa_1}.
\end{align}
Since $	\sum_{\ell=0}^{\lfloor \varkappa\rfloor} |F_{\ell,\alpha}(t)|^2(2\ell+1)\calC_\ell<\infty$ and
\begin{align}\label{Vpf}
		\sum_{\ell=\lfloor \varkappa\rfloor+1}^{\infty} (2\ell+1)\ell^{-\kappa_1}&\leq \sum_{\ell=\lfloor \varkappa\rfloor}^{\infty} (2\ell+1)\ell^{-\kappa_1}\notag\\
		&\leq\int_{\lfloor \varkappa\rfloor}^\infty \big(2x^{1-\kappa_1}+x^{-\kappa_1}\big)dx\notag\\
		&\leq  \Big(\frac{2\lfloor \varkappa\rfloor^{2-\kappa_1}}{\kappa_1-2}+\frac{\lfloor \varkappa\rfloor^{1-\kappa_1}}{\kappa_1-1}\Big),
	\end{align}
then by \eqref{Homvar12} and \eqref{Vpf} we get
\begin{align}\label{UHfinit}
    Var[U^{H}(t)]\leq \sum_{\ell=0}^{\lfloor \varkappa\rfloor} |F_{\ell,\alpha}(t)|^2(2\ell+1)\calC_\ell+\Big(\dfrac{\mathcal{H}(t)}{\sqrt{\lambda_{\lfloor \varkappa\rfloor}}}\Big)^2 \widetilde{C}(\varkappa))^2<\infty.
\end{align}

Using the inequality \cite[Theorem 1.6]{Podlubny}
\begin{align*}
E_{\alpha,\beta}(-z)\leq \dfrac{C}{1+z},\quad \alpha<2,\ \beta\in\R,\ C>0,\ z\geq0,
\end{align*}
we have $\big|E_{\alpha}\big(-z_\varkappa t^{\alpha}\big)\big|^2\leq C$ and when $\varkappa\in \N$ we obtain, using \eqref{FA},
\begin{align}\label{FAupper}
    |F_{\varkappa,\alpha}(t)|^2&= \big|E_{\alpha}\big(-z_\varkappa t^{\alpha}\big)\big(1+\frac{1}{2\alpha}c^2\gamma^{-1}t^{\alpha}\big)\big|^2\notag\\
    &\leq C^2\big(1+\frac{1}{2\alpha}c^2\gamma^{-1}t^{\alpha}\big)^2\notag\\
    &\leq C^2\big(1+c^2\gamma^{-1}t^{\alpha}\big)^2,
\end{align}

Thus by \eqref{FA} and \eqref{FAupper} we write, when $\varkappa\in\N$,
\begin{align}\label{Homvar2}
    Var[U^{H}(t)]&= |F_{0,\alpha}(t)|^2\calC_0+|F_{\varkappa,\alpha}(t)|^2(2\varkappa+1)\calC_\varkappa\notag\\
    &+\sum_{\substack{\ell=1 \\ \ell\neq \varkappa}}^{\infty} |F_{\ell,\alpha}(t)|^2(2\ell+1)\calC_\ell.
\end{align}
Using similar steps as those for \eqref{Homvar12}, we have $\sum_{\substack{\ell=1 \\ \ell\neq \varkappa}}^{\infty} |F_{\ell,\alpha}(t)|^2(2\ell+1)\calC_\ell<\infty$. Thus, $ Var[U^{H}(t)]<\infty$.

Let $T_{\ell}$, $\ell\in\N_0$, be defined as 
$
		T_{\ell}(t):=\sum_{m=-\ell}^{\ell}\widehat{U^{H}}_{\ell,m}(t)Y_{\ell,m}= \sum_{m=-\ell}^{\ell}F_{\ell,\alpha}(t) \widehat{\xi}_{\ell,m}Y_{\ell,m}.
$
Since $\widehat{\xi}_{\ell,m}$ are centred, independent random variables (by the 2-weak isotropy of $\xi$), then $T_{\ell}$
	is a Gaussian random variable with mean zero and variance (see Remark 6.13 in \cite{MarPec11}), using \eqref{var-xile}, \eqref{addition}, and properties of spherical harmonics, given by
	\begin{align}\label{var-T}
	    Var[T_{\ell}]=|F_{\ell,\alpha}(t)|^2(2\ell+1)\calC_\ell.
	\end{align}
Taking a sequence of independent random variables $U^H_n$, $n\geq1$, of the form $
U^H_n(t):=\sum_{\ell=0}^{n}T_{\ell}(t),
$
then it is easy to show, using \eqref{var-T}, that
	\begin{align}\label{var-Un}
		Var[U^H_n(t)]=\sum_{\ell=0}^{n}Var[T_{\ell}(t)]=\sum_{\ell=0}^{n}|F_{\ell,\alpha}(t)|^2(2\ell+1)\calC_\ell,
	\end{align}
and, using \eqref{UHfinit} when $\varkappa\in\N$ and \eqref{Homvar2} when $\varkappa\notin\N$, 
	\begin{align}\label{Lim}
		\lim_{n\to\infty}Var[U^H_n(t)]=\sum_{\ell=0}^{\infty}|F_{\ell,\alpha}(t)|^2(2\ell+1)\calC_\ell<\infty.
	\end{align}
	The characteristic function $\varphi_{U^H_n(t)}(r)$ of $U^H_n(t)$ can be written, using \eqref{var-T} and \eqref{var-Un}, as
	\begin{align*}
		\varphi_{U^H_n(t)}(r)=\bE\Big[e^{\mi r U^H_n(t)}\Big]&= \prod_{\ell=0}^{n}\bE\Big[e^{\mi r T_{\ell}(t)}\Big]
		=\prod_{\ell=0}^{n}  \varphi_{T_{\ell}(t)}(r)\\
		&=e^{-\frac{1}{2}r^2 \sum_{\ell=0}^{n}|F_{\ell,\alpha}(t)|^2(2\ell+1)\calC_\ell},   
	\end{align*}
	since $T_{\ell},\ \ell\geq0$, are independent Gaussian random variables.
	
\indent
	Hence,
	\[
	\lim_{n\to\infty} \varphi_{U^H_n(t)}(r)= e^{-\frac{1}{2}r^2 \sum_{\ell=0}^{\infty}|F_{\ell,\alpha}(t)|^2(2\ell+1)\calC_\ell},
	\]
which is the characteristic function of some Gaussian random variable with mean zero and variance $\sum_{\ell=0}^{\infty}|F_{\ell,\alpha}(t)|^2(2\ell+1)\calC_\ell<\infty$ (by \eqref{Lim}). Then by the continuity theorem \cite[Theorem 3.3.6]{Durrett} we conclude that the field $U^{H}$ is Gaussian, thus completing the proof. 
\end{proof}


\begin{rem} \label{rem4}
Let $\{\calC_{\ell}:\ell\in\N_{0}\}$ be the angular power spectrum of $\xi$. Then the Fourier coefficients $\widehat{U^H}_{\ell,m}$ of $U^{H}$ can be written, using \eqref{xile}, as
\begin{align}\label{UH-complexFourier}
		\widehat{U^H}_{\ell,m}(t)&:= \sqrt{\calC_{\ell}}F_{\ell,\alpha}(t)Z_{\ell,m},\quad \ell\in\N_0, m=-\ell,\dots,\ell,
	\end{align}
where $\{Z_{\ell,m},\ \ell\in\N_0,\ m=-\ell,\dots,\ell \}$ is a set of independent, real-valued, standard normally distributed random variables.  
	
\indent
Moreover, from \eqref{xile}, we observe that $\sqrt{\calC_{\ell}}Z_{\ell,m} \stackrel{d} {=} \widehat{\xi}_{\ell,m}$ for all $\ell \in \mathbb{N}_0, m=-\ell,\ldots,\ell$. Thus, the solution $U^{H}(t)$, $t\in[0,\infty)$, can be represented as
\begin{align}\label{New-HomSol}
		U^{H}(t)&= \sum_{\ell=0}^{\infty}\sum_{m=-\ell}^\ell \sqrt{\calC_{\ell}}F_{\ell,\alpha}(t)Z_{\ell,m} Y_{\ell,m},
\end{align}
where the above expansion is convergent in $L_2(\Omega\times\bS^2)$.
\end{rem}
\subsection{Solution of the inhomogeneous equation}\label{Sec_sol_inHom}

We say that $U^I$ is a solution to \eqref{Inhom}, in the $L_2(\Omega\times\bS^2)$ sense, if 
\[
\sup_{t \geq \tau} \bE \left\| \calK U^I(t)-\dfrac{c^2}{\Gamma(2\alpha)}\int_{\tau}^{t} (t-s)^{2\alpha-1}dW_{\tau}(s)\right\|^2_{L_2(\bS^2)} = 0.
\]
In this section, we derive the solution $U^{I}(t)$ for $t\in[0,\infty)$ to \eqref{Inhom} in the $L_2(\Omega\times\bS^2)$ sense, which takes the form
\begin{align*}
		U^{I}(t) = \sum_{\ell=0}^\infty \sum_{m=-\ell}^{\ell} \widehat{U^I}_{\ell,m}(t) Y_{\ell,m}, \quad t\in[0,\infty),
	\end{align*}
	where $
		\widehat{U^I}_{\ell,m}(t)=\int_{\bS^2}U^{I}(\bsx,t)Y_{\ell,m}(\bsx)\mu(d\bsx),\ t\in[0,\infty).
$

	By multiplying both sides of equation \eqref{Inhom} by $Y_{\ell,m}$ and integrating over $\bS^2$, we obtain a set of ordinary differential equations 
	\begin{equation}\label{Eqv1}
	\frac{1}{c^2}{^C}D_t^{2\alpha} \widehat{U^I}_{\ell,m}dt +\frac{1}{\gamma} {^C}D_t^{\alpha} \widehat{U^I}_{\ell,m} dt +k^2\lambda_\ell \widehat{U^I}_{\ell,m}dt= d \widehat{W}_{\ell,m,\tau},
	\end{equation}
with $\widehat{U^I}_{\ell,m}(0) =0,\ \frac{d}{d t}\widehat{U^I}_{\ell,m}(t)|_{t=0}=0$,
where $\widehat{W}_{\ell,m,\tau}$ are the Fourier coefficients of $W_\tau$.

\indent
	If $\widehat{W}_{\ell,m,\tau}$ were deterministic differentiable functions in time, then we could write \eqref{Eqv1} as
	\begin{equation}\label{Nn}
		\frac{1}{c^2}{^C}D_t^{2\alpha} \widehat{U^I}_{\ell,m} +\frac{1}{\gamma} {^C}D_t^{\alpha} \widehat{U^I}_{\ell,m}  +k^2\lambda_\ell \widehat{U^I}_{\ell,m}= \frac{d}{dt} \widehat{W}_{\ell,m,\tau}.
	\end{equation}
	By taking the Laplace transform of \eqref{Nn} we get
\begin{equation}\label{lap}
		\bigg(\frac{1}{c^2}z^{2\alpha} +\frac{1}{\gamma}z^{\alpha}+  k^2\lambda_\ell\bigg) \widetilde{U^I}_{\ell,m} (z) = e^{-z\tau} \widetilde{W}_{\ell,m,0}(z),
\end{equation}
	where $\widetilde{U^I}_{\ell,m}(z)$ and $e^{-z\tau} \widetilde{W}_{\ell,m,0}(z)$ are the Laplace transforms of $\widehat{U^I}_{\ell,m}(t)$ and $\frac{d}{dt}\widehat{W}_{\ell,m,\tau}(t)$ respectively. 
	
	\indent
	Now solving equation \eqref{lap} for $\widetilde{U^I}_{\ell,m} (z)$, we obtain
	\begin{equation}\label{Lap2}
		\widetilde{U^I}_{\ell,m} (z) = e^{-z\tau}\frac{ \widetilde{W}_{\ell,m,0}(z) }{ \frac{1}{c^2}z^{2\alpha} +\frac{1}{\gamma}z^{\alpha}+  k^2\lambda_\ell}.
	\end{equation}
	\noindent
By using the convolution theorem (see \cite{Dyke}, Theorem 3.2) and Lemma \ref{pf:psi}, we could write   
	\[
	\calL^{-1}\Big\{ \frac{ \widetilde{W}_{\ell,m,0}(z) }{ \frac{1}{c^2}z^{2\alpha} +\frac{1}{\gamma}z^{\alpha}+  k^2\lambda_\ell}\Big\}= 	\int_{0}^{t} \psi_{\ell,\alpha}( t-s) \frac{d}{ds} \widehat{W}_{\ell,m,0}(s)ds,
	\]
where $\psi_{\ell,\alpha}( \cdot)$ is given by \eqref{def_psinew}.

\indent
By taking the inverse Laplace transform in \eqref{Lap2} with the help of the second shift theorem \cite[Theorem 2.4]{Dyke}, which results in
\begin{align}\label{InHom-Solmm}
		\widehat{U^I}_{\ell,m} (t) = \mathscr{H}(t-\tau)\int_{0}^{t-\tau} \psi_{\ell,\alpha}( t-\tau-s) \dfrac{d}{ds} \widehat{W}_{\ell,m,0}(s)Y_{\ell,m},
	\end{align}
where $\mathscr{H}(\cdot)$ is the Heaviside unit step function (i.e, $\mathscr{H}(t)=1$, for $t\geq0$, and zero for $t<0$).

\noindent	
Since $\widehat{W}_{\ell,m,0}(\cdot)$ are $1$-dimensional, real-valued Brownian motions which are continuous but nowhere differentiable, we have
	to write instead 
	\[
	\calL^{-1}\Big\{ \frac{ \widetilde{W}_{\ell,m,0}(z) }{ \frac{1}{c^2}z^{2\alpha} +\frac{1}{\gamma}z^{\alpha}+  k^2\lambda_\ell}\Big\}= 	\int_{0}^{t} \psi_{\ell,\alpha}( t-s) d \widehat{W}_{\ell,m,0}(s).
	\]
Hence the solution $U^I(t)$, $t\ge \tau$, becomes
 \begin{align}\label{InHom-Sol}
		U^{I}(t) = \mathscr{H}(t-\tau)\sum_{\ell=0}^\infty \sum_{m=-\ell}^{\ell}\int_{0}^{t-\tau} \psi_{\ell,\alpha}( t-\tau-s) d \widehat{W}_{\ell,m,0}(s)Y_{\ell,m}.
	\end{align}
\begin{ass}
We assume that the angular power spectrum $\{\calA_\ell\}$ satisfies the condition:
\begin{align}\label{seconAl}
    \sum_{\ell=0}^{\infty}
		(2\ell+1)\lambda_{\ell}\calA_{\ell}.
\end{align} 
\end{ass}
\begin{lem}\label{UICaucy}
  Let the condition \eqref{seconAl} hold. Let $U_{L}^{I}(t)$, $t\geq\tau$, be the truncated version of $U^I(t)$. Then $\{U_{L}^{I}(t)\}$, $\{{_0}\calJ_{t}^{\alpha}U_{L}^{I}(t)\}$ and $\{{_0}\calJ_{t}^{2\alpha}\Delta_{\bS^2}U_{L}^{I}(t)\}$, are Cauchy sequences in $L_2(\Omega\times\bS^2)$.
\end{lem}
\begin{proof}
Let $M>L$ be fixed, and consider $\{{_0}\calJ_{t}^{\alpha}U_{L}^{I}(t)\}$. By \eqref{InHom-Sol} and \eqref{fracint} with Remark \ref{frac_psi} we have
\begin{align*}
J_1&:=\bE \|{_0}\calJ_{t}^{\alpha} U^I_L(t) - {_0}\calJ_{t}^{\alpha} U^I_M(t) \|^2_{L_2(\bS^2)}\notag\\
&= \sum_{\ell=L+1}^M
		\sum_{m=-\ell}^{\ell} \calA_{\ell} \|Y_{\ell,m}\|^2_{L_2(\bS^2)}
		\int_{0}^{t-\tau}s^{4\alpha-2}
		\big|(E_{\alpha,2\alpha}\big(-z_{\ell}^{-}s^{\alpha}\big)-E_{\alpha,2\alpha}\big(-z_{\ell}^{+}s^{\alpha}\big))\big|^2 ds\notag\\
&\le C\sum_{\ell=L+1}^M (2\ell+1)\calA_{\ell}\int_{0}^{t-\tau} s^{4\alpha-2}ds= C \dfrac{(t-\tau)^{4\alpha-1}}{4\alpha-1}\sum_{\ell=L+1}^M
		(2\ell+1)\calA_{\ell},
\end{align*}
where the constant $C > 0$ and the third step used Proposition \ref{BoundE}.
Thus, for a given $\epsilon>0$, using the condition \eqref{eq:condAell}, there is $L_0$ independent of $t$ such that for $L,M  \ge L_0$, the above RHS is smaller than $\epsilon$. So
	$\{{_0}\calJ_{t}^{\alpha}U_{L}^{I}(t)\}$ is a Cauchy sequence in $L_2(\Omega \times \bS^2)$ and hence it
	is convergent. We define ${_0}\calJ_{t}^{\alpha}U^{I}(t)$ as $ \lim_{L\to\infty} {_0}\calJ_{t}^{\alpha}U_{L}^{I}(t)$, which can be written, using Remark \ref{frac_psi}, as 
 \begin{align*}
  \lim_{L\to\infty} {_0}\calJ_{t}^{\alpha}U_{L}^{I}(t)  =\sum_{\ell=0}^{\infty}\sum_{m=-\ell}^{\ell}\int_{0}^{t-\tau} \psi^{2\alpha}_{\ell,\alpha}(s) d \widehat{W}_{\ell,m,0}(s)Y_{\ell,m},
 \end{align*}
where $\psi^{b}_{\ell,\alpha}$ are defined by \eqref{psi_balpha}.

Now we consider $\{{_0}\calJ_{t}^{2\alpha}\Delta_{\bS^2}U_{L}^{I}(t)\}$. By Proposition \ref{BoundE} we obtain 
\begin{align*}
J_2:&=\bE \|{_0}\calJ_{t}^{2\alpha}\Delta_{\bS^2}U_{L}^{I}(t) - {_0}\calJ_{t}^{2\alpha}\Delta_{\bS^2}U_{M}^{I}(t) \|^2_{L_2(\bS^2)}\notag\\
&= \sum_{\ell=L+1}^M
		\sum_{m=-\ell}^{\ell} (\lambda_{\ell})^{2}\calA_{\ell} \|Y_{\ell,m}\|^2_{L_2(\bS^2)}
		\int_{0}^{t-\tau}s^{4\alpha-2}
		\big|(E_{\alpha,2\alpha}\big(-z_{\ell}^{-}s^{\alpha}\big)-E_{\alpha,2\alpha}\big(-z_{\ell}^{+}s^{\alpha}\big))\big|^2 ds\notag\\
&\le C\sum_{\ell=L+1}^M (2\ell+1)\lambda_{\ell}\calA_{\ell}\int_{0}^{t-\tau} s^{4\alpha-2}ds\leq C \dfrac{(t-\tau)^{4\alpha-1}}{4\alpha-1}\sum_{\ell=L+1}^M (2\ell+1)\lambda_\ell\calA_{\ell}.
\end{align*}

The condition \eqref{seconAl} guarantees that $\{{_0}\calJ_{t}^{2\alpha}\Delta_{\bS^2}U_{L}^{I}(t)\}$ is a Cauchy sequence in $L_2(\Omega \times \bS^2)$. We define ${_o}\calJ_{t}^{2\alpha}\Delta_{\bS^2}U^{I}(t)$ as $\lim_{L\to\infty} {_0}\calJ_{t}^{2\alpha}\Delta_{\bS^2}U_{L}^{I}(t)$. One can use the upper bound \eqref{Propsigma} to prove that $\{U_{L}^{I}(t)\}$ is a Cauchy sequence. Thus the proof is complete.
\end{proof}

\begin{lem}\label{lemVI}
   Let $V^{I}_{L}(t)$, $t\geq\tau$, be defined as
   \begin{align}\label{VI}
   V^{I}_{L}(t):=\frac{c^2}{\gamma} {_0}\calJ_{t}^{\alpha}U^I_L(t) - (ck)^2{_0}\calJ_{t}^{2\alpha}\Delta_{\bS^2} U^I_L(t),    
   \end{align}
where $U_L^I$ is the truncated field of $U^I$. Then, $V^{I}_{L}(t)$ can be written as
\begin{align}\label{eq:VL}
V^{I}_L(t)= -U^{I}_L(t)+\dfrac{c^2}{\Gamma(2\alpha)}\sum_{\ell=0}^{L}\sum_{m=-\ell}^{\ell}\int_{0}^{t-\tau} s^{2\alpha-1}d \widehat{W}_{\ell,m,0}(s)Y_{\ell,m}.    
\end{align}
\end{lem}
\begin{proof}
    By \eqref{frac_psi}, the first term of \eqref{VI} rewrites as
    \begin{align*}
      &S_1(t):= \frac{c^2}{\gamma} {_0}\calJ_{t}^{\alpha}U^I_L(t)\\
      &=\sum_{\ell=0}^{L}\sum_{m=-\ell}^{\ell}\dfrac{c^2}{M_\ell}\int_{0}^{t-\tau} s^{2\alpha-1}(E_{\alpha,2\alpha}\big(-z_{\ell}^{-}s^{\alpha}\big)-E_{\alpha,2\alpha}\big(-z_{\ell}^{+}s^{\alpha}\big))d \widehat{W}_{\ell,m,0}(s)Y_{\ell,m}\\
      &=  \sum_{\ell=0}^{L}\sum_{m=-\ell}^{\ell}\int_{0}^{t-\tau} q_{\ell}(s)d\widehat{W}_{\ell,m,0}(s)Y_{\ell,m},
    \end{align*}
where
\[
q_{\ell}(s):= \dfrac{c^2}{M_\ell}s^{2\alpha-1}(E_{\alpha,2\alpha}\big(-z_{\ell}^{-}s^{\alpha}\big)-E_{\alpha,2\alpha}\big(-z_{\ell}^{+}s^{\alpha}\big)).
\]
Now, consider the second term of \ref{VI}, we have 
\begin{align}\label{VVI}
    &S_2(t):=(ck)^2{_0}\calJ_{t}^{2\alpha}\Delta_{\bS^2} U^I_L(t) =
	\sum_{\ell=0}^{L}\sum_{m=-\ell}^{\ell}\dfrac{k^2c^2\gamma (-\lambda_\ell)}{M_\ell} Y_{\ell,m} \notag\\
	&\times \Bigg( \int_{0}^{t-\tau} s^{3\alpha-1} 
	\Big(E_{\alpha,3\alpha}\big(-z_{\ell}^{-}s^{\alpha}\big) - E_{\alpha,3\alpha}\big(-z_{\ell}^{+}s^{\alpha}\big) \Big) 
	d \widehat{W}_{\ell,m,0}(s) \Bigg).
\end{align}
Using  \eqref{Mitpro} we write \ref{VVI} as 
\begin{align*}
	&S_2(t)
	=:\sum_{\ell=0}^{L}\sum_{m=-\ell}^{\ell} \int_{0}^{t-\tau} 
	\Big( P_{\ell}^{(1)}(s) - \dfrac{c^2}{\Gamma(2\alpha)}s^{2\alpha-1} \Big) d \widehat{W}_{\ell,m,0}(s) Y_{\ell,m},
\end{align*}
where 
\begin{align*}
 P_{\ell}^{(1)}(s):=  \dfrac{\gamma}{M_\ell}
	s^{2\alpha-1} 
	\Big( z_{\ell}^{+} E_{\alpha,2\alpha}\big(-z_{\ell}^{-}s^{\alpha}\big) 
	- z_{\ell}^{-} E_{\alpha,2\alpha}\big(-z_{\ell}^{+}s^{\alpha}\big) \Big).
\end{align*}

Using \eqref{z1z2} and the relation \eqref{Mitpro} we write
\begin{align*}
P_{\ell}^{(1)}(s)&=\dfrac{\gamma}{M_\ell} s^{2\alpha-1}\Big(z_{\ell}^{+}E_{\alpha,2\alpha}\big(-z_{\ell}^{-}s^{\alpha}\big)-z_{\ell}^{-}E_{\alpha,2\alpha}\big(-z_{\ell}^{+}s^{\alpha}\big)\Big).
\end{align*}

Now using the relation \eqref{Mitpro} we write
\begin{align*}
   q_{\ell}(s)-P^{(1)}_{\ell}(s)&=\dfrac{s^{2\alpha-1}}{M_\ell}\dfrac{1}{2} c^{2}\Big((1-M_{\ell})E_{\alpha,2\alpha}\big(-z_{\ell}^{-}s^{\alpha}\big)-(1+M_{\ell})E_{\alpha,2\alpha}\big(-z_{\ell}^{+}s^{\alpha}\big)\Big)\\
&=\dfrac{\gamma}{M_\ell}s^{2\alpha-1}\Big(z_{\ell}^{-}E_{\alpha,2\alpha}\big(-z_{\ell}^{-}s^{\alpha}\big)-z_{\ell}^{+}E_{\alpha,2\alpha}\big(-z_{\ell}^{+}s^{\alpha}\big)\Big)\\
&=-\dfrac{\gamma}{M_\ell}s^{\alpha-1}\Big(E_{\alpha,\alpha}\big(-z_{\ell}^{-}s^{\alpha}\big)-E_{\alpha,\alpha}\big(-z_{\ell}^{+}s^{\alpha}\big)\Big)\\
&= -\psi_{\ell,\alpha}(s).
\end{align*}
Thus, the result follows since $ V^{I}_{L}=S_1-S_2$. 
\end{proof}

\begin{lem}\label{cachu-UI}
Let $W_{\tau}$ be an $L_2(\bS^2)$-valued time-delayed Brownian motion. Let $\{\calA_{\ell}:\ell\in\N_{0}\}$, the angular power spectrum of $W_\tau$, satisfy \eqref{seconAl}.
Then, as $L\to\infty$, $V_{L}^{I}$, defined by \eqref{eq:VL}, is convergent to
\begin{equation}\label{eq:defVI}
		V^{I}(t):= -U^{I}(t)+\dfrac{c^2}{\Gamma(2\alpha)}\sum_{\ell=0}^{\infty}\sum_{m=-\ell}^{\ell}\int_{0}^{t-\tau} s^{2\alpha-1}d \widehat{W}_{\ell,m,0}(s)Y_{\ell,m},
\end{equation}
in the sense that 
\[
\sup_{t \geq \tau} \bE \left\| V_L^I(t) - V^I(t) \right\|^2_{L_2(\bS^2)} 
	\to 0,\quad  \text{ as } L \to \infty.
\]
\end{lem}

\begin{proof}
    Note that by Lemma \ref{UICaucy} we have $\{U^I_L\}$ is a Cauchy sequence in $L_2(\Omega\times\bS^2)$. Now consider the term 
    \[
   W_{\tau}^{L}(t):=\sum_{\ell=0}^{L}\sum_{m=-\ell}^{\ell}\int_{0}^{t-\tau} s^{2\alpha-1}d \widehat{W}_{\ell,m,0}(s)Y_{\ell,m}
    \]
    then 
   for $t\geq \tau$, $L<M$, by Parseval's formula and It\^o's isometry, we have
    \begin{align}
	\bE \|W_{\tau}^{L}(t) -W_{\tau}^{M}(t) \|^2_{L_2(\bS^2)}
		&= \bE \left[\sum_{\ell=L+1}^M
		\sum_{m=-\ell}^{\ell} 	\int_{0}^{t-\tau} s^{2\alpha-1} d \widehat{W}_{\ell,m,0}(s)Y_{\ell,m}
		 \right]\nonumber\\
		&=\sum_{\ell=L+1}^M
		(2\ell+1) \calA_{\ell} \int_{0}^{t-\tau} s^{4\alpha-2}
		ds\nonumber \\
  &= \dfrac{(t-\tau)^{4\alpha-1}}{4\alpha-1}\sum_{\ell=L+1}^M (2\ell+1) \calA_{\ell}.
\end{align}
For a given $\epsilon>0$, using the condition \eqref{eq:condAell}, there is 
	$L_0$ independent of $t$ 
such that for $L,M  \ge L_0$, 
\begin{align}\label{cauchynois}
    \bE \|W_{\tau}^{L}(t) -W_{\tau}^{M}(t) \|^2_{L_2(\bS^2)}<\epsilon.
\end{align}
So $\{V^I\}$ is a Cauchy sequence in $L_2(\Omega \times \bS^2)$ and hence it
	is convergent. We define the limit of $V_L^I$ as $L \to \infty$ as in
	\eqref{eq:defVI}. 
Note that 
by the condition \eqref{eq:condAell} we obtain
 \begin{align}\label{VIconv}
   \bE \Big[\|V^I(t)\|^2_{L_2(\bS^2)}\Big]\le C (t-\tau)^{4\alpha-1}
\sum_{\ell=0}^{\infty} (2\ell+1) \calA_{\ell}<\infty,  
 \end{align}
which guarantees that $V^I$ is well-defined in the $L_2(\Omega\times\bS^2)$ sense. Thus the proof is complete.
\end{proof}

\begin{prop}\label{inhom_esta}
Let the angular power spectrum $\{\calA_{\ell}:\ell\in\N_{0}\}$ of $W_\tau$ satisfy assumption \eqref{seconAl}. Then the random field $U^{I}$, given by {\rm(\ref{InHom-Sol})}, satisfies the equation {\rm(\ref{Inhom})} under the conditions $U^{I}(t)= 0$ for all $t\in(0,\tau)$, and $\partial U^I/\partial t = 0$ at $t=0$. In particular, for $t\geq\tau$ there holds
    \[
\sup_{t \geq \tau} \bE \left\| \calK U^I(t)-\dfrac{c^2}{\Gamma(2\alpha)}\int_{\tau}^{t} (t-s)^{2\alpha-1}dW_{\tau}(s)\right\|^2_{L_2(\bS^2)} = 0.
\]
\end{prop}
\begin{proof}
By Lemma \ref{lemVI} we conclude that 
\[
\sup_{t \in[\tau,T]}\bE \left\| V^{I}_L(t)+U^{I}_L(t)-\dfrac{c^2}{\Gamma(2\alpha)}\sum_{\ell=0}^{L}\sum_{m=-\ell}^{\ell}\int_{0}^{t-\tau} s^{2\alpha-1}d \widehat{W}_{\ell,m,0}(s)Y_{\ell,m}\right\|^2_{L_2(\bS^2)}=0.
\]
Also, by Lemma \ref{cachu-UI} there holds
\[
\sup_{t \in[\tau,T]} \bE \left\| V^I(t)+ U^I(t)-\dfrac{c^2}{\Gamma(2\alpha)}\sum_{\ell=0}^{\infty}\sum_{m=-\ell}^{\ell}\int_{0}^{t-\tau} s^{2\alpha-1}d \widehat{W}_{\ell,m,0}(s)Y_{\ell,m}\right\|^2_{L_2(\bS^2)}=0
\]
and by \eqref{cauchynois}, $\{W_{\tau}^{L}\}$ is a Cauchy sequence in $L_2(\Omega \times \bS^2)$. Thus, 
the result follows since by Lemma \ref{UICaucy}, $U^I_L$ and $V^I_L$ are convergent to $U^I$ and $V^I$ respectively.   
\end{proof}

The following result shows that the inhomogeneous solution $U^{I}(t)$, $t\in(\tau,\infty)$, to the equation (\ref{Inhom}) is a centered, 2-weakly isotropic Gaussian random field.  The angular power spectrum $\{\calA_\ell\}$ is assumed to satisfy the condition 
\eqref{seconAl}. To guarantee that the condition is satisfied we assume that  
$\{\calA_{\ell}:\ell\in\N_{0}\}$ decays algebraically with order $\kappa_2> 4$, i.e., there exist constants $\widetilde{K},\widetilde{A}>0$ such that
\begin{align}\label{New-Al}
	\calA_{\ell}\leq\begin{cases} 
		\widetilde{K}, & \ell=0, \\
		\widetilde{A}\ell^{-\kappa_2}, & \ell\geq1,\ \kappa_2>4.
	\end{cases}
\end{align}
Let $\widetilde{A}_{\kappa_2}(\varkappa)$ be defined as
\begin{align}\label{Akappa}
	\widetilde{A}_{\kappa_2}(\varkappa):=\bigg(\widetilde{A}\Big(\dfrac{2\lfloor \varkappa\rfloor^{\frac{2}{\alpha}-\kappa_2}}{\kappa_2-\frac{2}{\alpha}}+\dfrac{\lfloor \varkappa\rfloor^{\frac{2}{\alpha}-\kappa_2-1}}{\kappa_2-\frac{2}{\alpha}+1}\Big)\bigg)^{1/2}.
\end{align}
Note that the term $\widetilde{A}_{\kappa_2}(\varkappa)$ is well-defined for all $\kappa_2>4$, as for $\alpha\in(0.5,1]$ we have $2\leq\frac{2}{\alpha}<4$.

\begin{prop}\label{covUI}
Let the field $U^{I}$, given in {\rm(\ref{InHom-Sol})}, be the solution to the equation {\rm(\ref{Inhom})} under the conditions $U^{I}(t)= 0$, $t\in[0,\tau]$, 
and $\partial U^I/\partial t =0$ at $t=0$. Let $\widehat{W}_{\ell,m,0}$ be the Fourier coefficients of $W_0$ with 
angular power spectrum $\{\calA_{\ell}:\ell\in\N_{0}\}$ which satisfies \eqref{New-Al}. Then, for $t\in(\tau,\infty)$, $U^{I}(t)$ is a centered, $2$-weakly isotropic Gaussian random field on $\bS^2$, and its random coefficients 
	\begin{align}\label{zetalm}
		\widehat{U^{I}}_{\ell,m}(t)=\int_{0}^{t-\tau} \psi_{\ell,\alpha}( t-\tau-s) d \widehat{W}_{\ell,m,0}(s),
	\end{align}
	satisfy for $\ell,\ell^\prime\in\N_{0}$, $m=-\ell,\dots,\ell$ and $m^\prime=-\ell^\prime,\dots,\ell^\prime$, 
	\begin{align}\label{var-zet}
		\bE \left[ \widehat{U^{I}}_{\ell,m}(t)\widehat{U^{I}}_{\ell^\prime,m^\prime}(t)\right]=\calA_{\ell}\sigma_{\ell,t-\tau,\alpha}^2\delta_{\ell\ell^\prime}\delta_{mm^\prime},
	\end{align}
where \( \psi_{\ell,\alpha} \) and \( \sigma_{\ell,t,\alpha}^2 \) are defined by \eqref{psi} and \eqref{var}, respectively, with \( \delta_{\ell\ell^\prime} \) as the Kronecker delta.
\end{prop}

\begin{proof}
Since $W_0$ is centered, we have $\bE[U^{I}(t)]=0$. Let $t\in(\tau,\infty)$ and $\bsx,\bsy\in\bS^2$, then by using \eqref{InHom-Sol} and \eqref{psi} we write
\begin{align*}
&\bE \left[U^{I}(\bsx,t)U^{I}(\bsy,t)\right]=\sum_{\ell=0}^\infty\sum_{\ell^\prime=0}^\infty \sum_{m=-\ell}^{\ell}
\sum_{m\prime=-\ell^\prime}^{\ell^\prime}Y_{\ell,m}(\bsx)Y_{\ell^\prime,m^\prime}(\bsy)\\ 
&\times\bE\left[\left(\int_{0}^{t-\tau} \psi_{\ell,\alpha}(t-\tau-s)d\widehat{W}_{\ell,m,0}(s)\right)\left(\int_{0}^{t-\tau}\psi_{\ell,\alpha}(t-\tau-v)d\widehat{W}_{\ell^\prime,m^{\prime},0}(v)\right)\right].
\end{align*}
Using It\^{o}’s isometry (see \cite{Oksendal2003}, Lemma 3.1.5), the addition theorem (see \eqref{addition}), and Proposition \ref{PropVar}, we can write
	\begin{align*}
		\bE\left[U^{I}(\bsx,t)U^{I}(\bsy,t)\right]&=\sum_{\ell=0}^\infty\calA_{\ell}\int_{0}^{t-\tau} \big|\psi_{\ell,\alpha}(t-\tau-s)\big|^2ds\sum_{m=-\ell}^{\ell} Y_{\ell,m}(\bsx)Y_{\ell,m}(\bsy)\\
		&=\sum_{\ell=0}^\infty (2\ell+1)\calA_\ell \sigma_{\ell,t-\tau,\alpha}^2 P_{\ell}(\bsx\cdot\bsy),
	\end{align*}
		where $P_{\ell}(\cdot)$, $\ell\in\N_{0}$, is the Legendre polynomial of degree $\ell$. As $P_{\ell}(\bsx\cdot\bsy)$ depends only on the inner product of $\bsx$ and $\bsy$, we conclude that the covariance function $\bE\left[U^{I}(\bsx,t)U^{I}(\bsy,t)\right]$ is rotationally invariant. Note that for $\ell,\ell^\prime\in\N_{0}$, $m=-\ell,\dots,\ell$ and $m^\prime=-\ell^\prime,\dots,\ell^\prime$, by the 2-weak isotropy of $W_0$ and It\^{o}’s isometry, 
			\begin{align*}
		\bE \left[ \widehat{U^{I}}_{\ell,m}(t)\widehat{U^{I}}_{\ell^\prime,m^\prime}(t)\right]=\calA_{\ell}\sigma_{\ell,t-\tau,\alpha}^2\delta_{\ell\ell^\prime}\delta_{mm^\prime}.
	\end{align*}
	To prove that $U^{I}(t)$ is Gaussian, first we assume that $\varkappa\notin \N$. Then, the variance of $U^{I}(t)$ can be written as	
\begin{align}\label{Ga-InHom}
Var\left[U^{I}(t)\right]&=\sum_{\ell=0}^\infty (2\ell+1)\calA_\ell \sigma_{\ell,t-\tau,\alpha}^2\notag\\
&= \sum_{\ell=0}^{\lfloor \varkappa\rfloor}(2\ell+1)\calA_\ell \sigma_{\ell,t-\tau,\alpha}^2+\sum_{\ell=\lfloor \varkappa\rfloor+1}^\infty(2\ell+1)\calA_\ell \sigma_{\ell,t-\tau,\alpha}^2.
\end{align}	

Using \eqref{sig0} and \eqref{up1}, we have $ \sum_{\ell=0}^{\lfloor \varkappa\rfloor}(2\ell+1)\calA_\ell \sigma_{\ell,t-\tau,\alpha}^2<\infty$, 
and the last term in \eqref{Ga-InHom} can be bounded, using \eqref{New-Al}, \eqref{sig0}, \eqref{Ml}, and \eqref{up2}, by 
\begin{align}\label{P3}
\sum_{\ell=\lfloor \varkappa\rfloor+1}^\infty(2\ell+1)\calA_\ell \sigma_{\ell,t-\tau,\alpha}^2&\leq \widetilde{A}C_3^I\sum_{\ell=\lfloor \varkappa\rfloor+1}^\infty(2\ell+1)\ell^{-\kappa_2} \lambda_\ell^{\frac{1-\alpha}{\alpha}}\notag\\
&\leq 2^{\frac{1-\alpha}{\alpha}}\widetilde{A}C_3^I\sum_{\ell=\lfloor \varkappa\rfloor+1}^\infty(2\ell+1)\ell^{-(\kappa_2+\frac{2\alpha-2}{\alpha})}\notag\\
&\leq  C_3^I(\widetilde{A}_{\kappa_2}(\varkappa))^2<\infty,
\end{align}	
where the last step uses steps similar to those for \eqref{Vpf}.

Now assume that $\varkappa\in\N$ we write
\begin{align*}
    Var\left[U^{I}(t)\right]&=\calA_0 \sigma_{0,t-\tau,\alpha}^2+ (2\varkappa+1)\calA_\varkappa \sigma_{\varkappa,t-\tau,\alpha}^2+\sum_{\substack{\ell=1 \\ \ell\neq \varkappa}}^{\infty}(2\ell+1)\calA_\ell \sigma_{\ell,t-\tau,\alpha}^2.
\end{align*}
Thus, $Var\left[U^{I}(t)\right]<\infty$, since by using steps similar to those for the case $\varkappa\notin \N$,
\begin{align*}
   \sum_{\substack{\ell=1 \\ \ell\neq \varkappa}}^{\infty}(2\ell+1)\calA_\ell \sigma_{\ell,t-\tau,\alpha}^2&=\sum_{\ell=1}^{\varkappa-1}(2\ell+1)\calA_\ell \sigma_{\ell,t-\tau,\alpha}^2+\sum_{\ell=\varkappa+1}^{\infty}(2\ell+1)\calA_\ell \sigma_{\ell,t-\tau,\alpha}^2<\infty. 
\end{align*}

By Remark 6.13 in \cite{MarPec11} we have $R_{\ell}(t):=\sum_{m=-\ell}^{\ell} \widehat{U^{I}}_{\ell,m}(t)Y_{\ell,m}$ is a centred, Gaussian random variable with mean zero and variance given, using \eqref{var-zet}, by  $(2\ell+1)\calA_{\ell}\sigma_{\ell,t-\tau,\alpha}^2$.
	Define $U^{I}_{n}(t):=\sum_{\ell=0}^{n}R_{\ell}(t)$, $n\geq1$, then by \eqref{Ga-InHom} and similar steps as in \eqref{var-Un} and \eqref{Lim}, with direct application of the continuity theorem \cite[Theorem 3.3.6]{Durrett},
	we conclude that the field $U^{I}(t)$ is Gaussian, thus completing the proof. 
\end{proof}

\begin{rem}
	Let $\{\calA_{\ell}:\ell\in\N_{0}\}$ be the angular power spectrum of $W_{0}(t)$. Then for $t>\tau$ the Fourier coefficients $ \widehat{U^I}_{\ell,m}(t)$ of the solution $U^{I}(t)$ can be written, using \eqref{InHom-Sol}, as
	\begin{align}\label{complexFourier}
		\widehat{U^I}_{\ell,m}(t)&:=\sqrt{\calA_{\ell}}
			\mathcal{I}_{\ell,|m|,\alpha}(t-\tau),
	\end{align}
	where $\mathcal{I}_{\ell,m,\alpha}(\cdot)$ are given in \eqref{Int}.

	\indent
Moreover, the solution of \eqref{Inhom} for $t\leq\tau$ is $U^I(t)=0$, while for $t>\tau$ it can be expressed, using \eqref{complexFourier}, as
\begin{align}\label{Sol2}
	U^{I}(t)=
	\sum_{\ell=0}^\infty
	\sum_{m=-\ell}^\ell \widehat{U^I}_{\ell,m}(t) Y_{\ell,m}=\sum_{\ell=0}^\infty
	\sum_{m=-\ell}^\ell \sqrt{\calA_{\ell}}
			\mathcal{I}_{\ell,|m|,\alpha}(t-\tau)Y_{\ell,m}.
\end{align}
\end{rem}

%
\subsection{The combined solution}\label{FullSec}
This subsection demonstrates the solution $U(t)$, $t\in[0,\infty)$, to the equation \eqref{Sys}. 

\begin{theo}\label{Theo3}
	The solution $U(t)$, $t\in[0,\infty)$, of equation \eqref{Sys} is 
	\begin{align}\label{Exact}
		U(t):=U^{H}(t)+U^{I}(t),\quad t\in[0,\infty),
	\end{align}
	where $U^{H}$ is given by \eqref{homo}
	and $U^{I}(t)=0$ for $t\leq\tau$, while for $t>\tau$, $U^{I}(t)$ is given by \eqref{InHom-Sol}.
\end{theo}

The following result shows that the solution $U$ to the equation (\ref{Sys}) is a centered, 2-weakly isotropic Gaussian random field.

\begin{prop}\label{Fulliso}
	Let the field $U(t)$, $t>0$, given by {\rm(\ref{Exact})}, be the solution to the equation {\rm(\ref{Sys})}. Then, $U$ is a centred, $2$-weakly isotropic Gaussian random field with random coefficients
	\begin{align}\label{gamlm}
		\widehat{U}_{\ell,m}(t):=
		\begin{cases}
			\widehat{U^{H}}_{\ell,m}(t),&t\leq\tau,\\
			\widehat{U^{H}}_{\ell,m}(t)+ \widehat{U^{I}}_{\ell,m}(t),&t>\tau,
		\end{cases}
	\end{align}
	where $ \widehat{U^{H}}_{\ell,m}(t)$ and $ \widehat{U^{I}}_{\ell,m}(t)$ are given by \eqref{xile} and \eqref{zetalm} respectively. 
	
	\indent
	Moreover, for $\ell,\ell^\prime\in\N_{0}$, $m=-\ell,\dots,\ell$ and $m^\prime=-\ell^\prime,\dots,\ell^\prime$, there holds
	\begin{align*}
		\bE \left[ \widehat{U}_{\ell,m}(t)\overline{ \widehat{U}_{\ell^\prime,m^\prime}(t)}\right]=\delta_{\ell\ell^\prime}\delta_{mm^\prime}
		\begin{cases}
			\calC_\ell \vert F_{\ell,\alpha}(t)\vert^2,&t\leq\tau,\\
			\calC_\ell \vert F_{\ell,\alpha}(t)\vert^2 +\calA_\ell\sigma_{\ell,t-\tau,\alpha}^2,&t>\tau.
		\end{cases}
	\end{align*}
\end{prop}
\begin{proof}
    	For $t\leq\tau$ the result follows by Proposition \ref{Gausshomo}. Let $t>\tau$, then by Propositions \ref{Gausshomo}, \ref{covUI}, and the uncorrelatedness of $U^H$ and $U^I$, we obtain
\begin{align*}
\bE\left[U(\bsx,t)U(\bsy,t)\right]&=\bE\left[U^{H}(\bsx,t)U^{H}(\bsy,t)\right]+\bE\left[U^{I}(\bsx,t)U^{I}(\bsy,t)\right]\\
&= \sum_{\ell=0}^\infty(2\ell+1)\left(\calC_\ell (F_{\ell,\alpha}(t)^2) +\calA_\ell\sigma_{\ell,t-\tau,\alpha}^2\right)P_{\ell}(\bsx\cdot\bsy).
\end{align*}
As $P_{\ell}(\bsx\cdot\bsy)$ depends only on the inner product of $\bsx$ and $\bsy$, we conclude that the covariance function $\bE\left[U^{I}(\bsx,t)U^{I}(\bsy,t)\right]$ is rotationally invariant. Since $t>\tau$,  we note that for $\ell,\ell^\prime\in\N_{0}$, $m=-\ell,\dots,\ell$ and $m^\prime=-\ell^\prime,\dots,\ell^\prime$, by Propositions \ref{Gausshomo} and \ref{covUI},
	\begin{align*}
		\bE \left[ \widehat{U}_{\ell,m}(t)\overline{ \widehat{U}_{\ell^\prime,m^\prime}(t)}\right]=\big(\calC_\ell \vert F_{\ell,\alpha}(t)\vert^2 +\calA_\ell\sigma_{\ell,t-\tau,\alpha}^2\big)\delta_{\ell\ell^\prime}\delta_{mm^\prime}.
	\end{align*}
Using the results of Propositions \ref{Gausshomo} and \ref{covUI} we conclude that the field $U$ is Gaussian, thus completing the proof.
\end{proof}

\begin{rem}\label{rem5}
	The solution \( U \), as given in \eqref{Exact}, can be expressed using \eqref{New-HomSol} and \eqref{Sol2} as  
	\begin{align*}
		U(t)= \sum_{\ell=0}^{\infty}\big(U^{H}_{\ell}(t)+U^{I}_{\ell}(t)\big),\quad t\in[0,\infty),
	\end{align*}
	where $U^{H}_{\ell}$ is defined as 
	\begin{align}\label{tru-HomSol}
		U^{H}_{\ell}(t)&=F_{\ell,\alpha}(t)\sqrt{\calC_{\ell}}\sum_{m=-\ell}^{\ell} Z_{\ell,m} Y_{\ell,m},
	\end{align}
	and $U^{I}_{\ell}(t)=0$ for $t\leq\tau$, while for $t>\tau$,
	\begin{align}\label{tru-inho}
		U^{I}_{\ell}(t)&= \sqrt{\calA_{\ell}}\sum_{m=-\ell}^{\ell}\mathcal{I}_{\ell,|m|,\alpha}(t-\tau)Y_{\ell,m},
	\end{align}
	where $\mathcal{I}_{\ell,m,\alpha}(\cdot)$ are given in \eqref{Int}. 
\end{rem}

\subsection{H\"{o}lder continuity of the solution}\label{SecHoder}
This section derives continuity properties of the stochastic solution $U(t)$ at given time $t$ with respect to the geodesic distance on $\bS^2$, i.e.  $\tilde{d}(\bsx,\bsy)$, $\bsx,\bsy\in\bS^2$.

Notice that by Proposition  \ref{Fulliso} the stochastic solution $U$ is known to be mean square continuous (see \cite{MarPec2013}). However, to obtain sample H\"{o}lder continuity of the stochastic solution $U$ we need stronger assumptions on the angular power spectra $\calC_\ell$ and $\calA_\ell$, $\ell\in\N_{0}$.

\begin{ass}\label{ass1}
The angular power spectra $\{\calC_\ell,\ \ell\in\N_0\}$ and $\{\calA_\ell,\ \ell\in\N_0\}$ satisfy \eqref{New-Cl} and \eqref{New-Al} with $\kappa_1> 2(1+\beta^*)$ and $\kappa_2>2(1+\frac{1}{\alpha}+\beta^*)$ for some $\beta^*\in(0,1]$.
\end{ass}

Note that by Assumption {\rm\ref{ass1}} and similar steps for those of \eqref{Homvar1}-\eqref{UHfinit} and \eqref{Ga-InHom}-\eqref{P3}, we can show that there exist constants $K^{(1)}_{\beta^*},K^{(2)}_{\beta^*}>0$ and such that
\begin{align}\label{kbeta1}
   K^{(1)}_{\beta^*}:= \sum_{\ell=1}^\infty(2\ell+1)\calC_\ell\vert F_{\ell,\alpha}(t)\vert^2\lambda_{\ell}^{\beta^*}<\infty  
\end{align}
and
\begin{align}
    K^{(2)}_{\beta^*}:=\sum_{\ell=1}^\infty (2\ell+1)\calA_\ell \sigma_{\ell,t-\tau,\alpha}^2\lambda_{\ell}^{\beta^*}<\infty.
\end{align}

\begin{lem}\label{holder}
	Let $U(t)$, $t\in[0,\infty)$, be the solution given by \eqref{Exact} to the equation \eqref{Sys}, and the angular power spectra $\{\calC_\ell\}$ and $\{\calA_\ell\}$ satisfy the Assumption {\rm\ref{ass1}}. Then there exists a constant $K_{\beta^*}>0$ such that
	\[
	Var[U(\bsx,t)-U(\bsy,t)]\leq K_{\beta^*}\; (\tilde{d}(\bsx,\bsy))^{2\beta^*},\quad \beta^*\in(0,1].
	\]
\end{lem}

\begin{proof}
Let $t\leq\tau$. Then by \eqref{Exact} and \eqref{gamlm} there holds 
\begin{align*}
\bE\left[U(\bsx,t)U(\bsy,t)\right]&=\bE\left[U^{H}(\bsx,t)U^{H}(\bsy,t)\right]\\
		&= \sum_{\ell=0}^\infty(2\ell+1)\calC_\ell \vert F_{\ell,\alpha}(t)\vert^2 P_{\ell}(\bsx\cdot\bsy),
	\end{align*}
where $P_{\ell}(\cdot)$, $\ell\in\N_{0}$, is the Legendre polynomial of degree $\ell$. 

	\indent
	Thus for $t\leq\tau$ there holds
	\begin{align}\label{holder-1}
		Var[U(\bsx,t)-U(\bsy,t)]&=Var[U^H(\bsx,t)]+Var[U^H(\bsy,t)]\notag\\
		&-2\bE[U^H(\bsx,t)U^H(\bsy,t)]\notag\\
		&=2\big(	Var[U^H(\bsx,t)]-\bE[U^H(\bsx,t)U^H(\bsy,t)]\big)\notag\\
		&=2\sum_{\ell=0}^\infty(2\ell+1)\calC_\ell \vert F_{\ell,\alpha}(t)\vert^2(1-P_{\ell}(\bsx\cdot\bsy)).
	\end{align} 
Since $P_{\ell}(1)=1$ for all $\ell\in\N_0$ and $P_{\ell}^{\prime}(x)\leq P_{\ell}^{\prime}(1)$ for $x\in[-1,1]$, there holds
	\begin{align}\label{Holder1}
		\vert 1-P_{\ell}(x)\vert=\Big\vert \int_{x}^{1}P_{\ell}^{\prime}(y)dy\Big\vert\leq \vert1-x\vert\Big(\frac{\ell(\ell+1)}{2}\Big),\quad x\in[-1,1].
	\end{align}
	Also, we observe that 
	\begin{align}\label{Holder2}
		\vert 1-P_{\ell}(x)\vert\leq 2,\quad x\in[-1,1].
	\end{align}
	Then, for some $\beta^*\in(0,1]$ we can write, using \eqref{Holder1} and \eqref{Holder2},
	\begin{align}\label{Leg}
		\vert 1-P_{\ell}(x)\vert&=\vert 1-P_{\ell}(x)\vert^{\beta^*} \vert 1-P_{\ell}(x)\vert^{1-\beta^*}\notag\\
		&\leq \Big(\vert1-x\vert\frac{\ell(\ell+1)}{2}\Big)^{\beta^*} 2^{1-\beta^*}\notag\\
		&\leq 2^{1-2\beta^*}\vert1-x\vert^{\beta^*}(\ell(\ell+1))^{\beta^*}.
	\end{align}
Using the relation
$
		\vert 1-\cos r\vert= \Big\vert \int_{0}^r \sin x dx\Big\vert\leq r \sin r
		\leq r \Big\vert \int_{0}^r \cos x dx\Big\vert\leq r^2
$
with \eqref{Leg}, \eqref{holder-1} becomes bounded by
	\begin{align*}
		Var[U(\bsx,t)-U(\bsy,t)]&\leq  2^{2-2\beta^*}\vert 1-\cos \tilde{d}(\bsx,\bsy)\vert^{\beta^*}\sum_{\ell=0}^\infty(2\ell+1)\calC_\ell\vert F_{\ell,\alpha}(t)\vert^2\lambda_{\ell}^{\beta^*}\\
		&\leq 2^{2-2\beta^*}(\tilde{d}(\bsx,\bsy))^{2\beta^*}\sum_{\ell=1}^\infty(2\ell+1)\calC_\ell\vert F_{\ell,\alpha}(t)\vert^2\lambda_{\ell}^{\beta^*}\\
		&\leq 2^{2-\beta^*}(\tilde{d}(\bsx,\bsy))^{2\beta^*}K^{(1)}_{\beta^*},
	\end{align*}
Similarly, for $t>\tau$, we have
	\begin{align*}
		\bE\left[U(\bsx,t)U(\bsy,t)\right]&= \sum_{\ell=0}^\infty(2\ell+1)\left(\calC_\ell \vert F_{\ell,\alpha}(t)\vert^2 +\calA_\ell\sigma_{\ell,t-\tau,\alpha}^2\right)P_{\ell}(\bsx\cdot\bsy)
	\end{align*}
	and by Assumption \ref{ass1} and the upper bound \eqref{Leg} we obtain
	\begin{align*}
		Var[U(\bsx,t)-U(\bsy,t)]&=2\sum_{\ell=0}^\infty(2\ell+1)\left(\calC_\ell \vert F_{\ell,\alpha}(t)\vert^2 +\calA_\ell\sigma_{\ell,t-\tau,\alpha}^2\right)\\
		&\times(1-P_{\ell}(\bsx\cdot\bsy))\\
  &\leq 2^{2-2\beta^*}(\tilde{d}(\bsx,\bsy))^{2\beta^*}\\
  &\times\Big(\sum_{\ell=0}^\infty(2\ell+1)\calC_\ell\vert F_{\ell,\alpha}(t)\vert^2\lambda_{\ell}^{\beta^*}+ \sum_{\ell=0}^\infty(2\ell+1)\calA_\ell\sigma_{\ell,t-\tau,\alpha}^2\lambda_{\ell}^{\beta^*}\Big)\\
		&\leq 2^{2-2\beta^*}(\tilde{d}(\bsx,\bsy))^{2\beta^*}\big(K^{(1)}_{\beta^*}+K^{(2)}_{\beta^*}\big).
	\end{align*} 
	\indent
Thus for $t\geq0$, there exists a constant $K_{\beta^*}$ such that  
	\[
	Var[U(\bsx,t)-U(\bsy,t)]\leq K_{\beta^*}\; (\tilde{d}(\bsx,\bsy))^{2\beta^*},
	\]
	where
	\begin{align*}
		K_{\beta^*}:= 2^{2-2\beta^*}
		\begin{cases} 
			K^{(1)}_{\beta^*}, & t\leq\tau, \\
			K^{(1)}_{\beta^*}+K^{(2)}_{\beta^*}, & t>\tau,
		\end{cases}
	\end{align*}
	which completes the proof.
\end{proof}

Given our observations, we are now ready to prove the main result of this section.
\begin{theo}\label{modification}
    Let $U(t)$, $t\in[0,\infty)$, be the solution given by \eqref{Exact} to the equation \eqref{Sys}, and the angular power spectra $\{\calC_\ell\}$ and $\{\calA_\ell\}$ satisfy the Assumption {\rm\ref{ass1}}, then there exists a continuous modification of $U$ that is H\"{o}lder continuous with exponent $\gamma^* \in (0, \beta^*)$.
\end{theo}
\begin{proof}
Under the assumptions of Lemma \ref{holder}, applying Kolmogorov’s continuity criterion (see  \cite[Corollary 4.5]{Lanetal16}), for any fixed $t>0$, we conclude that there exists a locally H\"{o}lder continuous modification for $U$  of order $\gamma^{*}\in(0,\beta^*)$. Thus the proof is complete.
\end{proof}

\section{Approximation to the solution}\label{SecApprox}
This section provides an approximation to the solution $U$ by truncating its expansion at a degree 
(truncation level) $L\in\N$. Then we give an upper bound for the approximation error of the truncated solution $U_{L}$.

\begin{defin}
	The approximation $U_{L}$ of degree $L\in\N$ to the solution $U$ given in {\rm(\ref{Exact})} is defined as
	\begin{align}\label{Approx}
		U_{L}(t):&= U_L^H(t)+U_L^I(t)\notag\\
		&= \sum_{\ell=0}^{L}\Big(U^{H}_{\ell}(t)+U^{I}_{\ell}(t)\Big),\quad t\in[0,\infty),
	\end{align}
	where $U^{H}_{\ell}(t)$ is given by \eqref{tru-HomSol}
	and  for $t\leq\tau$, $U^{I}_{\ell}(t)=0$, while for $t>\tau$, $U^{I}_{\ell}(t)$ is given by \eqref{tru-inho}.
\end{defin}

\begin{prop}\label{theo3}
Let $\{\calC_{\ell}:\ell\in\N_{0}\}$, the angular power spectrum of $ \xi$, satisfy \eqref{New-Cl}. Let $Q^{H}_{L}(t)$, $t\in[0,\infty)$, be defined as
		\begin{align*}
		Q^{H}_{L}(t):&=\Big\Vert U^H(t)-U_L^H(t) \Big\Vert_{L_{2}(\Omega\times\bS^2)}=	\Big\Vert \sum_{\ell=L+1}^{\infty}U^{H}_{\ell}(t)\Big\Vert_{L_{2}(\Omega\times\bS^2)},
	\end{align*} 
where $U^{H}_{\ell}$ is given by \eqref{tru-HomSol}. Then, for $L>\max\{\varkappa,\ell_0\},\ \ell_0\in\N$ we have
	\[
	Q^{H}_{L}(t)\leq \mathcal{H}(t) \widetilde{C}_{\kappa_1}(1) L^{-\kappa_1/2},\quad t\geq0,
	\]
	where $\widetilde{C}_{\kappa_1}(\cdot)$ is given by \eqref{Ckappa} and $\mathcal{H}(\cdot)$ defined by \eqref{updf1}. 
\end{prop}

\begin{proof}
By \eqref{tru-HomSol} we can write
	\begin{align*}
		(Q^{H}_{L}(t))^2&=\bE\Bigg[\Big\Vert\sum_{\ell=L+1}^{\infty}\Big[ \sqrt{\calC_{\ell}} F_{\ell,\alpha}(t)\sum_{m=-\ell}^{\ell}Z_{\ell,m} Y_{\ell,m}\Big]\Big\Vert_{L_{2}(\bS^2)}^2\Bigg].
	\end{align*}
	Since $\{Z_{\ell,m},\ \ell\in\N_0,\ m=-\ell,\dots,\ell \}$ is a set of independent, real-valued, standard normally distributed random variables, we have
\begin{align*}
(Q^{H}_{L}(t))^2&=\sum_{\ell=L+1}^{\infty}\vert F_{\ell,\alpha}(t)\vert^2\calC_{\ell}\sum_{m=-\ell}^{\ell}\bE \left[\big(Z_{\ell,m}\big)^2\right]\Vert Y_{\ell,m}\Vert_{L_{2}(\bS^2)}^2\notag\\
&=\sum_{\ell=L+1}^{\infty}\vert F_{\ell,\alpha}(t)\vert^2\calC_{\ell}(2\ell+1),
\end{align*}
where the last step used the property
$
\sum_{m=-\ell}^{\ell}\Vert Y_{\ell,m}\Vert_{L_{2}(\bS^2)}^2=(2\ell+1)P_{\ell}(1)=(2\ell+1).
$

Using \eqref{F_upper} we obtain 
	\begin{align}\label{U1P_1}
		(Q^{H}_{L}(t))^2&=\sum_{\ell=L+1}^{\infty}\vert F_{\ell,\alpha}(t)\vert^2 (2\ell+1)\calC_{\ell}\notag\\
		&\leq \sum_{\ell=L+1}^{\infty}(2\ell+1)\dfrac{\vert \mathcal{H}(t)\vert^2}{\lambda_{\ell}}\calC_{\ell}\notag\\
		&\leq \dfrac{\vert \mathcal{H}(t)\vert^2}{\lambda_{L}}\sum_{\ell=L+1}^{\infty}(2\ell+1)\calC_{\ell}.
	\end{align}
Using \eqref{New-Cl} then \eqref{U1P_1} becomes bounded by
\begin{align}\label{N-up}
		(Q^{H}_{L}(t))^2&\leq \vert \mathcal{H}(t)\vert^2 L^{-2} \sum_{\ell=L+1}^{\infty} (2\ell+1)\calC_{\ell}\notag\\
		&\leq \vert \mathcal{H}(t)\vert^2 L^{-2}\widetilde{C}\int_{L}^{\infty}(2x+1)x^{-\kappa_1}dx\notag\\
		&= \vert \mathcal{H}(t)\vert^2 L^{-2}\widetilde{C}\bigg(\dfrac{2}{\kappa_1-2}+\dfrac{L^{-1}}{\kappa_1-1}\bigg)L^{2-\kappa_1}\notag\\
  &\le \vert \mathcal{H}(t)\vert^2 (\widetilde{C}_{\kappa_1}(1))^2L^{-\kappa_1},
	\end{align}
	which completes the proof.
\end{proof}

The following proposition provides an upper bound for the approximation errors of the truncated solution $U_{L}^I(t)$, $t>\tau$.
\begin{prop}\label{theo4}
	Let $\{\calA_{\ell}:\ell\in\N_{0}\}$, the angular power spectrum of $W_\tau$, satisfy \eqref{New-Al}. Let $Q^{I}_{L}(t)$, $t>\tau$, be defined as
	\begin{align}
		Q^{I}_{L}(t):&=\Big\Vert U^I(t)-U_L^I(t) \Big\Vert_{L_{2}(\Omega\times\bS^2)}=	\Big\Vert \sum_{\ell=L+1}^{\infty}U^{I}_{\ell}(t)\Big\Vert_{L_{2}(\Omega\times\bS^2)},
	\end{align}
	where for $t>\tau$, $U^{I}_{\ell}(t)$ is given by \eqref{tru-inho}. Then, for $L>\max\{\varkappa,\ell_0\},\ \ell_0\in\N$, we have 
	\[
Q^{I}_{L}(t)\leq	e_{\kappa_2} L^{-(\kappa_2-\frac{2}{\alpha})/2},
	\]
where $e_{\kappa_2}=2^{\frac{1-\alpha}{2\alpha}}\sqrt{C_3^I} \widetilde{A}_{\kappa_2}(1)$.
\end{prop}
\begin{proof}
Using \eqref{tru-inho} then $(Q^{I}_{L}(t))^2$ can be written for $t>\tau$ as
\begin{align*}
&(Q^{I}_{L}(t))^2=\Big\Vert \sum_{\ell=L+1}^{\infty}U_{\ell}^{I}(t)\Big\Vert_{L_{2}(\Omega\times\bS^2)}^2=\Big\Vert  \sum_{\ell=L+1}^{\infty}\sum_{m=-\ell}^{\ell}\sqrt{\calA_{\ell}}\mathcal{I}_{\ell,|m|,\alpha}(t-\tau)Y_{\ell,m}\Big\Vert_{L_{2}(\Omega\times\bS^2)}^2.
\end{align*}
By Proposition \ref{PropVar}, there holds
\begin{align}\label{upQ}
		(Q^{I}_{L}(t))^2&=\sum_{\ell=L+1}^{\infty}\sum_{m=-\ell}^{\ell}\calA_{\ell}\bE \Big[\big(\calI_{\ell,|m|,\alpha}(t-\tau)\big)^2\Big]\Vert  Y_{\ell,m}\Vert_{L_{2}(\bS^2)}^2\notag\\
		&=\sum_{\ell=L+1}^{\infty}\sigma_{\ell,t-\tau,\alpha}^2\calA_{\ell}\sum_{m=-\ell}^{\ell}\Vert Y_{\ell,m}\Vert_{L_{2}(\bS^2)}^2	\notag\\
		&= \sum_{\ell=L+1}^{\infty} \calA_{\ell}(2\ell+1)\sigma_{\ell,t-\tau,\alpha}^2\notag\\
	&\leq C_3^I \sum_{\ell=L+1}^{\infty} \calA_{\ell}(2\ell+1)\lambda_\ell^{\frac{1-\alpha}{\alpha}},
	\end{align}
	where the last step uses \eqref{up2}.
	
Now similar to \eqref{N-up}, the upper bound \eqref{upQ} can be bounded by
	\begin{align*}
	(Q^{I}_{L}(t))^2&\leq 2^{\frac{1-\alpha}{\alpha}}C_3^I\sum_{\ell=L+1}^{\infty} \calA_{\ell}(2\ell+1)\ell^{\frac{2-2\alpha}{\alpha}}
	\leq 2^{\frac{1-\alpha}{\alpha}}C_3^I (\widetilde{A}_{\kappa_2}(1))^2 L^{-(\kappa_2-\frac{2}{\alpha})},
\end{align*}
which completes the proof.
\end{proof}

The following theorem can be established by combining the results from Propositions \ref{theo3} and \ref{theo4}. 
\begin{theo}\label{The5}
	Let $U(t)$, $t\in[0,\infty)$, be the solution \eqref{Exact} to the equation \eqref{Sys}, and the conditions of Propositions {\rm\ref{theo3}} and {\rm\ref{theo4}} be satisfied. Let $\tilde{\kappa}_\alpha:=\min(\kappa_1,\kappa_2-\frac{2}{\alpha})$. Let $Q_L(t)$, $t\in[0,\infty)$, be defined as
	\begin{align}\label{fullq}
		Q_L(t):=\Big\Vert U(t)- U_{L}(t)\Big\Vert_{L_{2}(\Omega\times\bS^2)},\quad L\geq1,  
	\end{align}
	where $U_{L}$ is given by \eqref{Approx}. 
	Then, for $L>\max\{\varkappa,\ell_0\},\ \ell_0\in\N$, the following estimate holds true:
 		\begin{align*}
			Q_L(t)\leq \left((\mathcal{H}(t)\widetilde{C}_{\kappa_1}(1))^2+e_{\kappa_2}^2\right)^{\frac{1}{2}}L^{-\tilde{\kappa}_\alpha/2}.
		\end{align*}
\end{theo}

\section{Numerical studies}\label{Num}
In this section, we present numerical results to examine the evolution of the solution \( U \) for equation \eqref{Sys}. In particular, we provide several numerical examples that showcase the time-dependent behavior of the stochastic solution \( U(t) \) using simulated data inspired by the CMB map. Furthermore, we investigate the convergence rates of the truncation errors associated with the approximations \( U_{L} \). Finally, we investigate the sample H\"{o}lder property using a numerical example.

\subsection{Evolution of the solution}\label{Evol}
In this subsection, we examine the evolution of the stochastic solution $U(t)$ for equation \eqref{Sys} by utilizing simulated data inspired by the CMB map.

For the numerical results we consider $c = 1$, $\gamma = 1$, and $k = 0.05$. Also, as we deal with truncated versions of the random fields, it is not necessary that the angular power spectra $\{\calC_\ell, \ell \in \N_0\}$ and $\{\calA_\ell, \ell \in \N_0\}$ of the initial field $\xi$ and the noise $W_\tau$ strictly satisfy the conditions \eqref{eq:condCell} and \eqref{seconAl}. However, assuming that both $\kappa_1,\kappa_2>2$ is sufficient.
 
In Figure~\ref{fig:Moho cases}, we present two realizations of the truncated initial field with degree $L=600$ (i.e., $\xi=U_{600}(0)$) and different values of $\kappa_1$. The angular power spectrum $\{\calC_{\ell}\}$ of these realizations follow the form described in equation \eqref{New-Cl} with $\widetilde{D}=\widetilde{C}=1$ and $\kappa_1\in\{2.3,4.1\}$. These realizations were obtained using the Python HEALPy package \cite{Gorski2005}.

\begin{figure}[ht]
		\centering
		\subfloat[\centering Initial random field with $\kappa_1 = 2.3$, $L=600$. \label{initial_1}]{{\includegraphics[width=0.45\textwidth]{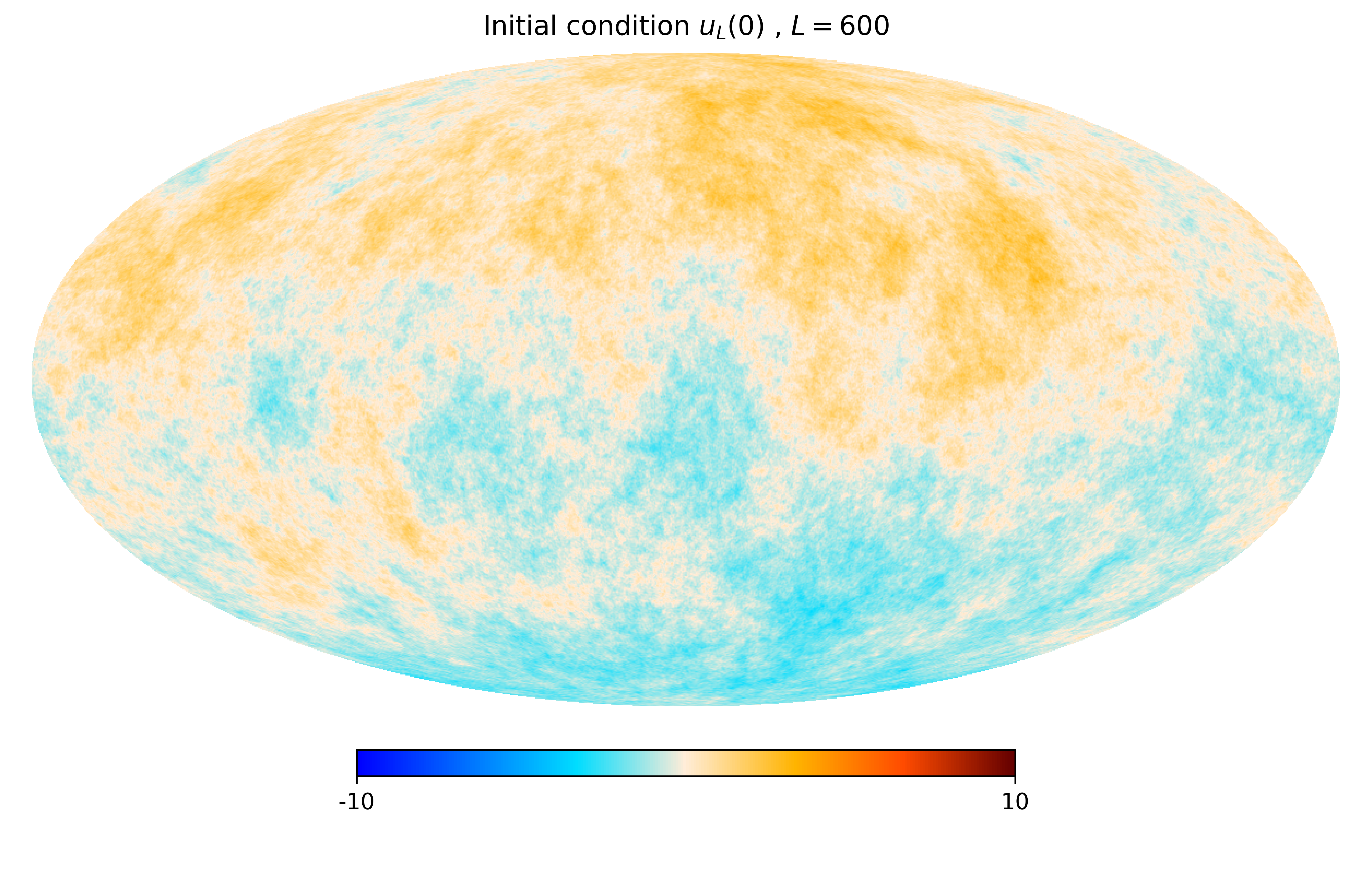} }}%
		\qquad
		\subfloat[\centering Initial random field with $\kappa_1 = 4.1$, $L=600$.\label{initial_2}]{{\includegraphics[width=0.45\textwidth]{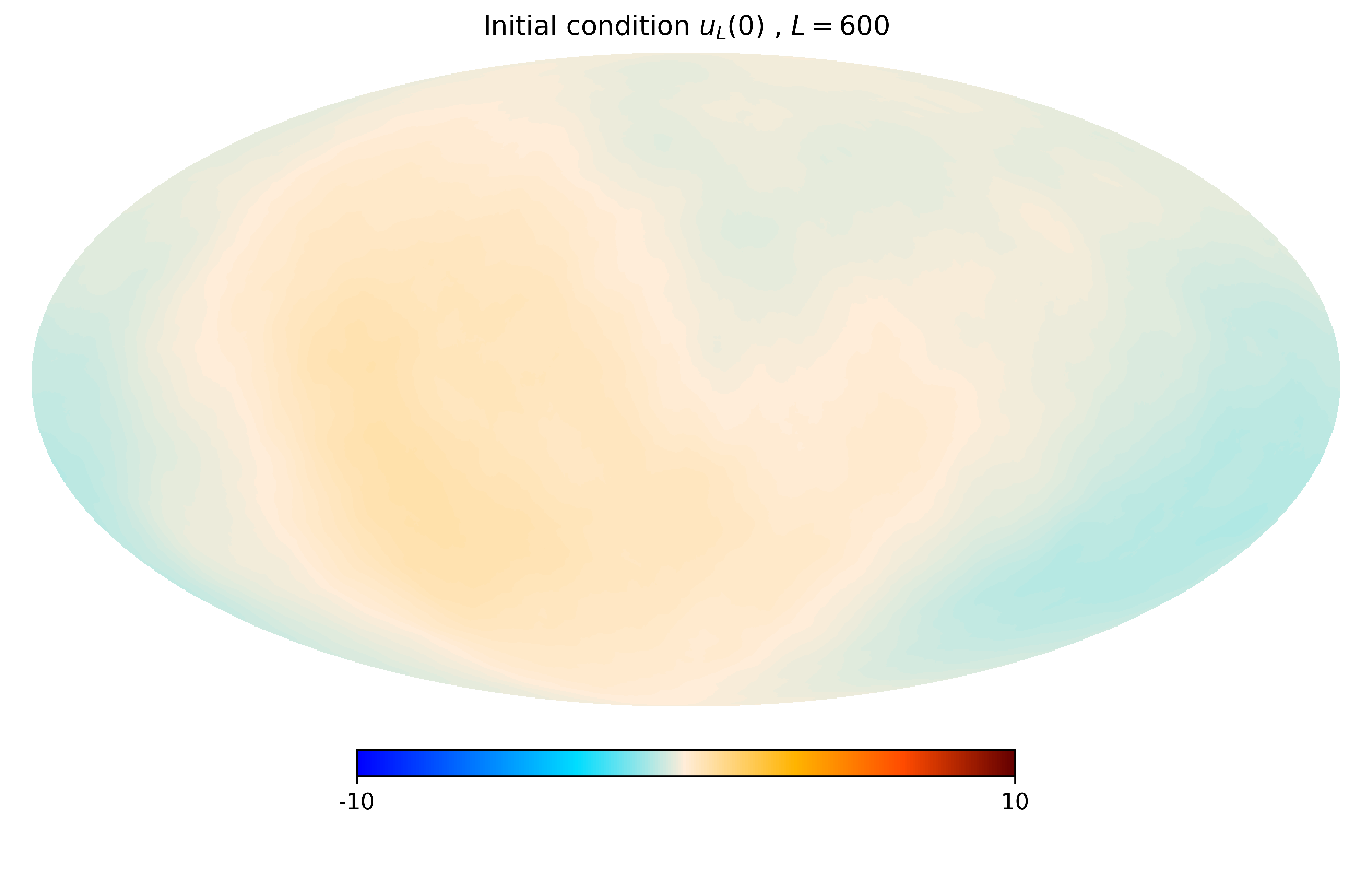} }}%
		\caption{Realizations of $U_{600}(0)$ for two different parameter values.}%
		\label{fig:Moho cases}%
	\end{figure}

Figure \ref{fig:Moho cases2} shows two different realizations of the truncated homogeneous random field at time $t=10\tau$ with $\tau=4\times10^{-2}$ and $\kappa_1\in\{2.3,4.1\}$ using the initial realizations obtained in Figure \ref{fig:Moho cases}. 
The corresponding Fourier coefficients $\widehat{U^H}_{\ell,m}(t)$ for $U_L^H$ are computed, using \eqref{UH-complexFourier}, i.e., for $m=-\ell,\dots,\ell$,
	\begin{align*}
		\widehat{U^H}_{\ell,m}(t)&= \sqrt{\calC_{\ell}}F_{\ell,\alpha}(t)Z_{\ell,m}.
	\end{align*}

\begin{figure}[ht]
		\centering
		\subfloat[\centering $U^H_{600}(10\tau)$ with $\tau=4\times10^{-2}$ and $\kappa_1 = 2.3$. \label{HOM1}]{{\includegraphics[width=0.45\textwidth]{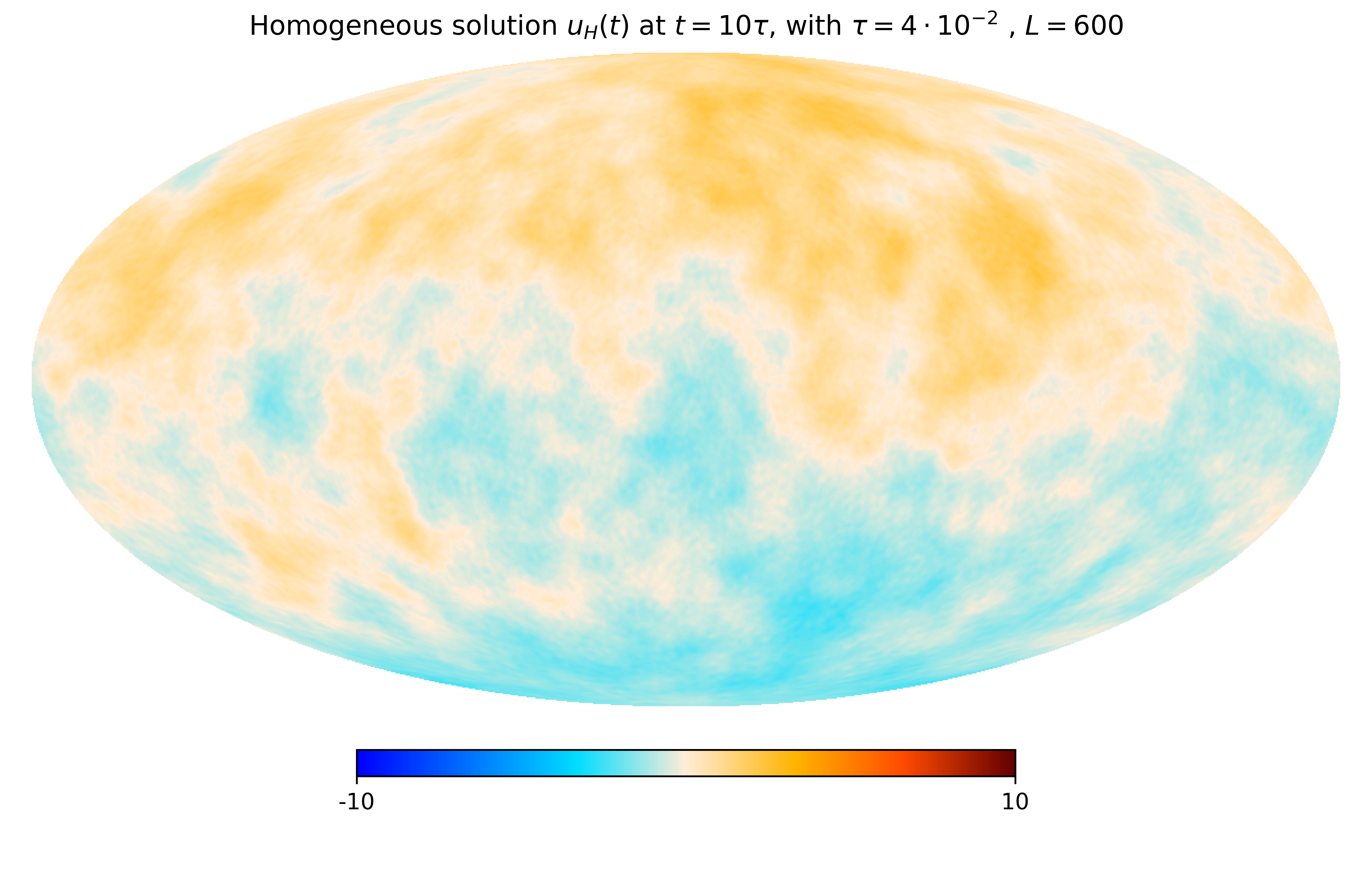} }}%
		\qquad
		\subfloat[\centering $U^H_{600}(10\tau)$ with $\tau=4\times10^{-2}$ and $\kappa_1 = 4.1$.\label{HOM2}]{{\includegraphics[width=0.45\textwidth]{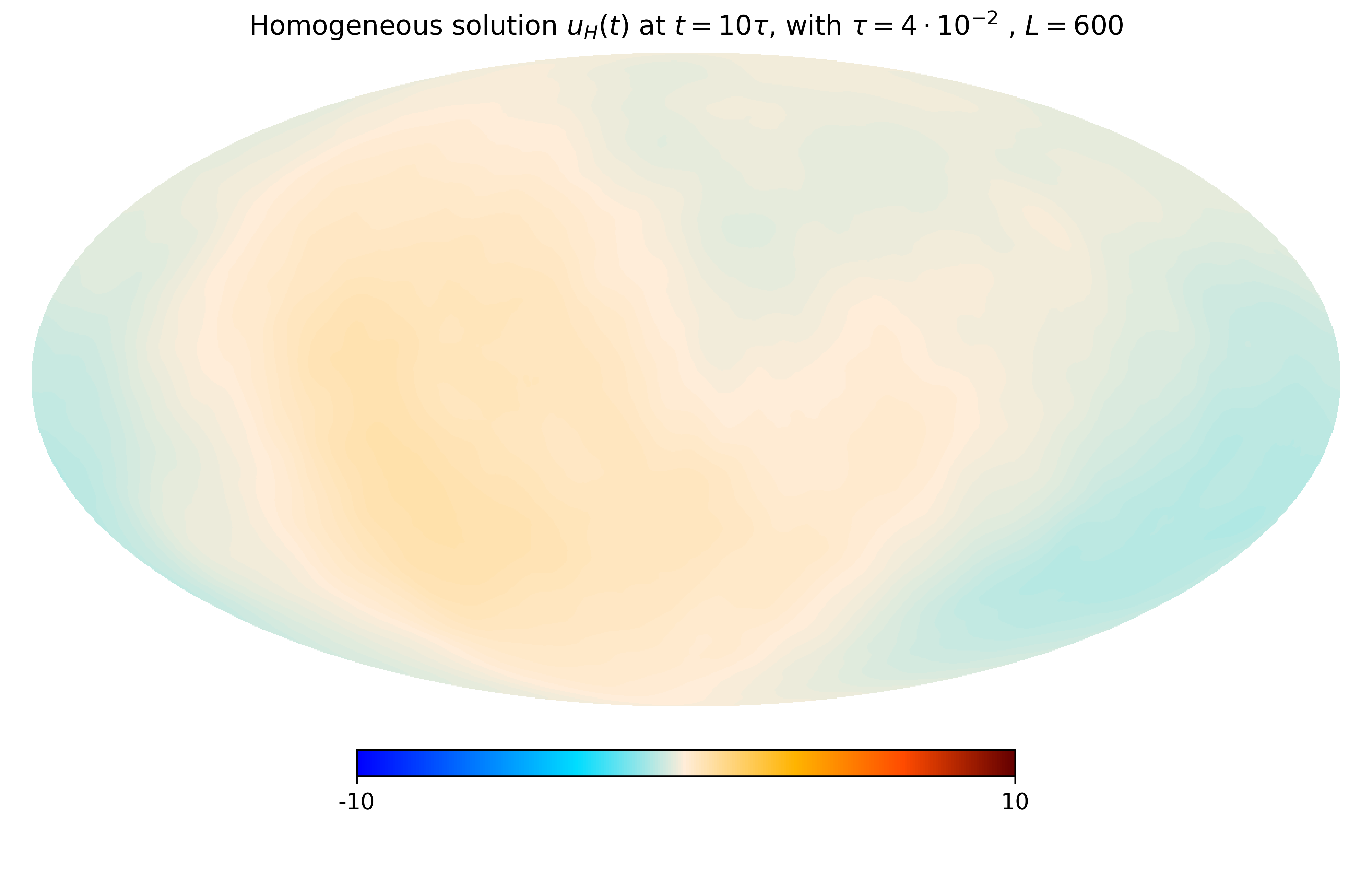} }}%
		\caption{ Truncated homogeneous solutions $U^H_{600}(10\tau)$, using the initial realizations in Figure \ref{fig:Moho cases}, with $\alpha=0.9$ and different values of $\kappa_1$.}%
		\label{fig:Moho cases2}%
	\end{figure}

In Figure \ref{fig:U400 at 10tau} we generate two different realizations using the truncated inhomogeneous field. The angular power spectrum $\{\calA_{\ell}\}$ of these realizations follows the form given by equation \eqref{New-Al} with the values $\widetilde{K}=1$, $\widetilde{A}=10$, and $\kappa_2=2.3$.
The corresponding Fourier coefficients $\widehat{U^I}_{\ell,m}(t)$ for the solution $U_L^I$ are computed, using \eqref{complexFourier}, i.e., for $m=-\ell,\dots,\ell$, 
\begin{align*}
		\widehat{U^I}_{\ell,m}(t)&=\sqrt{\calA_{\ell}}
			\mathcal{I}_{\ell,|m|,\alpha}(t-\tau),
	\end{align*}
where for each realization and given time $t>\tau$, we simulate the stochastic integrals $\calI_{\ell,|m|,\alpha}(t-\tau)$, using Proposition \ref{PropVar}, as $\mathcal{I}_{\ell,m,\alpha}(t)\sim \mathcal{N}(0,\sigma_{\ell,t,\alpha}^2)$ where $\sigma_{\ell,t,\alpha}^2$ is given by \eqref{var}.

\begin{figure}[ht]
		\centering
		\subfloat[\centering $U_{600}^{I}(10\tau)$ with $\tau=4\times10^{-2}$. \label{inHOM1}]{{\includegraphics[width=0.45\textwidth]{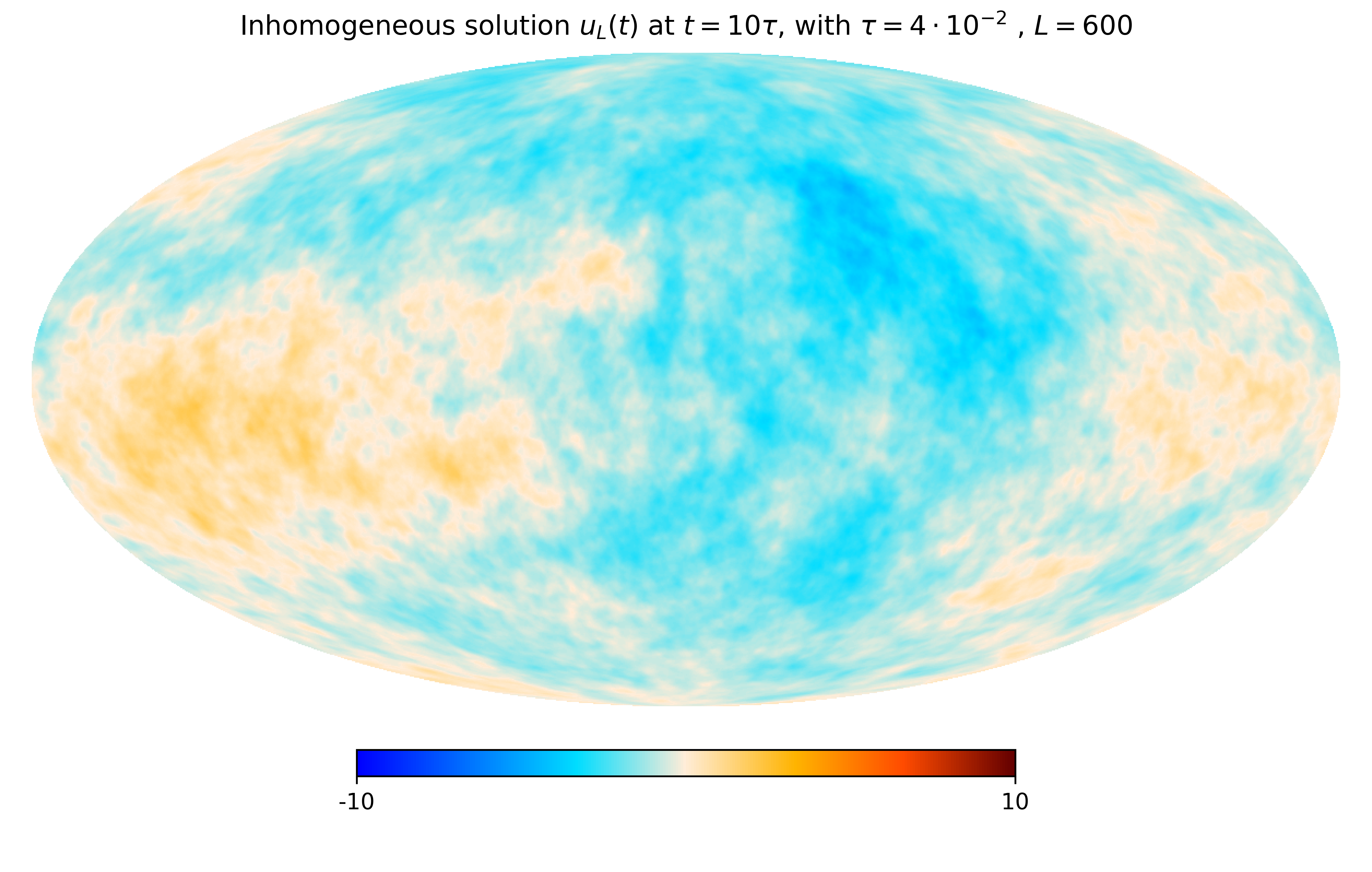} }}%
		\qquad
		\subfloat[\centering $U^I_{600}(10\tau)$ with $\tau=4\times10^{-2}$.\label{fullHOM2}]{{\includegraphics[width=0.45\textwidth]{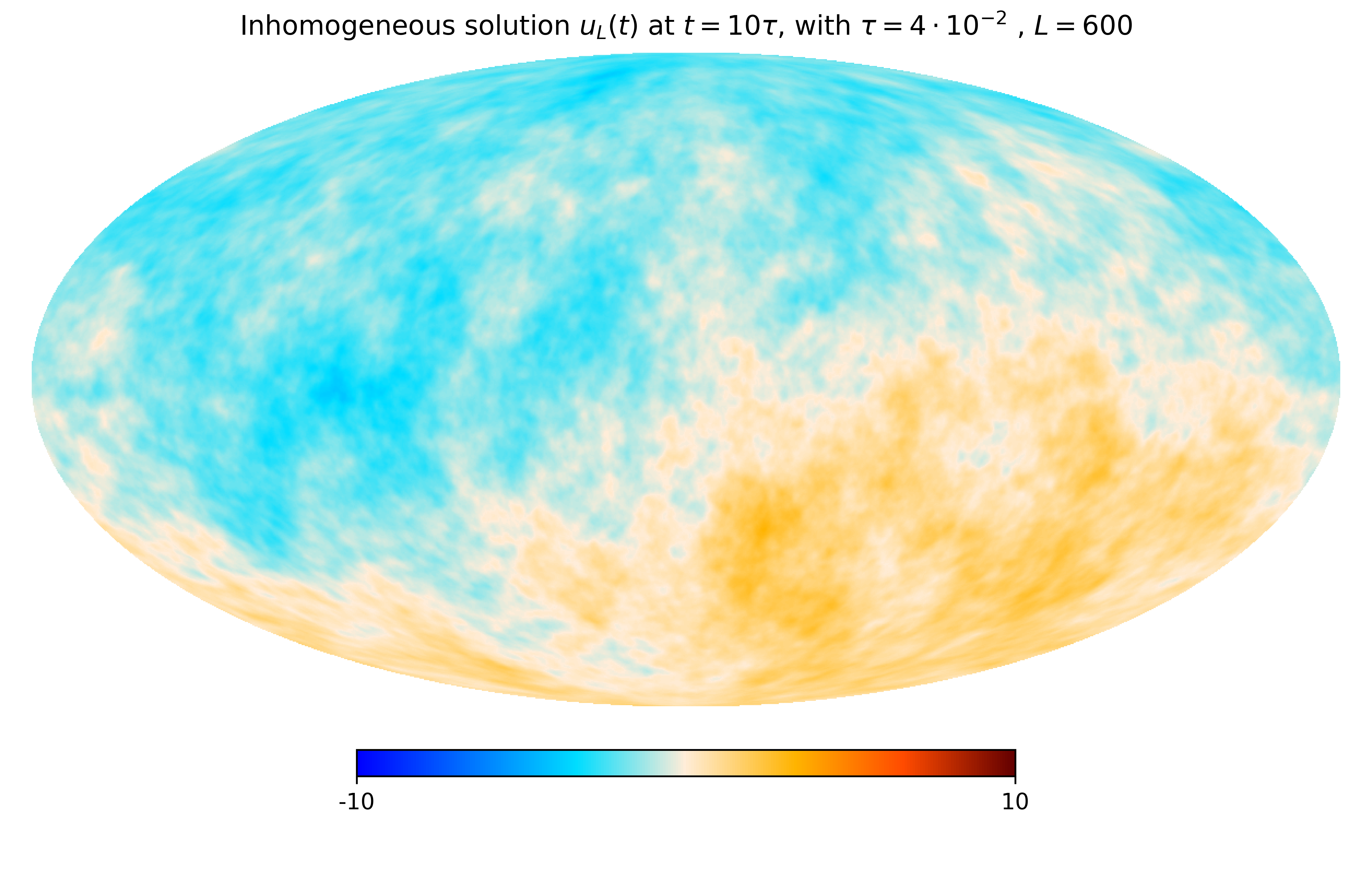} }}%
		\caption{ Realizations of the truncated inhomogeneous solution with $\alpha=0.9$ and $\kappa_2=2.5$.}%
		\label{fig:U400 at 10tau}%
	\end{figure}

Figure \ref{fig:U600 at 10tau} illustrates the evolution of the realizations shown in Figure \ref{fig:U400 at 10tau} at \( t=10\tau \) with the noise \( W_\tau \) included. In this scenario, the randomness arises from both the initial condition and \( W_\tau \), resulting in the solution \( U_{600}(10\tau) = U_{600}^H(10\tau) + U_{600}^I( 10\tau) \).
\begin{figure}[ht]
		\centering
		\subfloat[\centering $U_{600}(10\tau)$ with $\kappa_1=2.3$ and $\kappa_2=2.5$. \label{inHOM1wave}]{{\includegraphics[width=0.45\textwidth]{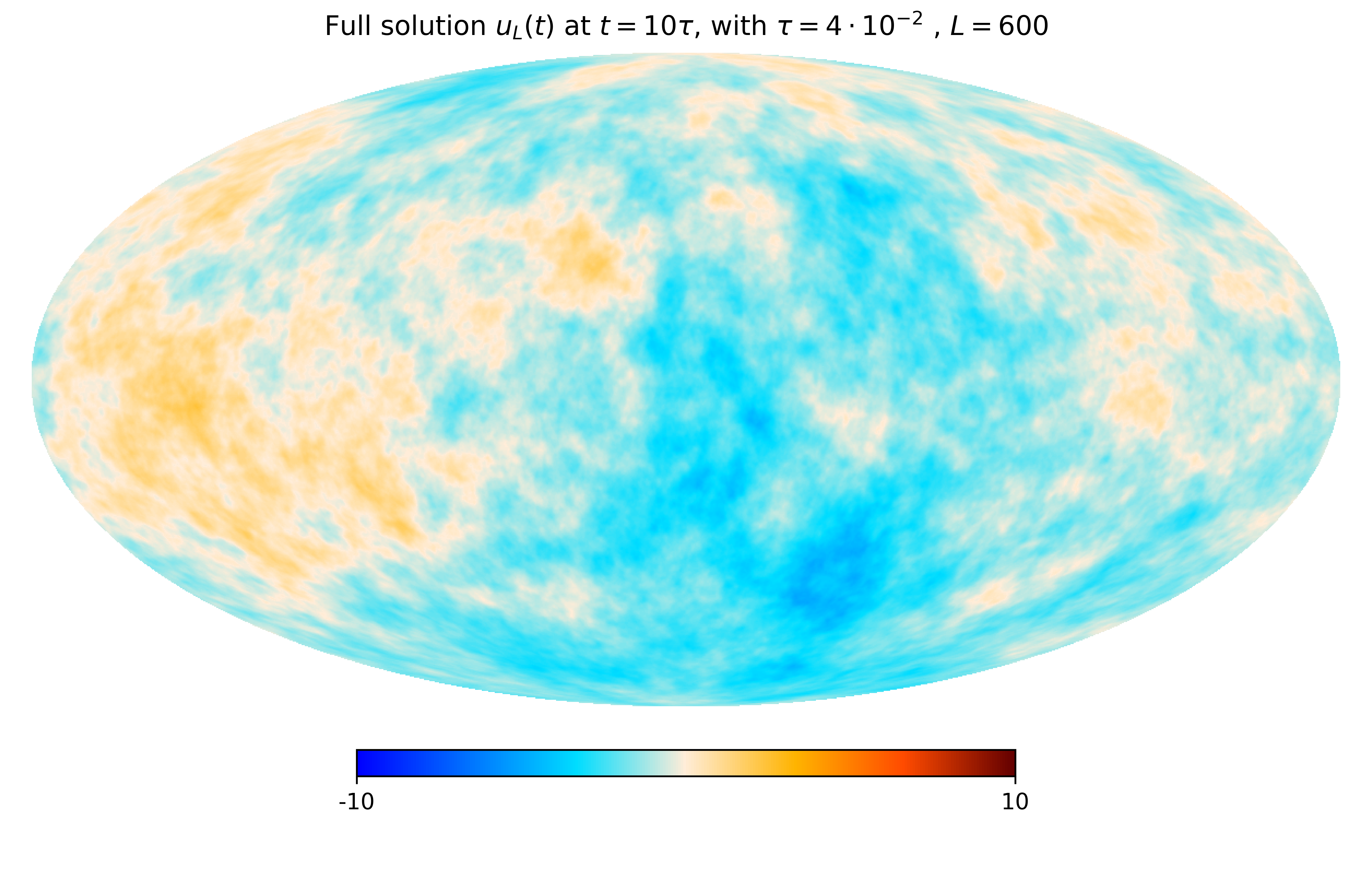} }}%
		\qquad
		\subfloat[\centering $U_{600}(10\tau)$ with $\kappa_1=4.1$ and $\kappa_2=2.5$.\label{fullHOM2wave}]{{\includegraphics[width=0.45\textwidth]{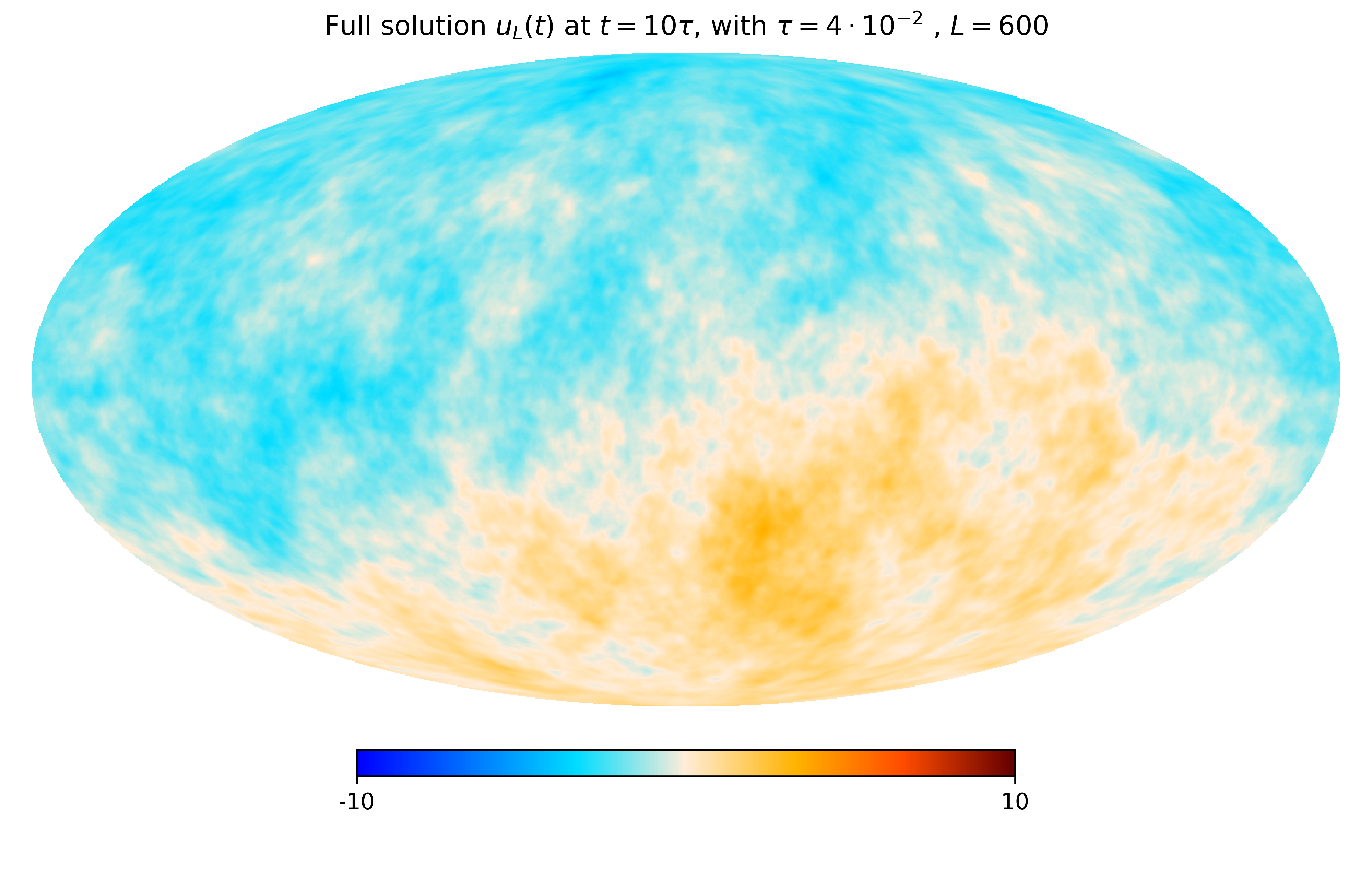} }}%
		\caption{ The truncated combined solution $U_{600}(10\tau)=U^H_{600}(10\tau)+U^I_{600}(10\tau)$ using the truncated homogeneous and inhomogeneous solutions in Figures \ref{fig:Moho cases2} and \ref{fig:U400 at 10tau}.}%
		\label{fig:U600 at 10tau}%
	\end{figure}


\begin{figure}[ht]
		\centering
		\subfloat[\centering A realization of the Inhomogeneous solution with $\kappa_2 = 4.5$. \label{UI1wave}]{{\includegraphics[width=0.45\textwidth]{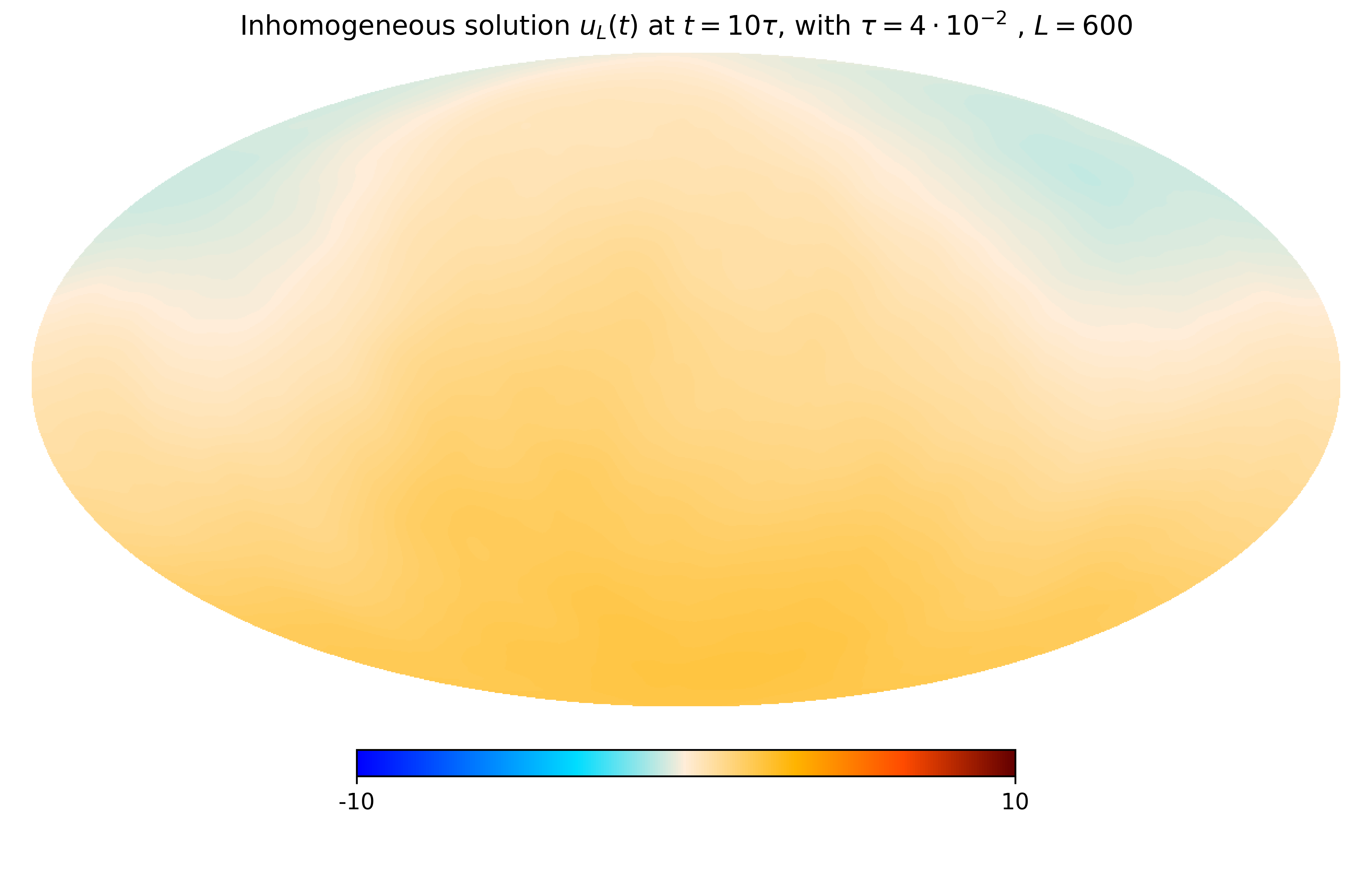} }}%
		\qquad
		\subfloat[\centering $U_{600}(10\tau)$ with $\tau=4\times 10^{-2}$, $\kappa_1=4.1$ and $\kappa_2=2.5$.\label{Fullnew}]{{\includegraphics[width=0.45\textwidth]{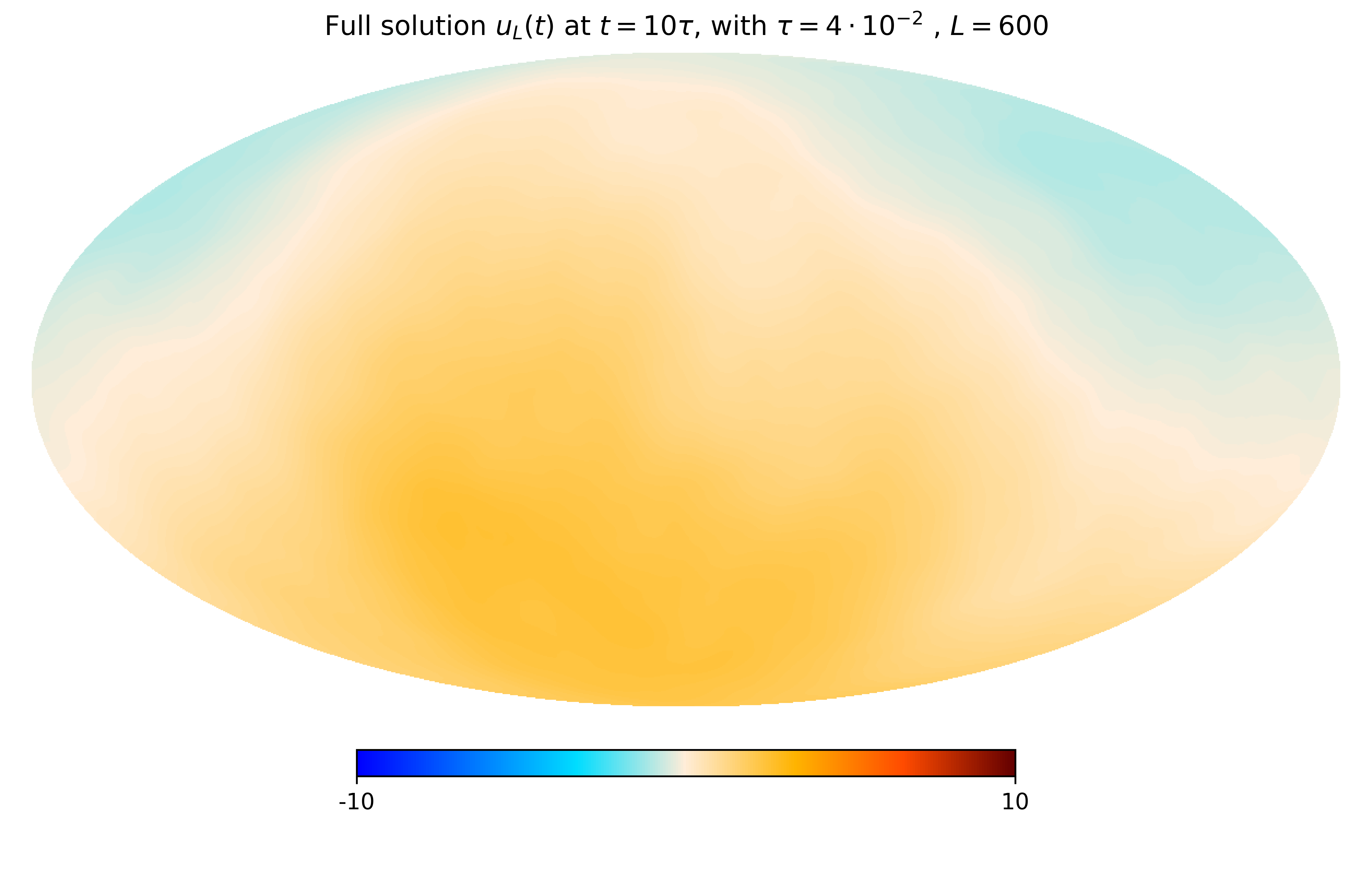} }}%
		\caption{ The truncated inhomogenous solution $U^I_{600}(10\tau)$ and the combined solution $U_{600}(10\tau)=U^H_{600}(10\tau)+U^I_{600}(10\tau)$, using the realization in Figure \ref{initial_2}.}%
		\label{newfig:U600 at 10tau}%
	\end{figure}

\subsection{Simulation of truncation errors and sample H\"{o}lder property}
In this subsection, we present some numerical examples to investigate the convergence rates of the truncation errors $Q_{L}(t)$, given by \eqref{fullq}, of the stochastic solution of \eqref{Sys}. Then, we investigate the sample H\"{o}lder property using a numerical example.

To produce numerical results, we use  $U_{\widetilde{L}}(t)$ with $\widetilde{L}=1500$ as a substitution of the solution $U(t)$ to the equation \eqref{Sys} which is given by \eqref{Exact}. Then, by using Parseval's identity, the (squared) mean $L_2$-errors can be approximated by
	\begin{align}\label{Err-system}
		(Q_{L,\widetilde{L}}(t))^2:&=\bE\Big[\Vert U_{\widetilde{L}}(\omega,t)-U_L(\omega,t)	\Vert_{L_{2}(\bS^2)}^2\Big]\notag\\ 
		&= \bE\Big[\Big\Vert \sum_{\ell=L+1}^{\widetilde{L}}\sum_{m=-\ell}^{\ell}\widehat{U}_{\ell,m}(\omega,t)Y_{\ell,m}\Big\Vert_{L_2(\bS^2)}^2\Big]\notag\\
		&\approx\dfrac{1}{\widehat{N}}\sum_{j=1}^{\widehat{N}}\sum_{\ell=L+1}^{\widetilde{L}}\sum_{m=-\ell}^{\ell}\Big\vert \widehat{U}_{\ell,m}(\omega_j,t)\Big\vert^2,
	\end{align}
where the last line approximates the expectation by the mean of $\widehat{N}$ realizations.

To illustrate the results of Theorem \ref{The5}, we consider two cases, namely $(\kappa_1,\kappa_2)=( 2.3,2.5)$ and $(\kappa_1,\kappa_2)=( 4.1,2.5)$ with $\alpha=0.9$. Then we compute the root mean $L_2$-errors using equation \eqref{Err-system} with $\widehat{N}=100$ realizations and degree up to $L=800$ and time $t=10\tau=0.4$.

Figure \ref{fig:Avg L2 error} shows numerical errors $Q_{L,\widetilde{L}}(t)$ of $\widehat{N}=100$ realizations and the corresponding theoretical errors at time $t=10\tau$ and $\tau=0.04$ corresponding to the choice of $(\kappa_1,\kappa_2)$ and $\alpha=0.9$.
The blue points in Figures \ref{A_Error}-\ref{B_Error} show the
	(sample) mean square approximation errors of the $\widehat{N}=100$ realizations of $Q_{L,\widetilde{L}}(t)$. The red line in each figure shows the corresponding theoretical approximation upper bound.
	The numerical results show that the convergence rates of the mean $L_2$-error of $U_{L}(t)$ are consistent with the corresponding theoretical results in Theorem \ref{The5}.

\begin{figure}[ht]
		\centering
		\subfloat[Average $L_2$ errors, $\kappa_1 = 2.3$, $\kappa_2 = 2.5$. \label{A_Error}]{\includegraphics[width=0.5\textwidth]{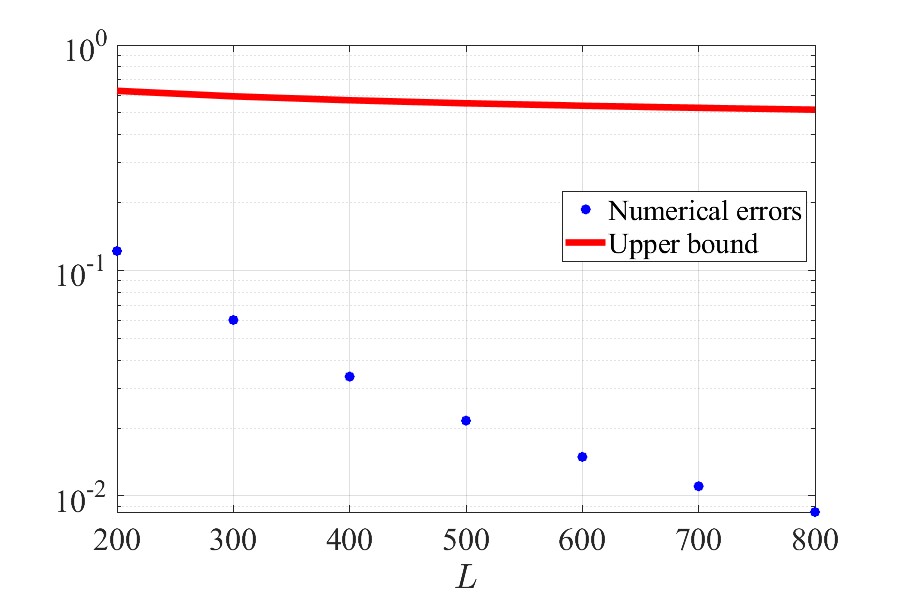}}
		\subfloat[Average $L_2$ errors, $\kappa_1 = 4.1$, $\kappa_2 = 2.5$.\label{B_Error}]{\includegraphics[width=0.5\textwidth]{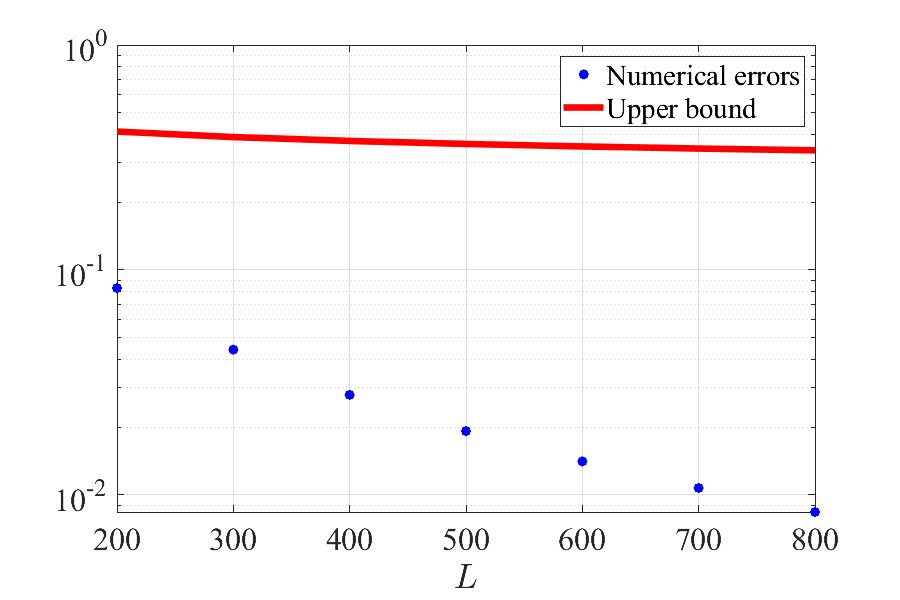}}
		\caption{Numerical errors for the combined solution $U(10\tau)$ with $\tau=0.04$ and $\alpha = 0.9$.}
		\label{fig:Avg L2 error}
	\end{figure}



To illustrate the result of Lemma \ref{holder}, we conducted a series of simulations with the following parameters: \(\alpha = 0.9\), \((\kappa_1, \kappa_2) = (2.3, 2.5)\), and time \( t = 10\tau \) with \(\tau = 0.04\). We fixed a reference point \(\bsx^* \in \bS^2\) close to the North Pole, represented in spherical coordinates,
see \eqref{sph-coord}, as $\bsx^*(\theta,\varphi)$ with $\theta=10^{-6}$, $\varphi = 0$. 
We then selected additional points \( \bsy_k \) such that the geodesic distance between \( \bsy_k \) and \(\bsx^*\) varied from 0 to \(\pi\). Using these points, we simulated $100$ realizations of the truncated random fields \( U_L^H \) and \( U_L^I \), based on their random coefficients. 

For the experiment, we defined a range of geodesic distances 
\( d_k  = 0.01k\) for $k=0,\ldots, \lceil\pi/0.01\rceil\). 
Each distance \( d_k \) corresponds to a point \( \bsy_k \) on the sphere with spherical coordinates \( (\theta_k, \phi_k) = (10^{-6} + d_k, 0)\). These points were converted to Cartesian coordinates and queried in the HEALPix map to obtain the values of the random field \( U_L \) at these locations.

We used the values of \( U_L \) at the reference point \(\bsx^*\) and at the points \( \bsy_k \) with varying distances to compute the variance of the truncated random field:
\[
{\mathcal V}_U(k) := \text{Var}(U_L(\bsx^*, t) - U_L(\bsy_k, t)).
\]
Figure \ref{fig:var} displays the simulated sample variances \( {\mathcal V}_U \) against $\tilde{d}\in(0,\pi)$. In particular, Figure \ref{Cap1} uses \(\beta^* = 0.125\), and Figure \ref{Cap2} corresponds to \(\beta^* = 0.15\). Both figures are consistent with the theoretical results presented in Lemma \ref{holder}.
The Python code for generating the numerical results is available at
\url{https://github.com/qlegia/Stochastic-wave-on-sphere}.



\begin{figure}[ht]
    \centering
    \subfloat[$\kappa_1 = 2.3$, $\kappa_2=2.5$, $\beta^\star = 0.125$ \label{Cap1}]{\includegraphics[width=0.5\textwidth]{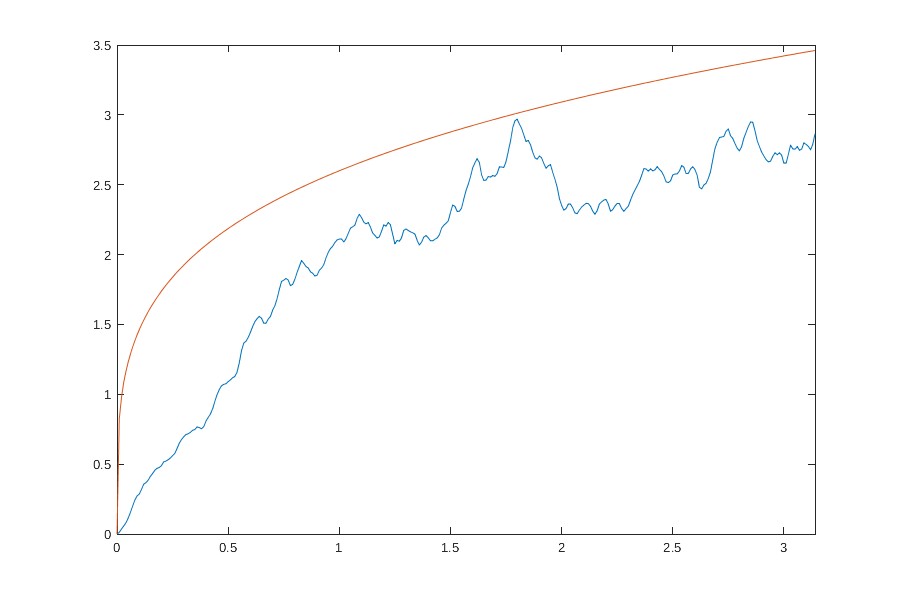}}
    \subfloat[$\kappa_1 = 4.1$, $\kappa_2=4.5$, $\beta^\star = 0.15$ \label{Cap2}]{\includegraphics[width=0.5\textwidth]{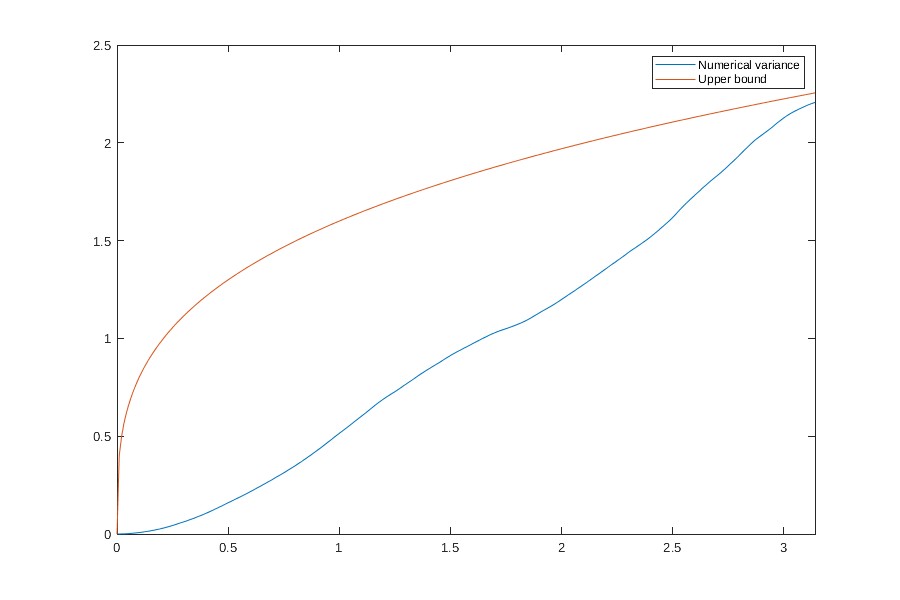}}
    \caption{ ${\mathcal V}_U$ with $\protect\tilde{d}\in[0,\pi]$.}
    \label{fig:var}
\end{figure}

\paragraph{Acknowledgements}
This research was supported under the Australian Research Council's Discovery Project funding scheme (Discovery Project numbers DP180100506 and DP220101811). This research includes extensive computations using the Linux computational cluster Katana~\cite{katana} supported by the Faculty of Science, The University of New South Wales, Sydney. The authors gratefully acknowledge helpful discussions with Prof Ian H Sloan (UNSW, Sydney) while writing the paper.



\end{document}